\documentclass[a4paper,final]{article}
\usepackage{todonotes}

\usepackage{tikz}
\usetikzlibrary{shapes}
\usetikzlibrary{arrows,decorations.markings}
\usepackage{pgf}
\usetikzlibrary{automata}
\usetikzlibrary{chains,decorations.pathmorphing,positioning,fit}
\usetikzlibrary{decorations.shapes,calc,backgrounds}
\usetikzlibrary{decorations.text,matrix}
\usetikzlibrary{arrows,shapes.geometric,shapes.symbols,scopes}
\usetikzlibrary{decorations.markings}
\usepackage{tcolorbox}
\usepackage[hidelinks]{hyperref}

\usepackage{graphicx,tikz}
\usetikzlibrary{arrows, automata,chains,fit,shapes}
\usetikzlibrary{decorations.markings}

\usepackage{microtype} 
\usepackage[english]{babel}

\usepackage{amsfonts}
\usepackage{amssymb}
\usepackage{amsmath}
\usepackage[arrow, matrix, curve]{xy}
\usepackage{enumerate,xspace}
\usepackage{tikz}
\usetikzlibrary{shapes}
\usetikzlibrary{arrows,automata}

\usetikzlibrary{trees}
\usetikzlibrary{decorations}
\usetikzlibrary{snakes}
\usepackage{ellipsis}
\usepackage{ifthen}

\numberwithin{equation}{section}

\usepackage{theorem}
\usepackage{prettyref}
\newcommand{\prref}{\prettyref}
\newcommand\refPedro{Chapter 1\xspace} 

\newrefformat{thm}{Theorem~\ref{#1}}
\newrefformat{lem}{Lemma~\ref{#1}}
\newrefformat{def}{Definition~\ref{#1}}
\newrefformat{cor}{Corollary~\ref{#1}} 
\newrefformat{prop}{Proposition~\ref{#1}}
\newrefformat{sec}{Section~\ref{#1}}
\newrefformat{kap}{Chapter~\ref{#1}}
\newrefformat{bsp}{Example~\ref{#1}}
\newrefformat{rem}{Remark~\ref{#1}}
\newrefformat{fig}{Figure~\ref{#1}}
\newrefformat{eq}{(\ref{#1})}
\newrefformat{ex}{Example~\ref{#1}}
\newrefformat{tab}{Table~\ref{#1}}

\newtheorem{theorem}{Theorem}[section]

\newtheorem{lemma}[theorem]{Lemma}
\newtheorem{proposition}[theorem]{Proposition}
\newtheorem{corollary}[theorem]{Corollary}
\newtheorem{definition}[theorem]{Definition}

\theoremstyle{break}

\theoremstyle{plain}
\theorembodyfont{\upshape}
\newtheorem{example}[theorem]{Example}
\newtheorem{remark}[theorem]{Remark}

\newenvironment{proof}[1][Proof.]{
\begin{trivlist}
\item[\hskip \labelsep {\bfseries #1}]}{\hspace*{\fill}$\Box$\end{trivlist}
}

\newcommand{\ov}[1]{\overline{#1}}

\newcommand\BST{Bass-Serre tree\xspace}
\newcommand\gog{graph of groups\xspace}
\newcommand\tX{\widetilde{X}}

\newcommand\links{\lam}
\newcommand\rechts{\rho}

\newcommand{\abs}[1]{\left|\mathinner{#1}\right|}

\newcommand{\gen}[1]{\left< \mathinner{#1} \right>}

\newcommand{\ggen}[1]{\left<\!\left< \mathinner{#1} \right>\!\right>}

\newcommand{\set}[2]{\left\{\, \mathinner{#1}\vphantom{#2}\: \left|\: \vphantom{#1}\mathinner{#2} \right.\,\right\}}
\newcommand{\oneset}[1]{\left\{\, \mathinner{#1} \,\right\}}
\newcommand{\os}[1]{\left\{\mathinner{#1}\right\}}
\newcommand{\smallset}[1]{\left\{\mathinner{#1}\right\}}
\newcommand{\oi}[1]{{#1}^{-1}}

\newcommand\lds{,\ldots ,} 

\newcommand{\sse}{\subseteq}
\newcommand{\es}{\emptyset}

\newcommand{\VDmatrix}[9]{\begin{pmatrix}{#1}&{#2}&{#3}\\
{#4}&{#5}&{#6}\\{#7}&{#8}&{#9}\end{pmatrix}}

\newcommand{\A}{\mathbb{A}}
\newcommand{\E}{\mathbb{E}}
\newcommand{\F}{\mathbb{F}}
\newcommand{\G}{\mathbb{G}}
\newcommand{\N}{\mathbb{N}}
\newcommand{\Z}{\mathbb{Z}}

\newcommand{\C}{\mathbb{C}}
\newcommand{\B}{\mathbb{B}}

\newcommand{\cG}{\mathcal{G}}
\newcommand{\cA}{\mathcal{A}}

\newcommand{\cM}{\mathcal{M}}

\newcommand\Ip{In particular,\xspace}

\newcommand\ie{i.e.\xspace}
\newcommand\IRR{\mathop\mathrm{IRR}}

\newcommand{\STS}{semi-Thue system\xspace}

\newcommand{\IFF}{if and only if\xspace}
\newcommand{\homo}{homo\-mor\-phism\xspace}
\newcommand{\homos}{homo\-mor\-phisms\xspace}
\newcommand{\iso}{iso\-momor\-phism\xspace}

\newcommand{\cfGR}{context-free group\xspace}
\newcommand{\cfg}{context-free grammar\xspace}
\newcommand{\cfl}{context-free language\xspace}

\newcommand{\pda}{push-down automaton\xspace}

\renewcommand{\phi}{\varphi}

\renewcommand{\wp}{word problem\xspace}
\newcommand{\WP}[1]{\mathop{\mathrm{WP}}({#1})}
\newcommand{\WPP}[2]{\mathop{\mathrm{WP}_{#2}}({#1})}

\newcommand{\eps}{\varepsilon}

\newcommand{\Sig}{\Sigma}
\newcommand{\SSig}{\Sigma^* \times \Sigma^*} 
\newcommand{\Gam}{\Gamma}
\newcommand{\Del}{\Delta}
\newcommand{\GG}{\Gamma}

\newcommand{\alp}{\alpha}
\newcommand{\bet}{\beta}
\newcommand{\gam}{\gamma}
\newcommand{\del}{\delta}
\newcommand{\lam}{\lambda}
\newcommand{\sig}{\sigma}

\newcommand{\bs}{\backslash}
\newcommand{\bags}{\mathop{bs}}

\newcommand{\Rat}{\mathrm{RAT}}

\newcommand{\RAT}{\mathrm{RAT}}

\newcommand\Copt{\cC_{\mathrm{opt}}}
\newcommand\RAS[2]{\overset{#1}{\underset{#2}{\Longrightarrow}}}

\newcommand\ra[1]{\overset{#1}{\longrightarrow}}

\newcommand\DAS[2]{\overset{#1}{\underset{#2}{\Longleftrightarrow}}}
\newcommand\OUTS[5]{#1
\overset{#2}{\underset{#3}{\Longleftarrow}} #4
\overset{#2}{\underset{#3}{\Longrightarrow}} #5}
\newcommand\INS[5]{#1
 \overset{#2}{\underset{#3}{\Longrightarrow}} #4
 \overset{#2}{\underset{#3}{\Longleftarrow}} #5}
\newcommand\RA[1]{\underset{#1}{\Longrightarrow}}

\newcommand\OUT[4]{#1
\underset{#2}{\Longleftarrow} #3
\underset{#2}{\Longrightarrow} #4}

\newcommand\RAA{\Longrightarrow}
\newcommand\LAA{\Longleftarrow}
\newcommand\DAA{\Longleftrightarrow}

\newcommand\RAO[1]{\overset{#1}{\Longrightarrow}}

\newcommand\DAO[1]{\overset{#1}{\Longleftrightarrow}}

\newcommand\cupi{\sqcup}

\newcommand{\genr}[2]{\left< \, \mathinner{#1}\vphantom{#2}\: \left|\: \vphantom{#1}\mathinner{#2} \right.\, \right>}
\newcommand{\Sym}[1]{\mathrm{Sym}({#1})}
\newcommand{\diam}{\mathrm{diam}}
\newcommand{\Aut}{\mathrm{Aut}}

\newcommand{\Comp}[1]{\overline{#1}}

\newcommand{\cC}{\mathcal{C}}

\newcommand{\ssnq}{\subsetneqq}
\newcommand{\sm}{\setminus}

\newcommand{\tto}{\overset{\sim}{\longrightarrow}}

\newcommand{\wt}[1]{\widetilde{#1}}
\newcommand{\wh}[1]{\widehat{#1}}

\newcommand\SL{\mathop\mathrm{SL}}

\newcommand{\dashrightarc}[1]{\overset{#1}\dashrightarrow}

\colorlet{DMnormalbackcolor}{gray!25}
\colorlet{DMsemilightbackcolor}{gray!15}
\colorlet{DMlightbackcolor}{gray!10}
\colorlet{DMmediumbackcolor}{gray!20}
\colorlet{DMdarkbackcolor}{gray!30}
\colorlet{DMmediumforecolor}{black!57}
\colorlet{DMlightforecolor}{black!46}

\newlength{\ksize} 
\newenvironment{aw}{\noindent\color{red}}{}
\newenvironment{fs}{\noindent\color{magenta}}{}
\newenvironment{vd}{\noindent\color{blue}}{}

 \renewcommand{\labelenumi}{(\roman{enumi})}

\begin{document}
\iftrue
\author{Volker Diekert \qquad Armin Wei\ss \\[5mm]
 Universit{\"a}t Stuttgart, FMI \\
 \texttt{$\{$diekert$,$weiss$\}$@fmi.uni-stuttgart.de}}

\title{Context-Free Groups and Bass-Serre Theory}
\maketitle
\tableofcontents
\section*{Introduction}
\addcontentsline{toc}{section}{Introduction}
The \emph{word problem} of a finitely generated group is the 
set of words over the generators which are equal to the identity in the group. The word problem is therefore a formal language. If this language happens to be context-free, then the group is called context-free.
 Finitely generated virtually free groups are context-free. 
In a seminal paper \cite{ms83} Muller and Schupp 
showed the converse:
every context-free group is virtually free. Over the past decades a wide range of other characterizations of context-free  groups have been found. This underlines that context-free  groups play a major role in combinatorial group theory. Among others these characterizations are:
\begin{itemize}
 \item Fundamental groups of finite graphs of finite groups (Karrass, Pietrowski and Solitar \cite{Karrass73}).
 \item Finitely generated groups having a Cayley graph which can be $k$-triangulated (Muller and Schupp \cite{ms83}).
 \item Finitely generated groups having a Cayley graph with finite treewidth (Kuske and Lohrey \cite{KuskeL05}).
 \item Universal groups of finite pregroups (Rimlinger \cite{Rimlinger87a}). 
 \item Groups admitting Stallings sections (Silva, Soler-Escriv\`a and Ventura \cite{SilvaSV16}).
 \item Groups having a finite presentation by some geodesic string rewriting system (Gilman, Hermiller, Holt and Rees \cite{GilHHR07}).
 \item Finitely generated groups having a Cayley graph with decidable monadic second-order theory (Muller and Schupp \cite{ms85}, Kuske and Lohrey \cite{KuskeL05}). 
\end{itemize}
For some other related results see the recent surveys \cite{Antolin11}
or \cite{CoornaertFS12}.
The proof of Muller and Schupp 
in \cite{ms83} relied on Stallings' structure theorem \cite{Stallings71}, and their result was stated first as a conjecture because Muller and Schupp needed 
the assumption that finitely presented groups are accessible.
The accessibility of finitely presented groups was proven later by Dunwoody \cite{Dunwoody85}.

The present notes survey most of the above characterizations. Our aim is to show how the different characterizations of context-free groups are interconnected. Moreover, we  present a self-contained access to the Muller-Schupp theorem without using Stallings' structure theorem or the accessibility result by Dunwoody.
We also give an introduction to some classical results linking groups with formal language theory.

Our notes start with formal language theory and rewriting systems.
 Next, we give an introduction to Bass-Serre theory using rewriting systems. As an application, we 
prove the theorem by Karrass, Pietrowski, and Solitar that the fundamental group of a finite graph of finite groups is virtually free.
In \prref{sec:pregroups} we relate pregroups and geodesic rewriting systems. 
After that we study geometric aspects of virtually free groups via their Cayley graphs. An easy, but fundamental, observation 
yields that Cayley graphs of context-free groups have finite treewidth.
The pictorial representation of finite treewidth is that the 
Cayley graph looks ``very tree-like from far away''. 
Starting with some group having a Cayley graph of finite treewidth, 
Bass-Serre theory shows the direction how to prove that the group is virtually free: We need an action on a tree with finite vertex stabilizers and finitely many orbits.
By the finite treewidth, a natural connection to some tree is visible, but the group does not act on this tree. So, the crucial step is to construct a tree where the group acts on.
This is done by developing the structure tree theory by Dicks and Dunwoody \cite{DicksD89}. Our presentation is based on the recent papers \cite{DunwoodyK15,kroen10}. Together with Bass-Serre theory this yields a proof ``which  explains'' why context-free groups are virtually free. 

Some of the material presented in these notes can also be found in the journal paper \cite{diekertW13}. This applies notably to \prref{sec:cuts}. A roadmap on the structure of these notes can be found at the end in \prref{fig:implikationen}.

These notes appeared in printed form as chapter \cite{DiekertW17crm} in the book \emph{Algorithmic and Geometric Topics Around Free Groups and Automorphisms}. They are based on the 5-hour course \emph{Locally finite graphs of finite tree width and virtually free groups} which was part of the {\em Summer School on Automorphisms of Free Groups} at CRM (Bellaterra, Barcelona, 25th to 29th September 2012).

In the present version me made some changes compared to \cite{DiekertW17crm}: we simplified the definition and construction of graphs of groups and corrected some minor mistakes in \prref{sec:bass_serre}.

\section{Preliminaries}\label{sec:prelims} 

\subsection{Rewriting Systems}\label{sec:rewrite}
Rewriting techniques have been used from the beginning of abstract group theory. 
The importance of rewriting techniques in other areas 
was emphasized in particular by Alonzo Church and John Barkley Rosser, Sr., 
when they wrote their seminal paper on lambda-calculus \cite{ChurchR36}. 
In our notes we are mainly interested in string rewriting systems.

Let $X$ be a set; a \emph{rewriting system} over $X$ is a binary relation
$\RAA\,\subseteq\, X \times X$. If $(x,y)\in\, \RAA$, we write $x \RAA y$.
 The idea of the notation is that $x\RAA y$ indicates that we can rewrite 
$x$ in one step into the element $y$. 
 We use the following notation for certain closure operators for the rewriting system $\RAA $. We denote by $\DAA$ its symmetric closure; 
by $\RAO*$ its reflexive and transitive closure; and by $\DAO*$ its reflexive, transitive, and symmetric closure. 
The relation $\DAO*$ is an equivalence relation. It is the smallest equivalence 
such that $x$ and $y$ are in the same class for all $x \RAA y$.

We also write 
$y \LAA x$ if $x \RAA y$, 
and $x \RAS{\leq k} {} y$
if $y$ can be reached in at most $k$ steps from $x$.
The rewriting system $\RA{}$ is called 
\begin{itemize}
\item \emph{strongly confluent}, if $\OUT y{}xz$ implies
$\exists w: \INS y{\leq 1}{}wz$,
\item \emph{confluent}, if $\OUTS y*{}xz$ implies
$\exists w: \INS y{*}{}wz$,
{}
\item \emph{Church-Rosser}, if $y\DAS * {}z$ implies 
$\exists w: \INS y{*}{}wz$,
{}
\item \emph{locally confluent}, if $\OUT y{}xz$ implies
$\exists w: \INS y{*}{}wz$,
{}
\item \emph{terminating} or \emph{Noetherian}, if there are no infinite chains \begin{align*}
 x_0 \RAS {}{} x_1 \RAS {}{} \cdots x_{i-1} \RAS {}{} x_i \RAS {}{} \cdots,
 \end{align*}
\item \emph{convergent}, if it is locally confluent and terminating.
\end{itemize}

Main properties of rewriting systems are stated in the following classical theorem. 
Proofs can be found in standard textbooks like \cite{bo93springer,jan88eatcs}.

\begin{theorem}\label{thm:konvtoCR}
The following assertions hold. 
\begin{enumerate}
\item Strong confluence implies confluence.
\item Confluence is equivalent to the Church-Rosser property.
\item A locally confluent and terminating system is confluent.
Thus, a convergent system satisfies the Church-Rosser property.
\end{enumerate}
\end{theorem}

\subsection{Presentations of Monoids and Groups}\label{sec:present} 
An \emph{alphabet} is simply a set, where
elements are called \emph{letters}. A \emph{word} $w$ is an element
in some $k$-fold Cartesian product $\Sig^k$ and $k$ is denoted as the length
$\abs w$. Frequently, we write $w = a_1 \cdots a_k$ to denote a word of length $k$.
The union $\Sig^* = \bigcup_{k\geq 0} \Sig^k$ is a monoid by
$$(a_1 \cdots a_k)\cdot (b_1 \cdots b_\ell) = a_1 \cdots a_kb_1 \cdots b_\ell.$$
The neutral element is the \emph{empty word}. It is the unique word of length $0$. 
According to the context the empty word is denoted either as $\eps$ or simply as $1$. 
The monoid $\Sig^*$ is \emph{free} over $\Sig$ because mappings from $\Sig$ to a monoid $M$ are in canonical one-to-one correspondence with \homos from 
$\Sig^*$ to $M$. If $\pi:\Sig^* \to M$ is surjective, then we call $\pi$ a \emph{presentation} of $M$ (by $\Sig^*$).

\newcommand\SGr{T}
Let $\SGr$ be a semigroup and $S \sse \SGr \times \SGr$ be a set of pairs. This defines a rewriting system
$\RAS{}S$ over $\SGr$ by 
$x \RAS{}S y$, if $x = u\ell v$ and $y = ur v$ for some 
$(\ell,r)\in S$. Thus, if a left-hand side $\ell$ of a rule $(\ell,r)\in S$
appears as a factor in a word $x \in \SGr$, then we can replace $\ell$ by the 
right-hand side $r$ and we obtain $y = ur v\in \SGr$. It is common to denote 
a rule $(\ell,r)\in S$ by $\ell \to r$.
Since $\DAS*{S}$ is an equivalence relation, we can form the set of classes. We 
let $\SGr/S = \set{[x]}{x \in \SGr}$, where $[x] = \set{y \in \SGr}{x \DAS*{S} y}$. 
The set $\SGr/S$ of equivalence classes becomes a semigroup $ \SGr/S$ by
$ [x]\cdot [y] = [xy]$.
The multiplication is well-defined because 
for $x \DAS*{S} x'$ and $y \DAS*{S} y'$ it holds: 
$$xy \DAS*{S} x'y \DAS*{S} x'y'.$$
The mapping $x\to [x]$ yields a canonical \homo $\pi:\SGr\to \SGr/S$. 
If $\SGr$ is a monoid (resp.\ group), then $\SGr/S$ is a monoid (resp.\ group), too. 
 
By a slight abuse of language we call $S \sse \SGr \times \SGr$ itself a rewriting system. (Moreover, properties like confluence transfer from $\RAS{}S$ to $S$.)
An element $x \in \SGr $ of the semigroup $\SGr$ is called \emph{irreducible} (w.\,r.\,t.\ $S$), if 
$x$ cannot be written in the form $x = u \ell v$ with $\ell \to r\in S$ and $u,v\in T$. The set of all irreducible elements is denoted by 
$\IRR(S)$. It is the set of elements where no left-hand side of $S$ can be applied to. If $S$ is confluent, then for every $x$ there is at most one element $\wh x \in \SGr$ with $x \RAS*S \wh x \in \IRR(S)$, and, if $S$ is terminating, then for every $x$ there is at least one element $\wh x \in \SGr$ with $x \RAS*S \wh x \in \IRR(S)$. Hence, if the system $S$ is convergent, then for every $x \in \SGr$ there is exactly 
one element $\wh x \in \SGr$ with $x \RAS*S \wh x \in \IRR(S)$. In the latter case, it follows that the canonical homo\-mor\-phism
$\pi:\SGr\to \SGr/S$ induces a bijection between $\IRR(S)$ and $\SGr/S$; and $\IRR(S)$
becomes a set of \emph{normal forms} for the quotient semigroup $\SGr/S$. 

In case we have $S \sse \SSig$, we call the set $S$ a \emph{\STS.} 
Thus, a \STS defines a quotient monoid $M= \Sig^*/S$ and a natural presentation 
$\pi:\Sig^*\to \Sig^*/S$. If we can choose $\Sig$ to be finite, then $M$ is called \emph{finitely generated}, and if, in addition, we can choose $S\sse \SSig$ to be finite, then $M$ is called \emph{finitely presented}.

If $G$ is a group and $R\sse G$ is a subset, then $\ggen R$ means the normal 
closure of $R$ and $G/\!\ggen R$ denotes the quotient group.
If we identify $R$ with the set of pairs $S = \set{(r, 1)\in G \times G} {r \in R}$, then
we have $G/\!\ggen R = G/S$. In the latter notation we view the system $S$ as a set of defining relations $\set{r = 1} {r \in R}$.

\begin{example}\label{ex:freeex}
Let $\Sig$ be a set and $\ov \Sig = \set{\ov a}{a \in \Sig}$ be a disjoint copy of $\Sig$. We extend $a \mapsto \ov a$ to an involution without fixed points on $\wt \Sig = \Sig \cup \ov \Sig$ by $\ov{\ov{a}} = a$. The following system is strongly confluent 
and terminating.
\begin{equation}\label{eq:freeny}
S= \set{a \ov a \to 1}{a \in \wt \Sig}
\end{equation}
The system $S$ defines the free group $F_{\Sig} = \wt \Sig^*/ S$ with basis $\Sig$.
The set $\IRR(S)$ is the set of \emph{freely reduced} normal forms.
\end{example}

\begin{example}\label{ex:freeex_dyck}
If $\Sig$ has $n$ elements, then the system $S$ of \prref{ex:freeex} uses 
$2n$ letters. It is possible, however, to use only $n+1$ letters. We show it for $n =2$. Let 
$F_{\{a,b\}}$ be the free group in two generators, choose a third letter $c$ and consider the so-called \emph{Dyck-system}
\begin{equation}\label{eq:dyck}
S_{\text{D}} = \os{abc \to 1, \, bca \to 1, \, cab \to 1}.
\end{equation}
Again, $S_{\text{D}}$ is strongly confluent 
and terminating, and we have $F_{\{a,b\}} = \os{a,b,c}^*/ S_{\text{D}}$.
Historically, this system $S_{\text{D}}$ was at the beginning of the 
abstract theory on free groups. It is very symmetric and there are no explicit 
``inverse letters''.
The system $S_{\text{D}}$ was used by Dyck in his classical papers \cite{dyck81,dyck83}. 
\end{example}

Our main focus is on finitely generated virtually free groups. 
\begin{definition}\label{def:virtfree}
A group is \emph{virtually free}, if it has a free subgroup of finite index. In our paper, we add the implicit assumption that 
a virtually free group is finitely generated. 
\end{definition}

\begin{example}\label{ex:vfreeex}
Let us show that virtually free groups have presentations by convergent string rewriting systems. 
We start with a virtually free $G$. Let $F_{\Sig}$ be a free subgroup of finite index, and 
$R\sse G$ be a set of representatives of right cosets. This means, we can write $G$ as a disjoint union $G = \overset{.}{\cup}\set{F_{\Sig}r}{r \in R}$. We may assume 
$1 \in R$ and we choose as generating set for $G$ the subset 
$\Del = \Sig \cup \Sig^{-1} \cup R \sm \os{1}$. For every pair $(a,b)\in \Del\times \Del$ 
let $w(a,b) \in (\Sig \cup \Sig^{-1})^*$ denote the freely reduced word and $r \in R$ such that $ab = w(a,b)r$ in $G$. This means $ab \in F_{\Sig}r$ and 
$w(a,b)$ is the freely reduced normal form for $abr^{-1}$. Define a \STS $S$ as follows.
\begin{equation}\label{eq:convVF}
S= \set{ab \to w(a,b)r}{a,b \in \Del, \; w(a,b) \neq a}
\end{equation}
The system $S$ is locally confluent and terminating. Hence, by \prref{thm:konvtoCR}, it is convergent. 
Moreover, we have $\Del^* / S = G$. Irreducible normal forms can be written as a product
$wr$ where $w$ is freely reduced over $(\Sig \cup \Sig^{-1})^*$ and $r\in R$.
\end{example}

\subsection{Preliminaries on Graphs}\label{sec:prelims_graphs}

This section fixes notation and recalls some basic properties of graphs. Readers who are interested mainly in formal language theory can jump directly to \prref{sec:FSfG}.

A \emph{directed graph} $\Gamma = (V,E,s,t)$ is given by the following data: 
A set of vertices $V= V(\Gamma)$, a set of edges $E=E(\GG)$ together with two
mappings $s:E \to V$ and $t:E \to V$. The vertex $s(e)$ is the \emph{source} of $e$ and $t(e)$ is the \emph{target} of $e$. A vertex $u$ and an edge $e$ are \emph{incident}
if $u \in \smallset{s(e), t(e)}$. Two vertices $u$ and $v$ are \emph{adjacent} 
if there is some $e\in E$ such that $\smallset{u,v} = \smallset{s(e), t(e)}$. The \emph{degree} of $u$ is the number of incident edges, and $\Gamma$ is called \emph{locally finite} if the degree of all vertices is finite. A graph is finite, if it has finitely many vertices and edges. A graph $\Gamma'= (V',E',s',t')$ is a \emph{subgraph} of $\Gamma= (V,E,s,t)$,
if $V'\sse V$, $E'\sse E$ and $s'$ and $t'$ are the restrictions of $s$ and $t$. For simplicity we write $\Gamma= (V,E)$ for a graph $\GG$ knowing that there are also the incidence functions.

An \emph{undirected graph} $\Gamma$ is a directed graph such that 
the set of edges $E$ is equipped with an involution $e \mapsto \ov e$ such that $e = \ov{\ov e} $, $e \neq \ov e $ and $s(e) = t(\ov e)$ for all $ e \in E$.
An \emph{undirected edge} is the set
$\smallset{e, \ov e}$. By abuse of language, we denote an undirected edge 
simply by $e$, too. We focus an graphs where the involution is
without fixed points.
In the following a \emph{graph} always means an undirected graph, otherwise 
we say specifically ``directed graph''. If $\Gam=(V,E)$ is an undirected graph, then $E^+$ refers to an orientation of $E$. 
Since $e \neq \ov e$ for all edges $e$, the set of edges $E$ is the disjoint union 
of $E^+$ and $E^-=\set{y\in E}{\ov y \in E^+}$. \Ip $2|E^+|= |E|$.

Frequently we consider graphs without loops (edges $e$ with $s(e)=t(e)$) and multi-edges (edges $e\neq f$ with $s(e)=s(f)$ and $t(e)=t(f)$). Such graphs are called \emph{simple}. For simple graphs we identify undirected edges with their sets of incident vertices. Hence, an edge given by $e$ yields the two element set $\smallset{s(e),t(e)}$. For simplicity, we also write $e=uv$ or $\smallset{u,v} \in E$ if $\smallset{s(e),t(e)} = \smallset{u,v}$ for some $e\in E$.

For $S\sse V(\GG)$ and $v \in V(\GG)$ define as usual in graph theory 
$\GG(S)$ (resp.{} $\GG - S$) to be the subgraph of $\GG$ which is induced by the vertex set $S$ (resp.{} $V(\GG) \sm S$) 
and $\GG - v = \GG - \{v\}$. We also write $\Comp S$ for the complement of $S$, \ie  $\Comp S = V(\GG) \sm S$.
Likewise for 
$e \in E(\GG)$ we let 
$\GG - e = (V(\GG), E(\GG) \sm \{e\})$.

A \emph{path} is a subgraph $(\{v_0,\ldots,v_n\}, \,\{e_1,\ldots,e_n\})$ such that $s(e_i) =v_{i-1}$ and $t(e_i)=v_{i}$ for all $1 \leq i \leq n$. It is \emph{simple} if 
the vertices are pairwise distinct. It is \emph{closed} if $v_0=v_n$.
A \emph{cycle} is a closed path with $n \geq 3$ such that $v_1,\ldots,v_n$ is a simple path. Depending on the situation we also denote paths simply by the sequence of edges or the sequence of vertices (e.\,g., when we consider simple graphs).
The \emph{distance} $d(u,v)$ between $u$ and $v$ is defined as the length (\ie  the number of edges) of a shortest path connecting $u$ and $v$. We let 
 $d(u,v) = \infty$ if there is no such path. 
A path $v_0,\dots,v_n$ is called \emph{geodesic} if $n= d(v_0,v_n)$. An infinite path is \emph{geodesic} if all its finite subpaths are geodesic. For $A,B\sse V(\Gamma)$ the distance is defined as $d(A,B) = \min\set{d(u,v)}{u\in A,v\in B}$. 
An undirected graph $\Gamma$ is called \emph{connected} if $d(u,v) < \infty$ for all vertices $u$ and $v$. 

A \emph{forest} is a simple graph without cycles. A \emph{tree} is a connected simple graph without any cycle (\ie  a connected forest). In particular, a tree is undirected. If $T =(V,E)$ is a tree, we may fix a \emph{root} $r \in V$. This gives an orientation $E^+ \sse E$ by directing all edges ``away from the root''. 
In this way a rooted tree becomes a directed graph $(V, E^+)$ which refers to the 
tree $T= (V,E^+\cup E^-)$, where $E^- = E\sm E^+$.

 A tree $T $ is called \emph{spanning tree} of a graph $\GG$ if $V(T) = V(\GG)$ and $E(T) \sse E(\GG)$, \ie  the tree $T$ connects all vertices of $\GG$.
 
 The following well-known statements will be used later. For proofs see e.\,g.\ \cite[Prop.\ I.9]{serre80} and
\cite[Lem.\ 7.1.3]{diestel06}.

\begin{lemma}\label{lem:spanning tree}
Every connected undirected graph has a spanning tree.
\end{lemma}
\begin{remark}
 In fact, an easy exercise shows that the existence of a spanning tree is equivalent to the axiom of choice.
\end{remark}

 \begin{lemma}[K\"onigs Lemma]\label{lem:konig}
Let $T$ be an infinite locally finite tree. Then there is an infinite geodesic path in $T$.
\end{lemma}

\subsubsection{Graph morphisms and Group Actions}\label{sec:gmga}
 Let $\GG, \GG'$ be graphs. A \emph{morphism of graphs} $\psi$ between $\GG$ and $\GG'$ is a function $\psi: V(\GG) \cup E(\GG) \to V(\GG') \cup E(\GG')$ with $\psi(V(\GG))\sse V(\GG')$ and $\psi(E(\GG))\sse E(\GG')$ which respects the incidences and involution, \ie  which satisfies $s(\psi(e)) = \psi(s(e))$, $t(\psi(e)) = \psi(t(e))$ and $\psi(\ov e) = \ov{\psi(e)}$ for every edge $e\in E(\GG')$. 

A graph morphism is called \emph{locally injective} if for every vertex it is injective on the edges leaving that vertex, \ie  if for all $e_1,e_2\in E(\GG)$, we have that $s(e_1) = s(e_2)$ and $\psi(e_1) = \psi(e_2)$ implies that $e_1 = e_2$. If a morphism of graphs is surjective, it is called an \emph{epimorphism} and if it is bijective it is called an \emph{isomorphism}. 

The graph isomorphisms $\GG \to \GG$ with composition form a group $\Aut(\Gamma)$. An action of a group $G$ on $\GG$ is a homomorphism $\alp:G\to \Aut(\Gamma)$. With other words, an action of a group on a graph consists of actions on the vertex set and on the edge set which respect both incidence functions and the involution. Writing $gx$ instead of $\alp(g)(x)$, we obtain $s(g e) = g  s(e)$, $t(g e) = g t(e)$ and $g \, \ov e = \ov {g e}$ for every edge $e\in E(\GG)$.

If $G$ acts on $\GG$, then we can define
the quotient graph $G \bs \GG$: its vertices (resp.~edges) are the 
\emph{orbits} $G\cdot u$ for $u \in V(\GG)$ (resp.~$G\cdot e$ for $e \in E(\GG)$) with incidences $s(G\cdot e)= G\cdot s(e)$ and $t(G\cdot e)= G\cdot t(e)$.
We say that \emph{$G$ acts with finitely many orbits} if $G \bs \GG$ is a finite graph. 

For $x\in V(\GG') \cup E(\GG')$ the stabilizer of $x$ is denoted by $G_x = \set{g\in G}{gx=x}$.
We call an action $G\to \Aut(\Gamma)$ \emph{free} if all vertex stabilizers are trivial, \ie  if $gx=x$ implies that $g=1$ for $x\in V(\GG') \cup E(\GG')$.

If we want a group to act freely on a graph viewed as topological space, we have to require additionally $ge \neq \ov e$ for all $e \in E(\GG), g \in G$.
We say an action is \emph{without inversion} if $ge \neq \ov e$ for every $g\in G$ and $e\in E(\GG)$.

In the following we will consider actions without inversion. 
However, note that this is not a real restriction since by passing to the barycentric subdivision of a graph we always can construct an action without inversion.
Here, the barycentric subdivision of an undirected graph is obtained by putting an additional vertex on every edge, \ie  for a graph $\GG$ we construct the barycentric subdivision $\GG'$ by $V(\GG') = V(\GG) \cup E(\GG) / \set{e = \ov e}{ e \in E(\GG)}$ and $E(\GG') = E(\GG) \times \smallset{0,1}$ with $s(e,0) = s(e)$, $t(e,0) = e$, $s(e,1) = e$, $t(e,1) = t(e)$ and $\ov{(e,i)} = (\ov e, 1-i)$.

\begin{lemma}\label{lem:locallyinj}
Let $T$ be a tree and $\GG$ a connected graph. Let $\psi:\GG \to T$ a surjective, locally injective graph morphism. Then $\psi$ is an isomorphism.
\end{lemma}
\begin{proof}
 See e.\,g.\ \cite[Lem.\ I.5]{serre80}.
\end{proof}

\section{Formal Language Theory for Groups}\label{sec:FSfG}
Throughout this section 
all alphabets are assumed to be finite and $G$ denotes a finitely generated group. 
The \emph{\wp} of $G$ is defined with respect to a surjective 
\homo $\pi: \Sig^*\to G$. 
We let $$\WPP{G}{\pi} = \set{w \in \Sig^*}{\pi(w) = 1}.$$
When the homomorphism $\pi$ is clear from the context or it is not important we write simply $\WP G$.
The word problem of $G$ is decidable \IFF and $\WP G \sse \Sig^*$ is a decidable 
language. This means there in algorithm which on input $w \in \Sig^*$ decides whether
$w \in \WP G$. If yes, then $\pi(w) =1$ in $G$. Otherwise, if no, then $\pi(w) \neq 1$ in $G$. We are interested only in properties, like decidability of the \wp, which 
do not depend on the chosen presentation $\pi$, but on $G$, only. 

In the following, we discuss some few families $\cC$ of formal languages $L \sse \Sig^*$, where
$\cC$ is closed under 
inverse \homos. This means, if $h:\Sig^* \to \Delta^*$ is a \homo between free monoids and 
$L \in \cC$ with $L \sse \Delta^*$, then $h^{-1}(L) \in \cC$, too. 
All language classes in the so-called Chomsky hierarchy are closed under inverse \homos. We focus on the lower levels in this hierarchy. These are regular and context-free
languages. 

\subsection{Regular Languages}

We start with the notion of regular language, but we do it in broader context and consider subsets of some arbitrary monoid $M$.
Therefore, we distinguish between recognizable and 
rational subsets of $M$. A subset $L \sse M$ is called 
\emph{recognizable}, if there is \homo $h:M\to N$ to a finite 
monoid $N$ such that $h^{-1}(h(L)) = L$. 

The class of recognizable sets is obviously closed under inverse \homos.
It is a Boolean algebra, this means it is closed under finite union and complementation. 
There is an easy description of recognizable subsets in groups: $L$ is a recognizable subset in a group $G$ \IFF there is a subgroup $H$ of finite index
such that $L$ is a union of left cosets $gH$ with $g \in G$. 
In particular, the one-element set $\os{1}$ is recognizable \IFF the group is finite. 

Therefore, recognizable subsets of groups are not a very interesting class.
In the context of groups another definition is more interesting. Let $M$ be again an arbitrary monoid.
The class $\Rat(M)$ of \emph{rational subsets} is inductively defined as follows (see also Section 1.1 of \refPedro). 
\begin{itemize}
\item Finite subsets of $M$ are rational.
\item If $K,L \in \Rat(M)$, then $K\cup L \in \Rat(M)$.
\item If $K,L \in \Rat(M)$, then $K\cdot L = 
\set{xy \in M}{x \in K, y \in L} \in \Rat(M)$. 
\item If $L\in \Rat(M)$, then the generated submonoid $L^*\in \Rat(M)$. 
\end{itemize}

Note that here $L^*$ does not denote the free monoid, but the submonoid generated by $L$. 
It will become clear from the context what we mean.

A classical result in formal language theory is Kleene's Theorem (\prref{thm:kleene}) stating that in finitely generated free monoids the recognizable and rational sets coincide. Moreover, they also can be characterized by finite automata.

A \emph{non-deterministic finite automaton} is a tuple $\cA= (Q,\Sig, \del, I, F)$, 
where $Q$ is a finite set of states, $\Sig$ is a finite alphabet, 
$\del \sse Q \times \Sig \times Q$ is a transition relation, 
$I \sse Q$ is a set 
of initial states, and $F\sse Q$ is a set of final states. 
We can view $\cA$ as an edge labeled finite directed graph. The automaton accepts 
a word $w = a_1\cdots a_n$ if there is path labeled by $w$ starting in some initial state $p \in I$ and ending in a final state $q \in F$. 
In pictures, initial states are indicated by incoming arcs without source
and final states have an inner circle. The following automaton is non-deterministic, because the initial state has two outgoing arcs labeled by $a$. It 
accepts the language $\os{a,b}^*a$ of words, where the last letter is $a$. 
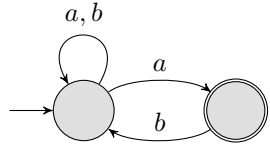
\begin{center}
\begin{tikzpicture}[>=stealth'] 
 \draw (0,0) 
 node[inner sep=0pt,minimum size=8mm,circle,fill=DMnormalbackcolor,draw] (i) {};
 \draw (i) +(180:11mm) node (init) {};
 \draw (2,0) 
 node[inner sep=0pt,minimum size=8mm,circle,double,fill=DMnormalbackcolor,draw] (f) {};
 \draw (f) +(0:11mm) node {};
 \draw[->] (init) -- (i);
 \draw[->] (i) .. controls (0.6,0.45) and (1.4,0.45) .. node[above] {$a$} (f);
 \draw[->] (f) .. controls (1.4,-0.45) and (0.6,-0.45) .. node[above] {$b$} (i);
 
 \draw (i) +(120:14mm) node (s) {};
 \draw (i) +(60:14mm) node (t) {};
 \draw[->] (i) .. controls (t) and (s) .. node[above] {$a,b$} (i);
\end{tikzpicture}
\end{center}
\vspace*{-0.2\baselineskip}

If $\abs I \leq 1$ and 
if for every state $p\in Q$ and letter $a \in \Sig$ there is at most one 
transition $(p,a,q) \in \del$, then the automaton is called \emph{deterministic}. 
The following automaton is deterministic and it accepts again the 
language $\os{a,b}^*a$.

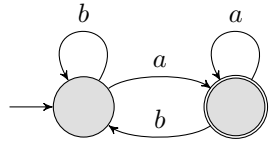
\begin{center}
\begin{tikzpicture}[>=stealth'] 
 \draw (0,0) 
 node[inner sep=0pt,minimum size=8mm,circle,fill=DMnormalbackcolor,draw] (i) {};
 \draw (i) +(180:11mm) node (init) {};
 \draw (2,0) 
 node[inner sep=0pt,minimum size=8mm,circle,double,fill=DMnormalbackcolor,draw] (f) {};
 \draw (f) +(0:11mm) node {};
 \draw[->] (init) -- (i);
 \draw[->] (i) .. controls (0.6,0.45) and (1.4,0.45) .. node[above] {$a$} (f);
 \draw[->] (f) .. controls (1.4,-0.45) and (0.6,-0.45) .. node[above] {$b$} (i);
 
 \draw (i) +(120:14mm) node (s) {};
 \draw (i) +(60:14mm) node (t) {};
 \draw (f) +(120:14mm) node (fs) {};
 \draw (f) +(60:14mm) node (ft) {};

 \draw[->] (i) .. controls (t) and (s) .. node[above] {$b$} (i);
 \draw[->] (f) .. controls (ft) and (fs) .. node[above] {$a$} (f);
\end{tikzpicture}
\end{center}
\vspace*{-0.2\baselineskip}

\begin{theorem}[Kleene]\label{thm:kleene}
Let $\Sig^*$ be a finitely generated free monoid and $L \sse \Sig^*$ be some language. Then the following statements are equivalent:
\begin{enumerate}
 \item $L$ is recognizable.\label{rec}
 \item $L$ is accepted by some deterministic finite automaton.\label{dfa}
 \item $L$ is accepted by some non-deterministic finite automaton.\label{nfa}
 \item $L$ is rational.\label{rat}
\end{enumerate}

\end{theorem}
A language $L \sse \Sig^*$ meeting one of these conditions is also called \emph{regular}. 

\begin{proof}
First, we show that \ref{rec} implies \ref{dfa}. Let $h:\Sig^*\to N$ be a \homo to a finite 
monoid $N$ recognizing $L$, \ie  $h^{-1}(h(L)) = L$. 
We can view $(N,\Sig, \del, \os{1}, h(L))$ as a deterministic finite 
automaton accepting $L$ if we define 
$$\del = \set{(n, a, n\cdot h(a))}{n \in N, \; a \in \Sig}.$$

The step from \ref{dfa} to
\ref{nfa} is immediate.
Hence, let $\cA = (Q,\Sig, \del, I, F)$ be any (non-deterministic) finite automaton accepting $L \sse \Sig^*$. We are going to show that $L$ is recognizable.
Without restriction the states of the automaton 
are numbers in $Q= \os{1 \lds n}$. A letter $a \in \Sig$ changes states, hence, the effect of reading a letter yields a Boolean $n\times n $ matrix $M(a)$, where $M(a)_{i,j}=1$ if and only if $(i,a,j)\in \delta$.
Reading a word $w = a_1\cdots a_n$ amounts to multiplying $M(w) = M(a_1)\cdots M(a_n)$
in the finite monoid $\B^{n\times n}$ of Boolean $n\times n $ matrices.
Now, if $v,w, \in \Sig^*$ with $v \in L$ and $M(w) = M(v)$, then $w \in L$, too. 
Hence, $\B^{n\times n}$ recognizes $L$.

Finally, we are going to show the equivalence between \ref{nfa} and \ref{rat}, \ie  that $L$ is accepted by a finite automaton if and only if $L$ is rational.
Let $\cA = (Q,\Sig, \del, I, F)$ be a finite automaton accepting $L \sse \Sig^*$.
We show that $L$ is rational by a 
\emph{dynamic-programming} paradigm. Again, we may assume that 
$Q= \os{1 \lds n}$. For $i,j \in Q$ and 
$0 \leq k \leq n$ we let $L_{i,j}^k$ to be the language which is accepted when considering $i$ as the only initial state and $j$ as the only final state and restricting the path between $i$ and $j$ to use states $q\in Q$ with $q \leq k$, only. In particular, 
$$L_{i,j}^0 = \begin{cases}
 \smallset{\eps} \cup \set{a \in \Sig}{(i,a,i) \in \del }&\text{for } i=j,\\
 \set{a \in \Sig}{(i,a,j) \in \del }&\text{for } i\neq j.
 \end{cases}
$$
The finite union $L= \bigcup\set{L_{i,j}^n}{i \in I,\, j \in F}$ is the accepted language of 
$\cA$. Hence, in order to show that $L$ is rational, it is enough to see that 
all $L_{i,j}^k$ are rational. This is now straightforward by induction since
$$L_{i,j}^k = L_{i,j}^{k-1} \cup L_{i,k}^{k-1} (L_{k,k}^{k-1})^* L_{k,j}^{k-1}.$$

For the proof that \ref{rat} implies \ref{nfa}, we start 
with a rational expression for $L$. By structural induction on the expression it is an easy exercise to design a non-deterministic finite 
automaton, accepting $L$. 
\end{proof}

\begin{example}
Consider the following non-deterministic finite automaton
accepting the all words over $\os{a,b}$ ending in a positive even number of $a$'s.
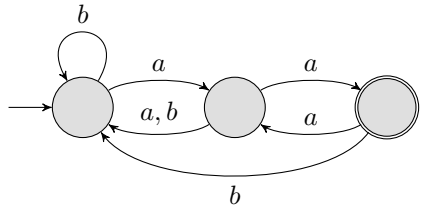
\begin{center}
\begin{tikzpicture}[>=stealth']
 \draw (0,0) 
 node[inner sep=0pt,minimum size=8mm,circle,fill=DMnormalbackcolor,draw] (i) {};
 \draw (i) +(180:11mm) node (init) {};
 \draw (2,0)
 node[inner sep=0pt,minimum size=8mm,circle,fill=DMnormalbackcolor,draw] (q) {};
 \draw (4,0) 
 node[inner sep=0pt,minimum size=8mm,circle,double,fill=DMnormalbackcolor,draw] (f) {};
 \draw (f) +(0:11mm) node {};
 \draw[->] (init) -- (i);
 \draw[->] (i) .. controls (0.6,0.4) and (1.4,0.4) .. node[above] {$a$} (q);
 \draw[->] (q) .. controls (1.4,-0.4) and (0.6,-0.4) .. node[above] {$a,b$} (i);
 \draw[->] (q) .. controls (2.6,0.4) and (3.4,0.4) .. node[above] {$a$} (f);
 \draw[->] (f) .. controls (3.2,-1.1) and (0.8,-1.1) .. node[below] {$b$} (i);
 \draw[->] (f) .. controls (3.4,-0.4) and (2.6,-0.4) .. node[above] {$a$} (q);
 \draw (i) +(120:14mm) node (s) {};
 \draw (i) +(60:14mm) node (t) {};
 \draw[->] (i) .. controls (t) and (s) .. node[above] {$b$} (i);
\end{tikzpicture}
\end{center}
\vspace*{-0.3\baselineskip}

If we number the states from left to right by $1,2,3$, then we obtain the
following Boolean $3 \times 3$ matrices. 
$$ M(a) = \VDmatrix 010 101 100 \quad M(b) = \VDmatrix 100 100 100 $$
A word is accepted by the automaton \IFF the upper right entry in the matrix
$M(w)$ is equal to $1$. The entry means, it is ``true'' that there is a $w$-labeled path from state $1$ to state $3$. 
\end{example}

\begin{corollary}[McKnight]\label{thm:mcknight}
 A monoid $M$ is finitely generated \IFF every recognizable subset is rational.
\end{corollary}

\begin{proof}
Let every recognizable subset be rational. The monoid $M$ is recognized by the trivial homomorphism $M \to \smallset{1}$. Hence, by hypothesis, $M$ is rational. 
But every $L \in \RAT(M)$ is contained in some finitely generated submonoid of $M$. Hence, $M$ is finitely generated.
 
Now, let $M$ be finitely generated. Consider a presentation $\pi: \Sig^* \to M$ where $\Sig$ is finite. 
 Let $L \sse M$ be recognizable. Then $K = \pi^{-1}(L)\sse \Sig^*$ is a regular.
 By \prref{thm:kleene} we find a rational expression for $K$. 
 This gives a rational expression for $\pi(K) \sse M$. (Rational sets are closed under 
 \homos.) We have $L = \pi(K)$, hence the result.
 \end{proof}

\begin{corollary}[Anisimov]\label{thm:anisreg}
 A finitely generated group $G$ has a regular \wp \IFF $G$ is finite.
\end{corollary}

\begin{proof}
 The singleton $\os{1}$ is a recognizable subset of $G$ \IFF and $G$ is finite.
 For a finitely generated group $\os{1}$ is recognizable \IFF $\WP G$ is recognizable
 which is equivalent of being rational or regular by \prref{thm:kleene}.
\end{proof}

Let $M$ be a monoid, then every finitely generated submonoid is rational, but the 
converse is false, in general. For example, consider $\N \times \N$ and
$N = \os {(0,0)} \cup \set{(m,n) \in \N \times \N}{m \geq 1}$. The submonoid $N$
is rational due the expression $N =(0,0) \cup (1,0)\cdot((0,0) \cup {(0,1)})^*$.
However, let $N'$ be generated by finitely many elements $(m_1,n_1) \lds (m_k,n_k)$ of $N$, then 
$(1,\max\os{n_1 \lds n_k} +1)$ belongs to $N\sm N'$. Such a situation is impossible 
for subgroups as Anisimov and Seifert showed. For the proof see \refPedro, Theorem 1.5.

\begin{proposition}[\cite{AnisimovS75}]\label{prop:ratgrup}
Let $G$ be a group and $H$ be a subgroup. Then $H$ is a rational set \IFF 
$H$ is finitely generated. 
\end{proposition}

\subsection{The Chomsky Hierarchy}

In this notes we are mainly interested in groups with a context-free \wp. Context-free languages are the Type-$2$ languages in the Chomsky hierarchy, whereas the regular languages correspond to Type-3. 
Before we focus on context-free languages we want to give an overview over the whole Chomsky hierarchy. In order to do so we develop the concept of grammars in general.

\begin{definition}
 A \emph{grammar} (of Type-$0$) is a tuple $\cG= (V,\Sig, P, S)$ with the following interpretation:%
 \begin{itemize}
 \item $V$ is the finite set of \emph{variables},
 \item $S \in V$ is the \emph{start symbol} or 
\emph{axiom},
\item $\Sig$ is the finite set of \emph{terminal symbols}, 
\item $P \sse ( V \cup \Sig)^* \times (V \cup \Sig)^*$ is a finite semi-Thue system, called the set of \emph{productions}.
 \end{itemize}
As for rewriting systems usual, we write 
$\links \to \rechts \in P$ if $(\links,\rechts) \in P$, and $\alp \links \bet \RAS{}P \alp \rechts \bet$ if $\links \to \rechts \in P$.
A grammar $\cG$ \emph{generates} the language 
$L(\cG) = \set{w \in \Sig^*}{S \RAS{*}P w}.$
\end{definition}

A grammar $\cG$ and the generated language $L(\cG)$ are called to be of 
Type-$i$ for $i = 0, 1, 2, 3$ if the following conditions hold:
\begin{itemize}
\item Type-0 (\emph{recursively enumerable}): There are no restrictions on $\links \to \rechts \in P$.
\item Type-$1$ (\emph{monotone} or \emph{context-sensitive}): All productions $\links \to \rechts\in P$ with one possible exception satisfy $\abs{\links} \leq \abs{\rechts}$. The exception is $S \to \eps$.
If this production is included in $P$, then $S$ must not occur in any right-hand-side of a production. 
\item Type-$2$ (\emph{context-free}): All productions $\links \to \rechts\in P$ satisfy $\abs \links \leq 1$. 
\item Type-3 (\emph{regular}): For for all $\links \to \rechts\in P$ we have $\links \in V$ and
$\rechts \in \Sig^*V \cup \os \eps$.
\end{itemize}
In examples we use some standard notation: $S$ is the axiom, capital letters $A,B,\dots$ denote variables and small letters $a,b,\dots$ denote terminal symbols. $A\to \alp_1\mid \cdots \mid \alp_k$ is 
the so-called Backus-Naur short-hand notation for denoting $k$ productions 
$A\to \alp_1 \lds A\to \alp_k$. 
Note that, although it is not immediately visible from the definition, the Type-$i$ languages form a subclass of the Type-$j$ languages for $i>j$.

Grammars are devices \emph{generating} a language, whereas 
``machines'' \emph{accept} languages. The following table gives an overview over the grammars and machines which generate resp.\ accept the different classes of the Chomsky hierarchy.

\begin{center}
\begin{tabular}{c|c|c}
Grammar Type & Languages & Machines\rule[-1.3ex]{0pt}{0ex} \\ \hline
$0$ & recursively enumerable & Turing machines \rule{0pt}{2.9ex} \\
$1$ & monotone, context-sensitive & linear bounded automata \\
$2$ & context-free &push-down automata \\
$3$ & regular & finite automata 
\end{tabular}
 \end{center}
Finite automata were introduced above. Later we will define
push-down automata. We do not define the other machines types because we only use push-down automata in the sequel.

From the characterization with machines it follows that
all classes of the Chomsky hierarchy are closed under inverse homomorphisms. 
The idea is that, if $h:\Sig^* \to \Delta^*$ is a \homo and $\cM$ a machine  accepting $L \sse \Delta^*$, one can construct a new machine accepting $h^{-1}(L)$ by simulating $\cM$ on input $h(w)$ for $w\in\Sig^*$. From the definition of the respective machine types it follows that this new machine is of the same type.
Therefore, the following definition makes sense and does not depend on the presentation.

\begin{definition}[Type-$i$ Group]\label{def:cfgroup}
A finitely generated group 
is called a \emph{Type-$i$} group if its \wp is a Type-$i$ language. It is called \emph{context-free} if it a Type-$2$ group (\ie  if its \wp is a \cfl).
\end{definition}

Type-3 languages are exactly the regular languages since every Type-3 grammar defines a finite automaton by viewing the variables as states and rules $A\to wB$ as transitions $(A,w,B) \in \delta$. The same way a finite automaton can be transformed into a Type-3 grammar.
This implies that the Type-3 groups are exactly the finite groups.

Examples for Type-$1$ groups are finitely generated linear groups, hyperbolic groups or abelian groups. There are also examples for Type-$1$ groups which are not finitely presented as the free metabelian group or the group in \prref{ex:nixfp}. No algebraic description for Type-$1$ groups is known. 

The Type-0 groups are the finitely generated groups which can be embedded into a finitely presented group by Higman's embedding theorem \cite[Thm. 7.1]{LS01}.

\subsection{Context-Free Languages}

\begin{example}\label{ex:firstcf}
We give some basic examples of context-free languages. None of them is regular. We sketch the proof for that in the first example, 
but similar arguments apply to both of them. 
\begin{itemize}
\item $S \to aSb\mid \eps$ generates the language $L= \set{a^n b^n }{n \geq 0}$. Let 
$h: \os{a,b}^* \to N$ recognize $L$, then we must have $h(a^m) \neq h(a^n)$ for all 
$m\neq n$. Hence, $N$ is infinite, and therefore $L$ is not regular. 
\item $S \to aSbS \mid \eps$ recognizes the set of words $w$ where the number of 
$a$'s is equal to the number $b$'s and where for every prefix of $w$ 
the number of 
$a$'s is not less than the number $b$'s. It is the set of words with correct bracketing, if $a$ is an opening bracket ``$($'' and $b$ is a closing bracket ``$)$''.
\end{itemize}
\end{example}

\begin{example}\label{ex:fgarecf}
Let $F_{\Sig}$ be the free group with basis $\Sig$. 
Then the \wp $\WP {F_{\Sig}}$ is generated by the following \cfg. 
$$S \to aSa^{-1}S\mid \eps \text{\quad for all } a \in \Sig \cup \Sig^{-1}$$
\end{example}

\begin{definition}[Chomsky Normal Form]\label{def:chomskyNF}
A context-free grammar is in \emph{Chomsky normal form}, if all productions are of the form $S\to \eps$, $A \to BC$ with $A,B,C \in V$ or $A\to a$ with $A\in V$ and 
$a \in \Sig$. If the rule $S\to \eps$ exists, then $S$ must not occur on any right-hand side of a production.
\end{definition}

It the following we use the well-known fact that context-free grammars can be transformed into Chomsky normal form. This also shows that the Type-$2$ languages form a subclass of the Type-$1$ languages.

Let $S = \alpha_0 \RAS{}P \alpha_1 \RAS{}P\cdots \RAS{}P \alpha_k = w$ be some derivation of a word $w$. We can view a derivation as a labeled rooted tree. Each node corresponds to a variable or terminal symbol occurring in the derivation. The root is labeled with the axiom $S$. The leaves are labeled with the letters from $w$ or the empty word. If a rule $A\to\rechts$ is applied in the derivation, then the corresponding node has one child for every symbol of $\rechts$.

\begin{example}
We consider the grammar of \prref{ex:firstcf} with the rules $S \to aSbS \mid \eps$. A derivation tree for the word $aabbab$ is as follows:
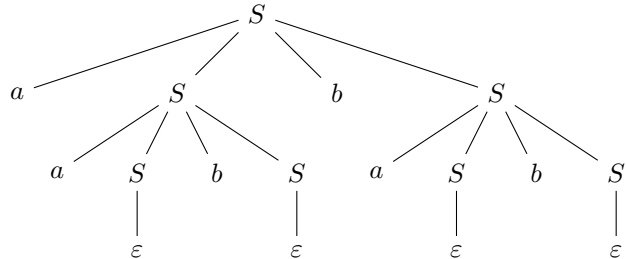
\begin{center}
\begin{tikzpicture}[scale = 0.7]
\node [minimum size=4mm, sibling distance = 5cm](r) at (6,5) {$S$}
 child[sibling distance = 3cm]{ node [minimum size=4mm] {$a$} 
 }
 child[sibling distance = 3cm]{ node [minimum size=4mm] {$S$} 
	child[sibling distance = 1.5cm]{ node [minimum size=4mm] {$a$} 
	}
	child[sibling distance = 1.5cm]{ node [minimum size=4mm] {$S$} 
		child[sibling distance =1cm]{ node [minimum size=4mm] { $\eps$} 
		}
	}
	child[sibling distance = 1.5cm]{ node [minimum size=4mm] {$b$} 
	}
	child[sibling distance = 1.5cm]{ node [minimum size=4mm] {$S$}
	child[sibling distance =1cm]{ node [minimum size=4mm] { $\eps$} 
		}
	} 
 }
 child[sibling distance = 3cm]{ node [minimum size=4mm] {$b$} 
 }
 child[sibling distance = 3cm]{ node [minimum size=4mm] {$S$}
	child[sibling distance = 1.5cm]{ node [minimum size=4mm] {$a$} 
	}
	child[sibling distance = 1.5cm]{ node [minimum size=4mm] {$S$} 
		child[sibling distance =1cm]{ node [minimum size=4mm] { $\eps$} 
		}	 
	}
	child[sibling distance = 1.5cm]{ node [minimum size=4mm] {$b$} 
	}
	child[sibling distance = 1.5cm]{ node [minimum size=4mm] {$S$}
	child[sibling distance =1cm]{ node [minimum size=4mm] { $\eps$} 
		}
	}
};
\end{tikzpicture}
\end{center}
\end{example}

The following criterion is useful to prove that certain languages are {\bf not}
context-free. The statement uses $5$ quantifier alternations. However, this somewhat complicated 
statement has an amazingly simple proof. 

\begin{lemma}[Pumping Lemma]\label{lem:pumpcf}
For every context-free language $L$ there exists some $n \geq 0$ such that for all words $z\in L $ of length at least $n$ there exists 
a factorization $z = uvwxy$ with $\abs {vwx} \leq n$ and $0 <\abs {vx}$ such that 
for all $i \in \N$ we have $uv^i wx^i y \in L$. 
\end{lemma}

\begin{proof}
We may assume that we have a context-free grammar $\cG= (V,\Sig, P, S)$ in Chomsky normal form with $L = L(\cG)$. Let $S = \alpha_0 \RAS{}P \alpha_1 \RAS{}P\cdots \RAS{}P \alpha_k = z$ be some derivation of a word $z$. 
If $z$ is long enough, there must be a variable which occurs at least twice on some path from the root to a leaf in the derivation tree, \ie  there is some $A\in V$ such that $S \RAS{*}P u Ay \RAS{*}P u v Axy\RAS{*}P uvwxy$ with $u,v,w,x,y \in \Sigma^*$. We choose $A$ so that in the derivation $ A \RAS{*}P v Ax\RAS{*}P vwx$
no other variable is repeated on any path and $A$ is repeated only once. (See also \prref{fig:vpopl}.)
 Since $\cG$ is in Chomsky normal form, this implies that $\abs{vwx} \leq 2^{\abs V}$ and $vx \neq \eps$.
\begin{figure}[ht]
\begin{center}
\begin{tikzpicture}[scale=.9]
\node (r) at (0,0) {$S$};
\draw (r) -- (-4,-6);
\draw (r) -- (4,-6);
\draw (-4,-6) -- (4,-6);
\node (A1) at (0.5,-3.5) {$A$};
\draw[decorate, decoration=snake, segment length=3mm] (r) -- (A1);
\node (A2) at (1,-5) {$A$};
\draw[decorate, decoration=snake, segment length=3mm] (A2) -- (A1);
\draw (A2) -- (1.666667,-6);
\draw (A2) -- (0.333333,-6);
\draw (A1) -- (0.5+2.5*2/3,-6);
\draw (A1) -- (0.5-2.5*2/3,-6);
\node (u) at (-2.5,-6.4) {$u$};
\node (u) at (-0.4,-6.4) {$v$};
\node (u) at (1,-6.4) {$w$};
\node (u) at (2,-6.4) {$x$};
\node (u) at (3.1,-6.4) {$y$};
\end{tikzpicture}
\end{center}
\caption{Proof of the Pumping Lemma}\label{fig:vpopl}
\end{figure}
\end{proof}

\begin{example}\label{ex:anbncn}
The language $L = \set{a^nb^nc^n}{n \in \N}$ is not context-free. Indeed, by contradiction, consider a factorization $a^nb^nc^n = uvwxy$ according to \prref{lem:pumpcf}. Then $\abs {vwx} \leq n$ implies that not all three letters $a$, $b$ and $c$ can occur in the factor $vwx$. Say, $c$ does not occur. Now, $0 \leq \abs {vx}$ implies that at least one letter occurs in $vx$. Letting $i = 0$, we should have $uwy \in L$, but in this word the number of $c$'s does not match the number of $a$'s or $b$'s. 

However, the complement $K = \os{a,b,c}^* \sm L$ is context-free: We can write 
$ K = K_1 \cup K_2 \cup K_3$ 
as a union of three context-free languages. For this we let $K_1 = \os{a,b,c}^* \sm a^*b^*c^*$, this language is regular and regular languages are context-free. We let 
$K_2 = \set{a^kb^mc^n}{k \neq m}$ and $K_3 = \set{a^kb^mc^n}{m \neq n}$. 
A possible context-free grammar for $K_2$ is given by the following productions: 
$$ S \to ATC\mid TBC ,\quad T \to aTb \mid \eps,\quad A \to a A \mid a ,\quad B \to bB \mid b,\quad
C \to cC \mid \eps.$$
A similar context-free grammar can be constructed for $K_3$. 

It follows that the class of context-free languages 
is not closed under complementation. 
\end{example}

A classical result of Anisimov \cite{anisimov72} states that 
\cfGR{}s are finitely presented. The proof is not difficult and can be derived from 
\prref{lem:pumpcf}. Actually, the assertion can be made more precise since the finite presentation is given by a \cfg. We need some preparation. 
For a subset $L$ in a free group $F_{\Sig}$ let $\ggen{L}$ denote the generated 
normal subgroup. Thus, $F_{\Sig}/\!\ggen{L}$ is the quotient group where all 
elements of $L$ are equal to $1$. For a grammar $\cG= (V,\Sig, P, S)$
we denote by $F_{V \cup \Sig}/ P$ the finitely presented group given by the 
defining relations $\set{\links = \rechts}{\links \to \rechts \in P}$. 
A variable $A \in V$ is called \emph{reachable}, if there 
exists a derivation $S \RAS *P \alp A \bet$; it is called \emph{productive}, if there 
exists a derivation $A \RAS *P w_A$ for some $w_A \in \Sig^*$. 
A \cfg is called \emph{reduced}, if all variables are reachable and productive. 
Given a \cfg we can ``reduce'' it in polynomial time. Therefore, a \cfg can always be assumed to be reduced. It was observed by Valkema in his diploma thesis 1974 and 
independently later by Hotz that for a reduced \cfg the group $F_{V \cup \Sig}/ P$
does not depend on the grammar, but only on its generated language, see 
\cite{FrougnySV82,Hotz80}. For more general results beyond \cfg{}s see \cite{die86ai,dm89rairo}.

\begin{theorem}[Hotz Isomorphism]\label{thm:hotz}
Let $\cG= (V,\Sig, P, S)$ be a reduced \cfg and $L = L(\cG)$ its generated language
with $1\in L$. 
Then the inclusion of $\Sig^*$ into $F_{V\cup \Sig}$ induces a canonical isomorphism:
$$\phi: F_{\Sig}/\!\ggen{L} \to F_{V \cup \Sig}/ P.$$
In particular, the minimal number of defining relations for $F_{\Sig}/\!\ggen{L}$
is a lower bound on the number of productions $P$. 
\end{theorem}

 \begin{proof}
 Since $1 \in L$, we have $1 \DAS *P w$ for all $w \in L$. Hence, the canonical homomorphism $\phi: F_{\Sig}/\!\ggen{L} \to F_{V \cup \Sig}/ P$ is well-defined. 
 Moreover, for each $A \in V$ there exists a word $w_A \in \Sig^*$ such that $A \RAS* P w_A$ since every variable is productive. 
 Hence, $A= w_A$ in $F_{V \cup \Sig}/ P$, and $\phi$ is surjective.
 
 The words $w_A$ define also a surjective \homo $\psi$ from $(V \cup \Sig)^*$ onto $\Sig^*$ by mapping $A$ to $w_A$ 
 for $A \in V$ and leaving letters $a \in \Sig$ invariant. It is enough to show that in $F_{\Sig}/\!\ggen{L}$ we have 
 $\psi(\links) = \psi(\rechts)$ for all $\links \to \rechts\in P$, because then $\psi \circ\phi$ 
 is the identity on $F_{\Sig}/\!\ggen{L}$; and $\phi$ is injective. Now consider a production
 $\links \to \rechts\in P$. Because every variable is reachable, there exists a derivation 
 $$S \RAS *P \alp \links \bet \RAS *P \psi(\alp) \links \psi(\bet) \RAS {}P \psi(\alp) \rechts \psi(\bet) \RAS {*}P \psi(\alp) \psi(\rechts) \psi(\bet).$$ 
 We have 
 $\links \RAS *P \psi(\links)$, hence $\psi(\alp) \psi(\links) \psi(\bet)\in 
 L(\cG)$ and $ \psi(\alp) \psi(\rechts) \psi(\bet) \in 
 L(\cG)$. This yields 
 $\psi(\alp) \psi(\links) \psi(\bet)= \psi(\alp) \psi(\rechts) \psi(\bet) \in F_{\Sig}/\!\ggen{L}$, and finally $\psi(\links)= \psi(\rechts) \in F_{\Sig}/\!\ggen{L}$.
 \end{proof}

\begin{corollary}[\cite{anisimov72}]\label{cor:anis}
Context-free groups and their finitely generated subgroups are finitely presented. 
\end{corollary}

An analogue of \prref{cor:anis} does not hold for groups with a 
Type-$1$ \wp, in general. See the following example.

\begin{example}\label{ex:nixfp}
Let $F_{\{a,b\}}$ and $F_{\{x,y\}}$ be free groups of rank two and let
$\phi: F_{\{a,b\}} \times F_{\{x,y\}} \to \Z$ be the \homo which maps all generators $a,b,x,y$ to $1$. The kernel $K$ of $\phi$ is a finitely generated subgroup. Possible generators are the three elements $a b^{-1}$, $x b^{-1}$, and $y b^{-1}$.
The group $K$ is a standard example of a finitely generated group which is 
not finitely presented. The \wp of $F_{\{a,b\}} \times F_{\{x,y\}}$ is of Type-$1$. It shows that 
the analogue of \prref{cor:anis} fails for Type-$1$ groups. 
\end{example}

We have seen the definition of context-free languages via grammars. Now, we want to introduce 
 machines which accept context-free languages. These are the so-called \emph{push-down automata}.

A \emph{PDA} (push-down automaton) is given by a tuple 
$\cM= (Q,\Sig,Z, \del, q_0, F)$. Here again $Q$ denotes a finite set of states, $\Sig$ and $Z$ 
are finite alphabets called \emph{input} and \emph{stack alphabet} respectively
, the transition table $\del\sse Z^*Q\Sig^* \times Z^*Q$ is finite,
$q_0\in Q$ is the initial state, and $F\sse Q$ is a set of final states. 
A \emph{configuration} of $\cM$ is a word $\alp = \gam p w$ with $\gam \in Z^*$, $p \in Q$, and 
$w\in \Sig^*$. We should think that $\gam$ is written on a stack and the 
top of the stack is the right end of $\gam$. The word $w$ is written on an input tape and can be accessed from the left, only. The machine is in state $p$. It can read a bounded 
suffix of $\gam$ and a bounded prefix of $w$. Then it performs a transition from the transition
table. More formally, if $\alp = \gam p u = \gam' \links p uv$ with $(\links p u, \rechts q) \in \del$, then it can switch to the configuration $\bet = \gam' \rechts q v$. The idea is that 
the machine performs the following actions: Read$(u)$; pop$(\links)$; push$(\rechts)$; change-state-to$(q)$.  
Viewing $\delta$ as a string rewriting system over the set of configurations, we can define the accepted language of a PDA $\cM$ by 
\begin{align}L(\cM) = \set{w \in \Sig^*}{q_0w \RAS{*}{\del} p , \, p\in F }.\label{eq:accept_condition}\end{align}

That means $\cM$ accepts a word $w$ if it reaches a configuration with empty stack and a final state after reading $w$. Note that for the accepted language for PDAs there are several definitions in literature. The more common definitions either require an empty stack or final state but generally not both. However, all these conditions are equivalent and for our purposes requiring empty stack and final state is more suited.

\begin{theorem}\label{thm:cfpda}
Let $\Sig$ be a finite alphabet and $L \sse \Sig^*$ be a language. Then 
$L$ is context-free \IFF there is a PDA $\cM$ accepting $L$.
\end{theorem}

\begin{proof}

Let $\cG= (V,\Sig, P, S)$ be a context-free grammar. We construct a PDA $\cM= (Q,\Sig,Z, \del, q_0, F)$ as follows. We let $Q = \{q_0, q_1\}$, $Z = V \cup \Sigma $, $F = \{q_1\}$ and $\del = \smallset{ (Sq_0,q_1)} \cup\set{ (\alpha q_0, Aq_0)}{ A\to \alpha \in P} \cup \set{(q_0b, bq_0) }{b \in \Sigma }$. That means the input symbols are shifted to the stack and, if possible, backward derivation steps are performed. It is straightforward to see, that $q_0w \RAS{*}{\del} Sq_0$ if and only if $S \RAS *P w$. Hence, $\cM$ accepts $w$ if and only if $w\in L(\cG)$.

For the other direction, let $\cM= (Q, \Sig, Z, \del, q_0, F)$ be a PDA. 
We introduce an additional stack symbol $\#$ for the stack bottom, which in the first step is written on the stack, and only can be removed when in a final state.
The new automaton is defined as $\cM'= (Q', \Sig,Z', \del' , q_0', F)$ with $Q' = Q\cup\{q_0'\}$, $Z' = Z \cup \{\#\}$, $\del' = \del \cup \set{(q_0', \# q_0), (\# f,f}{f\in F}$.

Without loss of generality we may assume that $\delta' \sse Z' Q \Sigma^* \times Z'^{\leq 2}Q \cup \{(q_0', \# q_0)\}$. 
In order to achieve this, rules of the form $(pu, \rechts q)$ for $p,q\in Q, u \in \Sigma^*, \rechts \in Z^*$ can be replaced by $\set{(zpu, z\rechts q)}{z\in Z'}$.
Furthermore, we introduce additional states for each $(\links p u, \rechts q) \in \del$ where $\abs{\links}> 1 $ or $\abs{\rechts}>2$. A rule $(\links p u, \rechts q) \in \del$ is replaced by a sequence of rules so that first, the stack top $\links$ is read into the state, then the transition is performed by only changing states and finally the new stack symbols are written. 

Now, we are ready to construct a grammar $\cG= (V,\Sig, P, S)$ for the language accepted by $\cM'$. We let $V = \oneset{S} \cup Q\times Z' \times Q$. The meaning of a variable $A = (p,z,q) \in V\setminus \oneset{S}$ will be the language which can be read by $\cM'$ when starting in state $p$, consuming the stack top $z$ and ending in state $q$. In order to do so, we define production rules $P$ as follows:
\begin{align*}
 P &= \set{S\to (q_0, \#, f)}{f\in F}\\
 & \qquad \cup \set{(p,z,q)\to u}{ (zpu, q) \in \del' }\\
 & \qquad \cup \set{(p,z,q)\to u(r,y,q)}{ (zpu, yr) \in \del', q \in Q}\\
 & \qquad \cup \set{(p,z,q)\to u(r,y,s)(s,x,q)}{ (zpu, xyr) \in \del', q,s \in Q }.
\end{align*}
It is straightforward to see that for all $z \in Z$, $p,q, \in Q$ we have
\[\set{w\in \Sigma^*}{zpw \RAS{*}{\del'} q} = \set{w\in \Sigma^*}{(p,z,q) \RAS{*}{P} w}.\]
 In particular, we have $L(\cM) = L(\cM') = \bigcup_{f\in F} \set{w\in \Sigma^*}{\#q_0w \RAS{*}{\del'} f} = \set{w\in \Sigma^*}{S \RAS{*}{P} w} = L(\cG)$.
\end{proof}

If $\cM= (Q,\Sig,Z, \del, q_0, F)$ is a \pda such that for every configuration $\alp$ there is at most one configuration 
$\bet$ with $\alp \RAS{}{\del} \bet$, then $\cM$ is called a \emph{special deterministic push-down automaton}.

\begin{remark} 
 Usually, deterministic push-down automata are defined in the literature to accept with final states, \ie  the language accepted by $\cM$ is $L(\cM) = \set{w \in \Sig^*}{q_0w \RAS{*}{\del} zp , \,z\in Z^*, p\in F }$. 
Languages which are accepted by a deterministic push-down automaton with final states are called \emph{deterministic context-free}. We use another acceptance condition (on empty stack \emph{and} final state) because our condition arises in a natural way when dealing with virtually free groups (see proof of \prref{prop:vfdetcf}). To make the distinction clear, our automata are called special deterministic PDA.
It is easy to see that every language which is accepted by a 
 special deterministic push-down automaton is also accepted by some deterministic push-down automaton (accepting with final states), and hence deterministic context-free. 
 Moreover, the languages accepted by special deterministic push-down automata form a proper subclass of the deterministic context-free languages as the following standard example shows.
\end{remark}

\begin{example}
 Let $L = \set{a^mb^n}{m\geq n}$. It is easy to see that $L$ is accepted by some deterministic PDA with final state. However, assume that $L$ is accepted by the special deterministic PDA $\cM= (Q, \Sig, Z, \del, q_0, F)$. Then, for $m$ large enough, there are by the pigeonhole principle $i<j<n$ and $q\in F$ such that $q_0a^mb^i \RAS{*}{\del} q$ and $q_0a^mb^j \RAS{*}{\del} q$. Now, we have $q_0a^mb^ib^{m-j+1} \RAS{*}{\del} qb^{m-j+1}\RAS{*}{\del} p\in F$ since $i+m-j+1 \leq m$. Hence, $q_0a^mb^jb^{m-j+1} \RAS{*}{\del} qb^{m-j+1}\RAS{*}{\del} p\in F$ and $a^mb^jb^{m-j+1}$ is accepted by $\cM$. However, $a^mb^{m+1}\not\in L$.
 
 Hence, $L$ is not accepted by a special deterministic push-down automaton.
\end{example}

\begin{proposition}\label{prop:vfdetcf}
Let $G$ be a finitely generated virtually free group. Then the \wp $\WP G$ is accepted by some special deterministic push-down automaton. 
\end{proposition}

\begin{proof}
 Let $F_{\Sig}$ be a free group of finite index and $1 \in R \sse G$ a subset which is 
 in one-to-one with $F_{\Sig}\backslash G $ via the canonical mapping $g \mapsto F_{\Sig}g$.
 Let $\Del = \Sig \cup \Sig^{-1} \cup R$. Then $\Del$ is a finite subset of $G$ and the inclusion 
 defines a presentation $\pi: \Del^* \to G$. 
 We now construct a special deterministic push-down automaton with $R \sse Q$ such that $q_0= 1 \in R$ is the initial state. The initial configuration on input $w \in \Del^*$ is $q_0w$. We construct the 
 machine in such a way that it reads its input and stops in some 
 configuration $ur$ where $r \in R$, $u \in (\Sig\cup \Sig^{-1})^*$ is freely reduced and 
 $\pi(w) = \pi(ur)$. That means $w = 1$ in $G$ if and only if the PDA stops in configuration $ur$ where $r =1$ and $u =1$.

The PDA basically performs the reductions of the convergent rewriting system \prref{eq:convVF} in \prref{sec:present}. Let us make this more precise. Assume that the PDA is in some 
configuration $urv$, where $r\in R$, $u$ is the stack contents and freely reduced, 
$v$ is the remaining input, and $\pi(w) = \pi(urv)$. If $v =1$, we are done. Otherwise, write 
$v = a v'$, where $a\in \Del$ is a letter. Now, there are some $s \in R$ and 
 $w(r,a)\in (\Sig \cup \Sig^{-1})^*$ such that $\pi(ra) = \pi(w(r,a)s)$. 
The PDA moves to the configuration $x s v'$, where $x$ is the freely reduced normal form of $uw(r,a)$. This can be done in one step because $x$ differs only in a constant number of the last symbols from $uw(r,a)$. The necessary information can be stored in a finite control.

Formally, the PDA $\cM= (Q,\Del,Z, \del, q_0, F)$ is described as follows: Let $m$ be some constant with $m\geq \max\set{\abs{w(r,a)}}{r\in R, a\in \Delta}$, $R'$ be a disjoint copy of $R$ and $(\Sig\cup \Sig^{-1})'$ a disjoint copy of $\Sig\cup \Sig^{-1}$. When writing equations in the group we consider $r\in R$ and its copy $r'\in R'$ as the same group element and likewise with $a\in \Sig\cup \Sig^{-1}$. We set $Q = R\cup R'$, $Z = \Sig\cup \Sig^{-1} \cup (\Sig\cup \Sig^{-1})'$, $q_0 = 1'\in R'$, $F=\smallset{q_0}$, and for $a\in \Del$, $r'\in R'$, $r\in R$
\begin{align*}
(r'a,w(r'a)s)&\in \delta 	&&\text{for } s\in R \text{ with } r'a=w(r',a)s,w(r',a)\neq 1\\
(r'a,s')&\in \delta 		&&\text{for } s'\in S' \text{ with } r'a= s',\\
(xra,ys) &\in \delta 		&&\text{for } s\in R, x,y\in (\Sig\cup \Sig^{-1})^m \text{ with } ra=w(r,a)s,\\
&				&&\text{and $y$ is the freely reduced normal form of $xw(r,a)$},\\
(xra,ys) &\in \delta 		&&\text{for } s\in R, x,y\in (\Sig\cup \Sig^{-1})'(\Sig\cup \Sig^{-1})^{<m} \text{ with } ra=w(r,a)s\\
&				&&\text{and $y$ is the freely reduced normal form of for $xw(r,a)$},\\
(xra,s') &\in \delta 		&&\text{for } s'\in R', x\in (\Sig\cup \Sig^{-1})'(\Sig\cup \Sig^{-1})^{<m} \text{ with } ra=w(r,a)s'\\
&				&&\text{and $xw(r,a)$ reduces freely to $1$}.
\end{align*}
It is easy to check that this PDA in fact is deterministic and performs the above described computation.
\end{proof}

\begin{corollary}\label{cor:vfarecf}
Finitely generated virtually free groups are context-free.
\end{corollary}

 \begin{proposition}\label{prop:finindcf}
Let $G$ be a group.
If $G$ has a context-free subgroup of finite index, then $G$ is context-free. 
If $G$ is context-free, then every finitely generated subgroup is context-free, too. 
 \end{proposition}
 
\begin{proof}
 If $H$ is context-free subgroup of $G$ of finite index, then there is a PDA accepting
 $\WP H$. A modification of the construction in the proof of \prref{prop:vfdetcf}
 yields a PDA accepting $\WP G$. Details of the construction are left to the reader. 
 Thus, $G$ is context-free. If $G$ is context-free and $H$ is finitely generated subgroup, then we can choose monoid generators $\Sig'$ for $G$ and $\Sig$ for $H$
 such that $\Sig \sse \Sig'$. It follows that $\WP H = \WP G \cap \Sig^*$ is context-free, too. 
\end{proof}

It is well-known that the class of deterministic \cfl{}s is closed under 
complementation, see e.\,g.\ \cite{HU}. Thus, the language $K = \os{a,b,c}^* \sm
\set{a^nb^nc^n}{n \in \N}$ in \prref{ex:anbncn} is a witness for a \cfl which is not deterministic context-free. 

\begin{proposition}\label{prop:cfnixcompl}
The free abelian group $\Z\times \Z$ is not context-free, but the \wp 
$\WP{\Z\times \Z}$ is the complement of a \cfl.
\end{proposition}

\begin{proof}
Consider the presentation $\pi: \os{a,b,c}^* \to \Z\times \Z$ which is given by 
$\pi(a) = (1,0)$, $\pi(b) = (0,1)$, and $\pi(c) = (-1,-1)$. 
By essentially the same arguments as in \prref{ex:anbncn} we obtain that $\WP{\Z\times \Z}$ is not context-free, but its complement is context-free.
\end{proof}

Let $N$ be a subgroup of a finitely generated group $G$. We say that 
$N$ has a \emph{context-free enumeration} \cite{FrougnySS89},
if there is a \cfl $L \sse \Sig^*$ and a presentation $\pi:\Sig^* \to G$ with
 $\pi(L) = N$. (Note that if we replace ``context-free'' by ``regular'', we just
 have an alternative definition of a rational subgroup. In particular, 
 every rational subgroup has a context-free enumeration.)

\begin{theorem}[\cite{FrougnySS89}]\label{thm:cfenum}
Let $G$ be a finitely generated group and $N$ be a normal subgroup. Then $N$ has a context-free enumeration \IFF $N = \ggen{R}$ for some finite set $ R\sse G$. 
\end{theorem}

\begin{proof} First, let $N = \ggen{R}$ for some finite set $ R\sse G$. 
 Since $G$ is finitely generated, there is some finite generating set $\Sig$. 
 Then the following \cfg 
 generates a language which yields of context-free enumeration for the normal subgroup $N$: 
 $$S \to a Sa^{-1} S \mid r \mid \eps \text{\quad for all $a \in \Sig \cup \Sig^{-1}$ and $r \in R$}.$$
 
 For the converse, let $\pi:\Sig^* \to G$ be a presentation with $\pi(L) = N$ and 
 $L \sse \Sig^*$ context-free. By \prref{thm:hotz} the group $F_{\Sig}/\!\ggen L $ 
 is finitely presented and hence $\ggen L = \ggen R$ for some finite set 
 $R$. Hence, $N = \ggen{\pi(R)}$.
\end{proof}

\begin{corollary}\label{cor:cfenumfp}
Let $G$ be a finitely presented group and $N$ be a normal subgroup.
Then $N$ has a context-free enumeration \IFF $G/N$ is finitely presented. 
\end{corollary}

\fi

\section{Bass-Serre Theory}\label{sec:bass_serre}
The following results are from \cite{serre80}. Our proofs differ from the original proofs by using rewriting techniques as presented in \prref{sec:prelims}, otherwise our notation remains (intentionally)
close to \cite{serre80}. 
\subsection{Fundamental group of a graph of groups}\label{sec:fggog}
In the following we only consider connected graphs.

\begin{definition}[Graph of Groups]\label{def:gog}
Let $Y = (V(Y),E(Y))$ be a connected graph. 
A \emph{graph of groups} $\cG$ over $Y$ is given by the following data:
\begin{enumerate}
\item For each vertex $P \in V(Y)$ there is a \emph{vertex group} 
$G_P$.
\item For each edge $y \in E(Y) $ there is an \emph{edge group}
$G_y $ such that $G_y \leq G_{s(y)}$ together with an isomorphism from $G_y$ to $G_{\ov y}$ which is denoted by $a \mapsto a^{\ov y}$  such that $(a^{\ov y})^{y} = a$ for all $a \in G_y$.
\end{enumerate}
The graph of groups is called \emph{finite} if $Y$ is a finite graph. 
\end{definition}
By definition, $G_y = G_{\ov y}$ as abstract groups.
We realize $G_y$ always as a subgroup of $G_{s(y)}$ because it simplifies some formulas 
and it is helpful for (effective) construction.
Thus, for $y \in E(Y) $ with $s(y)= P$ and $t(y)= Q$ there are two 
isomorphisms which are inverse to each other
\begin{align*}
G_y\tto G_{\ov y} \leq G_Q&,\qquad a \mapsto a^{\ov y}, \\
G_{\ov y}\tto G_y \leq G_P&, \qquad a \mapsto a^{y}.
\end{align*}

\begin{remark}\label{rem:simpgra}
A special case is when $G_P=\os 1$ for all vertices in $Y$. This implies $G_y=\os 1$ and the isomorphism $a\mapsto a^{\ov y}$ is the identity of the trivial group for all edges $y$. Thus, a \gog with trivial groups is nothing but an ordinary connected graph. As such we can considers its fundamental group (as a topological space).
Thus, we have the classical fact that for each $P\in V(Y)$ and 
each spanning tree $T$ there is a canonical \iso of the fundamental group $\pi_1(Y,P)$ and $\pi_1(Y,T)$.  This \iso is induced by embedding
$\pi_1(Y,P)$ into the free group $F(Y^+) = Y^*/\set{y\ov y=1}{y\in E(Y)}$ and then projecting $F(Y^+)$ onto $\pi_1(Y,T)= 
F(Y^+)/\set{y}{y\in T}.$
\end{remark}


Let $\cG$ be a \gog
over a connected graph $Y$. The roadmap for the next sections is as follows. We define groups ${\pi_1}(\cG,P)$  and ${\pi_1}(\cG,T)$ and a natural \iso between them. Each of them defines the \emph{fundamental group} ${\pi_1}(\cG)$ of $\cG$. Next, we construct the 
the \BST $\wt X$, which can be seen as universal covering of $Y$. The group ${\pi_1}(\cG)$ acts on $\wt X$ and the quotient
${\pi_1}(\cG) \bs \wt X$ yields back the original graph of groups $\cG$ over $Y$. 

\paragraph{Disjoint unions.}
In the construction we encounter several disjoint unions. 
A disjoint union is denoted either as 
$\cupi_{k\in K}S_k $ or as $\cupi\set{S_k}{k\in K}$. Formally, the disjoint union is the set $\set{(s,k)\in S\times K}{s\in S_k}$ where $S$ is the union over all $S_k$. Frequently, we denote elements
$(s,k)$ in a disjoint union in ``dot notation'' as $s\cdot k$ in order to avoid some brackets. Whenever we use the $s\cdot k$, there will be no risk of confusion.

\paragraph{The group $F(\cG)$.} The group $F(\cG)$ is defined to be the free product of the free group 
$F_{E(Y)}$ with basis $E(Y)$ and the 
groups $G_P$ with $P \in V(Y)$ modulo the set of defining relations
$\set{\ov{y}a y=a^{\ov{y}}}{a \in G_y, \, y \in E(Y)}$ according to \prref{def:gog}. 
We give an alternative definition of $F(\cG)$ in terms of a convergent
\STS. As a set of (monoid) generators we choose the disjoint union
\[\Sigma= \left(\cupi\set{G_P \setminus \smallset{1}}{P\in V(Y)}\right) \cupi E(Y).\]
Thus, the alphabet $\Sig$ is the disjoint union of the vertex groups without the neutral elements together with the disjoint union of the edges of the graph. We fix the alphabet $\Sigma $ throughout this section. 
Recall that $F_\Sig$ is the free group with basis $\Sig$.
Having fixed $\Sig$ we can write present $F(\cG)$ in two ways:
\begin{align*}
F(\cG) 	&= \bigl(\underset{ P\in V(Y)}{\star}G_P\bigr)  \star F_{E(Y)} \, / \set{\ov{y}{a}y=a^{\ov{y}}}{y\in E(Y), \, a\in G_y}\\
&=F_\Sigma/\!\set{ gh=[gh],\, \ov{y}{a}y=a^{\ov{y}}}{P\in V(Y),\, g,h \in G_P;\, y\in E(Y), \, a\in G_y}
\end{align*}
where $[gh]$ denotes the element obtained by multiplying $g $ and $h$ in $G_P$.

In addition, for each edge $y \in E(Y)$ with $s(y) = P$ we choose a set $C_y$ of representatives of the left cosets $G_P/G_{y}$ with $1 \in C_y$.
Thus, each $g \in G_P$ admits a unique factorization $g = ca$ with 
$c \in C_y$ and $a \in G_{y}$.
We now define a rewriting system $S_\cG \sse \Sigma^*\times \Sigma^*$:
\[
\begin{array}{rcll}
gh &\longrightarrow & [gh] & \quad \text{for } P\in V(Y),g, h \in G_P\setminus \os{1}\\ 
\ [ca] y &\ra{} & cy a^{\ov y} & \quad \text{for } y \in E(Y), c \in C_y, a \in G_y \setminus \smallset{1}\\
\ov y y &\ra{}& 1 & \quad \text{for } y \in E(Y)
\end{array}
\]

\begin{proposition}\label{prop:ScG_conv}
We have $F(\cG)=\Sigma^*/S_\cG$ and ${S_{\cG}}$ is a convergent \STS.
\end{proposition}

\begin{proof}
The \homo from $\Sigma^*/{S_{\cG}}$ to $F(\cG)$ is induced by the inclusion 
of $\Sig$ into the free group $F_\Sigma$. It is immediate that this is an \iso. 
Local confluence of $S_\cG$ follows by a direct inspection. 
We prove termination as follows. Consider any 
derivation sequence 
$$w_0 \RAS{}{S_\cG} w_1 \RAS{}{S_\cG} w_2\RAS{}{S_\cG} w_3 \cdots $$
 We have to show that it is finite. After finitely many steps 
 the number of letters $y \in E(Y)$ in each $w_i$ remains
stable. If $w_i$ does not contain any $y \in E(Y)$, then only length reducing rules can be 
applied. Thus, the sequence is finite. We may assume that every word $w_i$ has a
prefix $u_iy$ where no $y \in E(Y)$ occurs in $u_i$. 
Hence, after finitely many steps the prefix $u_iy$ remains stable and does not change anymore. Termination follows now by induction on the number of 
$y$ occurring in $w_0$. 
\end{proof}

For $P,Q\in V(Y) $ we denote the set of paths from $P$ to $Q$ in $Y$ with ${\Pi}(P,Q)$; more precisely
 $${\Pi}(P,Q) = \set{y_1\cdots y_k}{s(y_1) = P,\, t(y_k) = Q,\, t(y_{i}) = s(y_{i+1}) \text{ for $1\leq i < k$}}.$$
Furthermore, we define subsets $\pi(\cG, P,Q)$ of the group $F(\cG)$ by 
the elements 
$g_0y_1\cdots g_{k-1}y_kg_{k}\in F(\cG)$, which satisfy 
\begin{equation}\label{eq:pipq}
y_1\cdots y_k\in\Pi(P,Q) ,\, g_i\in G_{s(y_{i+1})} \text{ for $0\leq i < k$},\, g_{k} \in G_{Q}.
\end{equation}

Note that $\pi(\cG, P,Q)\cdot \pi(\cG, Q,R) = \pi(\cG, P,R)$ and 
$\pi(\cG, P,P)$ is a group for all vertices $P\in V(Y)$ whereas $\pi(\cG, P,R)$ is not a group for $P\neq R$.

\begin{definition}Let $P\in V(Y)$. 
The \emph{fundamental group} of $\cG$ with respect to the base point $P$ is defined as $\pi_1(\cG,P)=\pi(\cG,P,P)$. 
\end{definition}

Recall that we have made the hypothesis that $Y$ is connected. 
Thus, there exists a 
spanning tree $T= (V(Y),E(T))$ of $Y$.
The following classical result will be used and can be viewed as the  germ of the Bass-Serre theory. 
\begin{proposition}\label{prop:germ}
Suppose that $G_P=\os 1$ for all $P\in V(Y)$. Then~$F=\pi_1(\cG,P)$ is a free group. A generating set for $F$ is given by the $E(Y)\sm E(T)$.
\Ip if $V(Y)$ is finite, then the rank $r(F)$ is given
by $r(F)-1= |E^+(Y)| - |V(Y)|$ where  $E^+(Y)$ is any orientation of 
$E(Y)$.  
\end{proposition}
\begin{proof}
If $G_P=\os 1$ for all $P\in V(Y)$, then we are in the classical situation of connected graphs as explained in \prref{rem:simpgra}. Thus
$F= \pi_1(Y,P)=\pi_1(Y,T)$ is free group; and the rank formula is well-known.
\end{proof}

\begin{definition}
Let $T$ be a spanning tree of $Y$. The \emph{fundamental group} of $\cG $ with respect to $T$ is defined by: 
\[\pi_1(\cG,T)= F(\cG)/\set{y=1}{y\in T}.\]
\end{definition}

\begin{proposition}\label{prop:twofunds}
The canonical homomorphism $\eta$ from the subgroup $\pi_1(\cG, P)$ 
of $F(\cG)$ to the quotient group $\pi_1(\cG,T)$ is an isomorphism.
\end{proposition}

\begin{proof}
 For two vertices $Q,R \in V(Y)$ we write $T[Q,R]=y_1 \cdots y_k$ for the sequence of edges of the unique shortest path from $Q$ to $R$ in the spanning tree $T$. We can read the word $T[Q,R]\in E(Y)^*$ as an element in the group $F(\cG)$. 
 It yields a homomorphism $\tau: F(\cG)\to {\pi_1}(\cG,P) $ by 
 \begin{align*}
 	&&\tau(y) &= T[P,s(y)] \,y\, T[t(y),P] &&\text{ for } y \in E(Y) \text{ and}\\
 	&&\tau(g) &= T[P,Q] \,g \,T[Q,P] &&\text{ for } Q\in V(Y)  \text{ and } g\in G_Q.
 \end{align*}
The \homo $\tau$ is well-defined because $\tau(\ov y a y) = \tau(a^{\ov y})$. 
Moreover, $\tau$ is surjective and $\tau(y) = 1$ for all edges $y \in E(T)$. 
Hence, $\tau$ induces a surjective \homo (also denoted by $\tau$) of $\pi_1(\cG,T)$ onto ${\pi_1}(\cG,P)$. 
The composition $\pi_1(\cG,T)\overset{\tau}{\to} {\pi_1}(\cG,P) \overset{\eta}{\to}\pi_1(\cG,T) $ is the identity.\footnote{The notation $\eta$ and $\tau$ reflects that 
$\eta$ is defined through the ``embedding" of ${\pi_1}(\cG,P)$ into $F(\cG)$ and that 
$\tau$ depends on the chosen (spanning) tree $T$.}
This shows the result. 
\end{proof}

\prref{prop:twofunds} shows in particular that the two definitions of a fundamental group are independent of the choice of the base point or the spanning tree.
As an abstract group we denote it by ${\pi_1}(\cG)$. It is the 
fundamental group of the graph of groups $\cG$ over the connected graph $Y$.

\begin{corollary}\label{cor:allsame}
The canonical mapping from the subset $\pi(\cG, P,Q)$ 
of $F(\cG)$ to the quotient group $\pi_1(\cG,T)$ is a bijection.
\end{corollary}

\begin{proof}

Choose an edge sequence $T[Q,P]$ in the spanning tree from $Q$ to $P$. 
Then the mapping $g \mapsto g T[Q,P]$ yields a bijection between $\pi(\cG, P,Q)$ 
and $\pi_1(\cG, P)$. The image of $g$ and of $g T[Q,P]$ in $\pi_1(\cG,T)$
is the same. The result follows by \prref{prop:twofunds}.
\end{proof}

\begin{corollary}\label{cor:embedP}
Let $G_P$ be a vertex group. 
The canonical homomorphism of $G_P$ to $\pi_1(\cG,T)$ is injective. 
\end{corollary}

\begin{proof}
 By \prref{prop:twofunds} we may identify $\pi_1(\cG,T)$ with $\pi_1(\cG, P)$.
 We have $\pi_1(\cG, P) \sse F(\cG)$, but all elements of $G_P$ are irreducible 
 with respect to the convergent 
\STS $S_\cG$ of \prref{prop:ScG_conv}. Thus, $G_P \sse \pi_1(\cG, P) \sse F(\cG)$.
\end{proof}

\begin{example}\label{ex:free_group_gog}
Let $Y = (V,E)$ be a finite nonempty connected graph with $n$ vertices and $m$ edges and
$G_P= \os{1} $ for all vertices $P \in V$. Then $\pi_1(\cG)$ is the 
usual fundamental group of an undirected graph. It is a free group of rank 
$m-n+1$. In its simplest form $Y$ is the following graph:
 
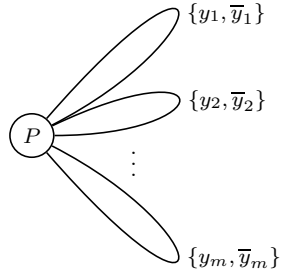
\begin{center}
 \begin{footnotesize}
 \begin{tikzpicture}
 \node [draw,semithick,circle ](b) at (0,0 ) {$P$};
 \draw [semithick,-] (b) .. controls(2.5,2.9) and (2.5,1.3) .. (b)node[ pos=.5,right]{$\smallset{y_1,\ov y_1}$};
 \draw [semithick,-] (b) .. controls(2.5,1.2) and (2.5,0.0) .. (b)node[ pos=.5,right]{$\smallset{y_2,\ov y_2}$};
 \path [semithick,-] (b) .. controls(1.5,0.7) and (1.5,-1.4) .. (b)node[ pos=.5,right]{$\vdots$};
 \draw [semithick,-] (b) .. controls(2.5,-1.3) and (2.5,-2.9) .. (b)node[ pos=.5,right]{$\smallset{y_m,\ov y_m}$};
 \end{tikzpicture}
 \end{footnotesize}
 \end{center}
 \vspace*{-1cm}
 Then ${\pi_1}(\cG)$ is the free group with basis $\smallset{y_1,\dots, y_m}$.
 \end{example}

\begin{example}\label{ex:amal}
Let $G_P$ and $G_Q$ be two groups with a common subgroup $G_y = G_P \cap G_Q$.
Then the amalgamated product $G_P \star_{G_y}G_Q$ of $G_P$ and $G_Q$ over $G_y$ is the 
fundamental group of the graph of groups over $Y$, where 
 $Y$ is as follows:
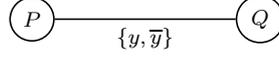
\begin{center}
\begin{footnotesize}
\begin{tikzpicture}
\node[draw,semithick, circle ] (a) at (-1.5, 0) {$P$};
\node [draw,semithick,circle ](b) at (1.5,0 ) {$Q$};
\draw [semithick,-] (a) -- (b)node[ pos=.5,below]{$\smallset{y,\ov y}$};
\end{tikzpicture}
\end{footnotesize}
\end{center}
\end{example}

\begin{example}\label{ex:hnn}
 Let $\cG$ be a graph of groups over the following graph:
 
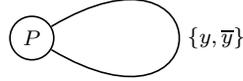
\vspace*{-0.7cm}
 \begin{center}
 \begin{footnotesize}
 \begin{tikzpicture}
 \node [draw,semithick,circle ](b) at (0,0 ) {$P$};
 \draw [semithick,-] (b) .. controls(2.5,1.5) and (2.5,-1.5) .. (b)node[ pos=.5,right]{$\smallset{y,\ov y}$};
 \end{tikzpicture}
 \end{footnotesize}
 \end{center}
 \vspace*{-1cm}
Then the fundamental group ${\pi_1}(\cG)$ is the 
HNN-extension with stable letter $y$ which is defined as 
$${\pi_1}(\cG) = G_P \star F_{\{y\}}/ \set{\ov{y}{a}y=a^{\ov{y}}}{a \in G_y}.$$
 \end{example}
 
 \subsection{Britton Reductions over Graphs of Groups}\label{sec:britton}
For some purposes we do not need unique normal forms, hence the rewriting system $S_\cG$ is too precise and too complicated. The following observation gives rise to another rewriting system:
 If $w = g_0y_1\cdots g_{k-1}y_kg_{k}$ with $y_1 \cdots y_k\in \Pi(P,Q)$, $g_i\in G_{s(y_{i+1})}$, and $y_{i}g_i y_{i+1} = \ov y {a} y$ for some $i$ with $a\in G_y$ (note that this implies $\ov y {a}y \RAS*{S_\cG} a^{\ov y}$), then also $y_1 \cdots y_{i-1}y_{i+2} \cdots y_k\in \Pi(P,Q)$ and
$$g_0y_1\cdots g_{i-2} y_{i-1}[g_{i-1}a^{\ov y}g_{i+1}] y_{i+2}g_{i+2}\cdots y_kg_{k} \in \pi(\cG, P,Q).$$

Motivated by \prref{ex:hnn}, this leads to the notion of \emph{Britton reduction}. Britton reductions are given by the rewriting system $B_\cG\sse \Sigma^*\times \Sigma^*$ with the following rules:
\begin{align*}
&&gh &\ra{} [gh] & \quad \quad \text{for } P\in V(Y),\ g, h \in G_P \sm \smallset{1}\\
&&\ov{y}{a}y&\ra{} a^{\ov{y}} & \quad \quad \text{for } y\in E(Y),\ a\in G_y
\end{align*}
As $B_\cG$ is length reducing, it is terminating. 
A word is called \emph{Britton-reduced}\footnote{Originally, the notation \emph{Britton reduction} was used only for HNN extensions, like in \cite{LS01}.} if no rule of $B_\cG$ can be applied. 
Furthermore, the reflexive, symmetric and transitive closures $\DAS{*}{B_\cG}$ and $\DAS{*}{S_\cG}$ are the same. In particular, $F(\cG) = \Sigma^* / B_\cG$.

\begin{remark}\label{rem:britton}
The following facts are crucial: 
\begin{enumerate}
\item If $g_0x_1 \cdots g_{k-1}x_k g_{k} \RAS*{S_{\cG}} h_0y_1 \cdots h_{\ell-1}y_\ell h_{\ell}$, with $x_1 \cdots x_k\in \Pi(P,Q)$, then $y_1 \cdots y_\ell\in \Pi(P,Q)$, too.
\item If $ g \in \pi(\cG,P,Q)$ with $g= g_0y_1 \cdots g_{k-1}y_k g_{k} \in \IRR(B_\cG)$, then the path $y_1 \ldots y_k\in \Pi(P,Q)$
is uniquely defined by $g$. If we have $g \RAS*{S_{\cG}} \hat g\in \IRR({S_{\cG}})$,
then the ``$y$-sequence'' $y_1 \ldots y_k$ does not change, and we can write $\hat g = c_0y_1 \cdots c_{k-1}y_k \tilde g_{k}$ as a word in $\Sig^*$ where $c_i \in C_{y_{i+1}}$ and $\tilde g_{k}\in G_Q$. 
Moreover, the prefix $c_0y_1 \cdots c_{k-1}y_k$ depends on $gG_Q$, only. 
\end{enumerate}
\end{remark}

\prref{rem:britton} gives us an easy algorithm for the decidability of the word problem of $\pi_1(\cG,P)$ as we will see next.  
\begin{corollary}\label{cor:WPgog}
 Let $\cG$ be a graph of groups with underlying graph $Y$. For all $y\in E(Y)$ let the membership problem of $G_y$ in $G_{s(y)}$ be decidable and let the word problem for $G_P$ be decidable for some (or equivalently every) $P\in V(Y)$. Furthermore, let the isomorphisms $G_y \to G_{\ov y}$ be effectively computable for all $y\in E(Y)$. Then the word problem of $\pi_1(\cG,P)$ is decidable.
\end{corollary}
Here, the membership problem of a subgroup $H$ in a group $G$ is to decide whether a given $g \in G$ is contained in $H$. 
\begin{proof}
 Apply Britton reductions until it is not possible anymore. Britton reductions can be effectively computed since the membership problem for $G_y$ in $G_{s(y)}$ is decidable and the isomorphisms $G_y \to G_{\ov y}$ are computable. If at the end there is still some $y$ in the resulting word, the input word is not equal to $1$ in $\pi_1(\cG,P)$. Otherwise, the algorithm for the word problem in $G_P$ is applied.
\end{proof}

Note that 
\prref{cor:WPgog} does not state anything about the complexity. The problem is that even if all computations can be performed efficiently, the blow up due to the calculations of the isomorphisms $G_y \to G_{\ov y}$ might prevent an efficient solution of the word problem in $\pi_1(\cG,P)$. An example is the Baumslag group $\mathrm{BG}(1,2) = \genr{a,t,b}{tat^{-1}= a^2, bab^{-1}=t}$ which is an HNN-extension of the Baumslag-Solitar group $\mathrm{BS}(1,2)$. For this group, the straightforward algorithm as in \prref{cor:WPgog} leads to a non-elementary running time (\ie there is no fixed tower of exponentials as bound for the running time). However, in \cite{muw11bg} it is shown that the word problem still can be solved in polynomial time by using super exponential compression of the intermediate results.

\subsection{The Bass-Serre Tree}\label{sec:bst}
Let $\cG$ be a graph of groups over $Y$.
Our construction of the \BST for $\cG$ relies on the convergent 
\STS $S_\cG$ of \prref{prop:ScG_conv}.

We define the \emph{\BST} $\widetilde X$ as a subset of 
$\IRR({S_{\cG}})$. Since $\IRR({S_{\cG}})$ is a prefix-closed subset of $\Sigma^*$, there is a natural tree structure given by the prefix relation on words. 
First, we fix a vertex $P_0 \in V(Y)$ as a base point. 
The nodes in $\widetilde X$ are the words $v = c_0y_1 \cdots c_{k-1}y_k \in \IRR({S_{\cG}})$ 
such that $v \in \pi(\cG, P_0,P)$ for some
$P\in V(Y)$, \ie  $y_1 \ldots y_k \in \Pi(P_0,t(y_k))$ and $c_{i-1}\in C_{y_{i}} \sse  G_{s(y_{i})}$, where $C_{y_{i}}$ is the system of coset representatives for $G_{s(y_{i})}/G_{y_{i}}$ with $1\in C_{y_{i}}$. The root of $\widetilde X$ is the empty word $1$. 
If $v = c_0y_1 \cdots c_{k-1}y_k$ is a node, then the children
of $v$ are the words $vcy$ where $vcy \in \IRR({S_{\cG}})$. 
We label such an edge from $v$ to its child by $cy$. This means
either $k = 0$ or $c \neq 1$ or $k >0$ and $\ov y \neq y_k$. 
If $v = c_0y_1 \cdots c_{k-1}y_k$ is not the root (\ie  $k >0$), then 
$c_0y_1 \cdots c_{k-2}y_{k-1}$ is the parent node of $v$. 
We label this edge from $v$ to its parent by $\ov{y_k}$. 
The node $P = t(y_k)$ is uniquely defined by $v$. Moreover, for each 
edge $y\in {E(Y)}$ and each $c \in C_y$ there is a unique 
edge leaving $v$ with label $cy$.

The node set $V(\wt X)$ is in a canonical bijection with the
disjoint union
$\cupi \set{\pi({\cG},P_0,P)/ G_P}{P \in V(Y)}$. The edge set of $\wt X$ corresponds to the
disjoint union $ \cupi \set{\pi({\cG},P_0,s(y))/ G_y}{y \in {E(Y)}}$.
Recall our convention in \prref{sec:fggog} to denote elements in a disjoint union. 
Therefore, vertices in $V(\wt X)$ are denoted as $gG_P\cdot P$ with $gG_P \in \pi_1({\cG},T)/ G_P$ and $P \in V(Y)$. Likewise we denote edges as $gG_y \cdot y$. 

In fact, let $v=c_0y_1 \cdots c_{k-1}y_k$. Then the map $v\mapsto G_{P_0} \cdot P_0$ for $k=0$ and $v\mapsto c_0y_1 \cdots c_{k-1}y_kG_{t(y_k)}\cdot {t(y_k)}$ is a bijection because its inverse is well-defined by \prref{rem:britton}.

Directing the edges ``away-from-the-root'' $P$, a directed edge 
\begin{align*}
\big(c_0y_1 \cdots c_{k-1}y_k \,,\, 
c_0y_1 \cdots c_{k-1}y_kcy\big)\in E^+(\wt X)
\end{align*}
of $\wt X$ is mapped to 
the element $c_0y_1 \cdots c_{k-1}y_kc\, G_y \cdot y$. We conclude $\ov{gG_y \cdot y}=gyG_{\ov y}\cdot \ov y$.\footnote{In order to illustrate that formula, consider $k=1$ and $c_0y_1h$ with  $h\in G_{t(y_1)}=G_{s(y)}$ for some $y$. 
Then we have $h\in cG_{y}$ for some unique $c\in C_{y}$; and there is a directed edge to a coset which contains the element $c_0y_1cy$ with $c\in G_{t(y_1)}$. We have $t(y)=s(\ov y)$ and therefore the edge $\ov y$ leads back to the coset containing $c_0y_1cy\ov y = c_0y_1c$ which is indeed $c_0y_1G_{t(y_1)} = c_0y_1hG_{t(y_1)}$.}
 The incidences are $s(gG_y \cdot y) = gG_{s(y)}\cdot s(y)$ and $t(gG_y\cdot y)=s(gyG_{\ov y}\cdot \ov y) = gyG_{t(y)} \cdot t(y)$ for all edges $gG_y \cdot y$. That is, the source of an edge is defined by set wise inclusion.

Using \prref{cor:allsame} we may identify $\pi({\cG},P_0,P)$ with the 
fundamental group $\pi_1({\cG},T)$.
Therefore, for every $P_0\in V(Y)$ and every spanning tree $T$ of $Y$ we may write vertex and edge sets of the Bass-Serre tree $\wt X$ as disjoint unions:
\[
\begin{array}{lcl}
V(\wt X) = \cupi  \set{\pi_1({\cG},P_0)/ G_P}{P \in V(Y)}\!\!\! &=&\!\!\! \cupi  \set{\pi_1({\cG},T)/ G_P}{P \in V(Y)},\\
E(\wt X) = \cupi \set{\pi_1({\cG},P_0)/ G_y}{y \in E(Y)} \!\!\!&=&\!\!\! \cupi \set{\pi_1({\cG},T)/ G_y}{y \in E(Y)}.
\end{array}
\]

There is a natural action of ${\pi_1}({\cG},T)$ on the \BST\ $\widetilde X$ by left-multiplication: 
$g(f G_P\cdot P) = (gfG_P) \cdot P\in {\pi_1}({\cG},T)/G_P \sse V(\wt X)$.
Obviously, the action is without edge inversion. We have 
\[{\pi_1}({\cG},T)\bs \widetilde X = Y.\]


\subsection{Groups acting on connected graphs}\label{sec:gcg}
Graphs of groups arise in a natural way when a group $G$ acts on some connected graph $X = (V,E)$ (without edge inversion, \ie $\ov e \notin Ge$ for all $e \in E$) as it was described in \prref{sec:gmga}.

We let $Y = G\bs X$ be the quotient graph with 
vertex set $V(Y) = \set{Gv}{v \in V}$ and edge set $E(Y) = \set{Ge}{e \in E}$. Since $X$ is connected, $Y$ is connected, too. 
Choosing representatives $\iota(P)$ for $P\in V(Y)$, we assume $V(Y) \sse V$ (and henceforth write simply $P$ instead of $\iota(P)$). 
For $y\in E(Y)$ we choose a representative $\iota(y)\in E$ such that 
$s(\iota(y))= \iota(s(y))\in V(Y) \sse V$. 
For $P \in V(Y)$ we define the vertex group as stabilizer of the representative: $G_P = \{g \in G \mid gP = P\}$.
In the same way, for an edge $y \in E(Y)$, we define the edge group as stabilizer of the representative: $G_y = \{g \in G \mid g\iota(y) = \iota(y)\} \leq  G_{s(y)}$.

Note that we cannot enforce $t(\iota(y))\in V(Y)$.
However, there is some $g_y\in G$ such that $ g_y\, \iota(\ov y) = \ov {\iota(y)}$, see \prref{fig:gy}.
Note that $g_y$ is not unique, but only determined up to the edge stabilizer $G_{y}$.
We can choose $g_{\ov y}$ such that $g_y^{-1} = g_{\ov y}$. Indeed, given 
an undirected edge $\os{y,\ov y}$ we may choose first some $g_y$ such that $g_y\, \iota(\ov y) = \ov {\iota(y)}$. Then we define $g_{\ov y}$ as $\oi g_y$. It is clear that $g_y(g_{\ov y}(\ov y)) = \ov y$. Moreover, if $\iota(\ov y) = \ov{\iota(y)}$, we additionally require $ {g}_y = {g}_{\ov y} = 1\in G$. 

We obtain the isomorphisms $G_y\tto G_{\ov y}, a \mapsto g_y^{-1} a g_y$.
Let $F_{E(Y)}$ be the free group with basis $E(Y)$. We obtain a \homo $\phi:F_{E(Y)}\to G$ by defining
$\phi(y)=g_y$.

Recall that for vertices $P\in V(Y)$ we typically omit $\iota$ because we don't have to distinguish between $P$ and $\iota P$. For edges $y\in E(Y)$ we keep the notation $\iota(y)\in E$ because, in general,
$s(\iota(y))\neq t(\iota(\ov y))$, in contrast to $s(y)=t(\ov y)$ for $y\in E(Y)$.
We have $s(\iota(\ov y))\in V(Y)$ and $\oi{g}_{y}\cdot t(\iota(\ov y)) \in V(Y)$. 
Note that as abstract groups the vertex $G_P$ and edge groups $G_y$ are independent of the choice of representatives since stabilizers in the same orbit are conjugate.

\begin{figure}[ht]
	\begin{center} 
	\begin{tikzpicture}[xscale=1,yscale=1]
    \draw (-1.4,1) node (Pi-1) {\hspace{1.5cm}$\,s(\iota(y))\in V(Y)$};
    \draw (2.8,1) node (Pi) {\hspace{1.5cm}$\,\iota(t(y))\in V(Y)$};
    \draw (0.8,-0.2) node (Q') {\hspace{0cm}$Q'$};
    \draw (0.6,2) node (Q) {$Q$};
\draw[->,>=latex] (Pi-1) to node [above] {$\iota(y)\qquad$} (Q);
\draw[dashed, bend left=8,->,>=latex] (Pi) to node [above] {\hspace{2cm}$g_{y}=\phi(y)$} (Q);
\draw[dashed, bend left=8,->,>=latex] (Q') to node [above] {\hspace{1cm}$g_{y}$} (Pi-1);
\draw[->,>=latex] (Q') to node [below] {\hspace{1.7cm}$\oi g_y\iota(y)= \ov{\iota(\ov y)}$} (Pi);
\end{tikzpicture}
\end{center}
	\vspace{-0.5cm}
	\caption{Solid arrows are directed edges in $E$. Dashed arrows indicate actions of $G$. The picture holds for all $y\in Y$ for the chosen $g_y\in G$.}\label{fig:gy}
\end{figure}

\subsubsection{The homomorphism $\phi: F(\cG) \to G$}\label{sec:homphi}
Recall that $F(\cG)$ is the quotient of the free product 
of $\star_{P\in V(Y)} G_P$ and $F_{E(Y)}$ which is defined by the defining equations $\ov y a y= a^{\ov y}$ for all $y\in E(Y)$ and $a\in G_{y}$. 
Consider the \homo $\phi: \star_{P\in V(Y)} G_P \star F_{E(Y)}\to G$ defined by $\phi(g) = g$ for $g \in G_P$ for $P \in V(Y)$ and $\phi(y) = g_y \in G$ for every edge $y \in E(Y)$.  Recall that 
$\oi {g}_y =  g_{\ov y}$ and  $g_y^{-1}\cdot t(\iota(y))= s(\iota(\ov y))$. 
Since 
$ \phi(\ov y {a} y) =\phi(\ov y)\phi(a) \phi(y) =  g_{\ov y} {a}  g_y = a^{\ov y} = \phi(a^{\ov y})$, we obtain a well-defined
\homo $\phi: F(\cG) \to G$.
\begin{proposition}
 \label{prop:gXX}Let $\phi: F(\cG) \to G$ be defined as above
by $\phi(g)=g$ for $g\in G_P$ and  $\phi(y)=g_y$ for $y\in E(Y)$. Then
the restriction $\phi: \pi_1(\cG,P)\to G$ is surjective and injective 
on the vertex groups $G_P$ where $P\in V(Y)$. 
\end{proposition}
\begin{figure}[ht]
	\begin{center} 
\begin{tikzpicture}[xscale=1,yscale=1]
    \draw (-1.4,1) node (Pi-1) {\hspace{1.5cm}$\,\iota(P_{i-1})\in V(Y)$};
    \draw (2.8,1) node (Pi) {\hspace{1.5cm}$\,\iota(P_{i})\in V(Y)$};
    \draw (0.8,-0.2) node (Q') {\hspace{0cm}$Q'$};
    \draw (0.6,2) node (Q) {$Q$};
    \draw (2,-3) node (ePi-1) {$e_{i-1}\,\iota(P_{i-1})$};
    \draw (4,-1.8) node (ePi) {\hspace{1.5cm}$e_i\,\iota(P_{i})= e_{i-1}Q$};
\draw[->,>=latex] (Pi-1) to node [above] {$\iota(y_{i})\qquad$} (Q);
\draw[dashed, bend right=50, ->,>=latex] (Pi-1) to node [right] {$e_{i-1}$} (ePi-1);
\draw[dashed, bend left=8,->,>=latex] (Pi) to node [above] {\hspace{2.5cm}$g_{y_{i}}
= \phi(y_i)$} (Q);
\draw[dashed, bend right=10,->,>=latex] (Pi) to node [right] {$e_{i}$} (ePi);
\draw[->,>=latex] (ePi-1) to node [above] {\hspace{-.8cm}$e_{i-1}\,\iota(y_i)$} (ePi);
\draw[dashed, bend left=8,->,>=latex] (Q') to node [above] {\hspace{1cm}$g_{y_{i}}$} (Pi-1);
\draw[->,>=latex] (Q') to node [above] {${}$} (Pi);
\end{tikzpicture}
\end{center}
	\vspace{-0.5cm}
	\caption{We have  $e_0=1$ and $P_i=\iota(P_i)$ in the notation of (\ref{eq:pathX}).}\label{fig:STHD}
\end{figure}
\begin{proof}
Its restriction to the vertex groups $G_P$ is injective because $\phi(g) = g$ for $g\in G_P$. Let us show that the restriction of $\phi$ to $\pi_1(\cG,P)$ is surjective. 

Let $g \in G$. Recall that we assumed $X$ to be connected. Hence, in $X$ we 
can find a path
\begin{align}\label{eq:pathX}
P = e_0P_0 \ra{e_{0} \iota(y_1)}e_1P_1 \cdots \ra{e_{k-1}\iota(y_k)}e_kP_k = gP
\end{align}
such
that $P_i = \iota(P_i)\in V(Y)$, $y_i \in E(Y)$, $P_0 = P_k = P$, $e_i\in G$, and $e_0 \in G_P$. Since $g\in e_kG_P$, it is enough to show that $e_k \in \phi(\pi_1(\cG,P))$. 
To see this, let us show $e_i \in \phi(\pi(\cG,P,P_i))$ for all $0\leq i \leq k$ by induction.

Since $e_0 \in G_P$, we may assume $i\geq 1$ and  $e_{i-1} \in \phi(\pi_1(\cG,P,P_{i-1}))$ by induction. 
As $ e_{i-1}t(\iota(y_{i})) = e_{i} P_{i}$, we have $ e_{i-1}^{-1}e_{i} P_{i}  = t(\iota(y_{i})) = s(\ov{\iota(y_{i})}) = s( g_{ y_{i}}\iota(\ov y_{i})) = g_{ y_{i}} P_i = \phi(y_i) P_i$. 
This implies $e_i= e_{i-1}\phi(y_i) h_i$ for some $h_i\in G_{P_{i}}$. 
We have $e_{i-1}=\phi(h_0y_1h_1 \cdots y_{i-1}h_{i-1})$ 
with $h_0y_1h_1 \cdots y_{i-1}h_{i-1}\in \pi(\cG,P,P_{i-1})$. Hence
\begin{align}\label{eq:ei-1}
e_{i}= \phi(h_0y_1h_1 \cdots y_{i-1}h_{i-1}y_i h_{i})\in \pi(\cG,P,P_i).
\end{align}
Hence, $\phi: \pi_1(\cG,P)\to G$ is surjective.
The situation is explained and depicted in \prref{fig:STHD}.
In the figure solid arrows are directed edges in $E(X)$ and 
dashed arrows indicate the action of $G$ on $E(X)$: a dashed arrow 
$R\dashrightarc{g}R'$ means $g\cdot R=R'$. Note that
$g'\cdot R=R'$ implies $g'=gh$ for some unique $h\in G_R$. 
\end{proof}
\begin{corollary}
 \label{cor:gXX} Let $T$ be a spanning tree of $Y=G\bs X$ and  the representatives $\iota: Y \to X$ chosen such that $T' = \iota(T)$ is a subgraph of $X$ (as in \prref{lem:serreprop14} below). 
  Let $\phi: F(\cG) \to G$ be defined as above. Then
the restriction $\phi: \pi_1(\cG,P)\to G$ factorizes through 
$\pi_1(\cG,T)$.
\end{corollary}

\begin{proof}
As $T'$ is a subgraph of $X$, we have  $\iota(\ov y) = \ov{\iota(y)}$ for all $y \in E(T)$. Thus, by our assumption $\phi(y)=g_y=1$ for all $y\in E(T)$. Hence the result follows.
\end{proof}
For later use, we refine the embedding of $V(Y)$ into $V$ by adding another condition. For that we choose any spanning tree $T$
in the graph $Y$. It exists because $X$ is connected which implies that $Y$ is connected, too. 
A proof for the following (well-known) result can be found in \cite[Prop.~14]{serre80}. 
\begin{lemma}\label{lem:serreprop14}
Let $X$ be connected and $T$ be a spanning tree in $Y=G\bs X$. 
Then there is an embedding $\iota$ of $V(Y)$ into 
$V$ and  tree $T'=(\iota(V),E(T'))$ in $X$ such that the mapping
$e\in E(T')$ to $Ge\in E(Y)$ induces a bijection between $T'$ and~$T$.
\end{lemma}
\begin{proof}
For finite $Y$ the proof is an easy induction on $\abs Y$. The extension to infinite graphs $Y$ uses the Lemma of Zorn. For details we refer to \cite{serre80}.
\end{proof}

\subsubsection{The graph morphism $\psi:\wt X \to X$}\label{sec:morpsi}

In the following we let $Y = G\bs X$ for $X = (V,E)$ be the quotient graph and $\cG$ the resulting \gog.  
Moreover, let $T$ be a spanning tree of $Y$. 
By \prref{lem:serreprop14}, $T$ lifts to a tree $T'$ in $X$. 
We choose our representatives such that $ V(Y) = V(T')\sse V$ (as before we identify a vertex $P$ of $Y$ with its representative $\iota(P)$).
We lift an edge $y\in E(T)$ to the corresponding edge $\iota(y) \in E(T')$ (\ie $y = G \iota(y)$). Notice that this means $s(\iota(y))\in V(T) =  V(Y)$, which is consistent with the requirements we posed at the beginning of \prref{sec:gcg}, and we have $\iota(\ov y) = \ov{\iota(y)}$ for $y\in E(T)$. We lift all other edges arbitrarily as described at the beginning of \prref{sec:gcg} meaning that $s(\iota(y))\in V(Y)\sse V$ for all $y \in E(Y)$.

Then we have (note that the unions are indeed disjoint)\footnote{Note that the dot here denotes the action of $G$ on $X$. However, reading it as the dot of a disjoint union also gives the right understanding.} 
\[\begin{array}{rcl}
V(X) &=& 
\bigcup \set{G \cdot P}{P\in V(Y)}
\\
E(X) &=& 
\bigcup \set{G\cdot \iota(y)}{y\in E(Y)}
\end{array}
\]

Let us summarize our construction up to now: We have a homomorphism $\phi:  \pi_1(\cG,P_0) \to G$ with $\phi(g) = g$ for $g \in G_P$ and $\phi(y) = g_y$ for $y \in E(Y)$.
As $\iota(\ov y) = \ov{\iota(y)}$ for $y \in E(T)$, by assumption, we have $g_y = 1$ for $y \in E(T)$. Hence, by \prref{cor:gXX}, we see that $\phi:  \pi_1(\cG,P_0) \to G$ factorizes through $\pi_1(\cG,T)$. We also write  $\phi$ for the induced homomorphism $\phi:  \pi_1(\cG,T) \to G$.

As before, let $\wt X$ denote the Bass-Serre tree.
We define a mapping $\psi: \tX\to X$ by
\[\begin{array}{rcll}
\psi(gG_P\cdot P )&=&\phi(g)\cdot P &\text{ for }P\in V(Y),\\
\psi(gG_y\cdot y)&=&\phi(g) \cdot \iota(y)& \text{ for }y\in E(Y).
\end{array}
\]
Notice that here we use the definition of the Bass-Serre tree via $\pi_1({\cG},T)$.

\begin{proposition}\label{prop:locinj}
The mapping $\psi:\wt X \to X$ is a locally injective graph epimorphism. 
\end{proposition}
\begin{proof}
Let $x = gG_y\cdot y \in E(\wt X)$ be an edge. Then $\psi$
respects the sources (and hence the targets) thanks to the following two lines:
\begin{align*}
		 \psi(s(gG_y\cdot y )) &= \psi(g G_{s(y)} \cdot s(y)) = \phi(g) s(y) \\
		 s(\psi(gG_y\cdot y )) &= s(\phi(g) \cdot \iota(y)) = \phi(g) s(\iota(y)) = \phi(g) s(y) 
\intertext{The mapping $\psi$
respects the involution thanks to}
	\psi(\ov{gG_y\cdot y }) &= \psi(gy G_{\ov y} \cdot \ov y) = \phi(gy) \iota(\ov y)= \phi(g) g_y \iota(\ov y) 
	\\
	\ov{\psi(gG_y\cdot y )} &= \ov{\phi(g) \cdot \iota(y)} = \phi(g) \ov{\iota(y)}= \phi(g) g_y \iota(\ov y)
\end{align*}
Thus, $\psi$ is a graph morphism.
The surjectivity of $\psi$ on vertices and edges follows immediately from the definition and the surjectivity of the \homo $\phi:\pi_1(\cG,T) \to G$. 
	
Consider any $Q\in V(\wt X)$ and edges $x,z\in E(\wt X)$ with 
$s(x)=s(z)=Q$. We have to show that $x\neq z$ implies $\psi(x)\neq \psi(z)$ in $E(X)$. This is clear, if $Gx\neq Gz$ because the orbit
$Gx$ is an edge in $Y$ and $\psi(gG_y\cdot y) \in G\cdot \iota(y)$ for
$y=Gx\in G\bs X$. 
Thus, we may assume $Gx=Gz$. 

According to \prref{sec:bst}, after choosing a base point $P_0\in  V(Y)$, we can identify 
$Q$ with some $v= c_0y_1 \cdots c_{k-1}y_k\in \IRR(S_\cG)$ and direct the edge $x$ and $z$ 
such that first, $x,z\in E^+(\wt X)$ and second, 
\begin{align}\label{eq:xSG}
x&=\big(c_0y_1 \cdots c_{k-1}y_k \,,\, 
c_0y_1 \cdots c_{k-1}y_kc y\big)\\
z&=\big(c_0y_1 \cdots c_{k-1}y_k \,,\, 
c_0y_1 \cdots c_{k-1}y_k d  y\big)
\end{align}
where $c,d\in C_{y}$ are coset representatives with  $1\in  C_{y}\sse G_{s(y)}$. 

A directed edge 
\begin{align*}
\big(c_0y_1 \cdots c_{k-1}y_k \,,\, 
c_0y_1 \cdots c_{k-1}y_kcy\big)\in E^+(\wt X)
\end{align*}
of $\wt X$ is mapped to 
the element $\phi(c_0y_1 \cdots c_{k-1}y_kc) \cdot \iota(y)$. As for  $c,d\in C_{y}$ we have $c \cdot \iota(y) = d \cdot \iota(y)$ if and only if $c= d$, it follows that
 $\phi(x) = \phi(z)$ if and only if $c= d$ if and only if $x = z$.
\end{proof}
\subsection{Groups Acting on Trees}\label{sec:actT}
The following result describes the structure of a group acting on a tree without edge inversion
as a fundamental group of a \gog. It is the central result in Bass-Serre theory. 

\begin{theorem}[\cite{serre80}]\label{thm:bst}
Let $G$ be a group acting on a connected graph $X$ without edge inversion and let $\cG$ be the associated graph of groups over $Y = G \bs X$. 
Let $T$ be a spanning tree of $Y$, $\psi : \wt X \to X$ 
and $\phi : {\pi_1}(\cG, T) \to G$ as above. Then the following assertions are equivalent:
\begin{enumerate}
\item The graph $X$ is a tree. \label{bst_1}
\item The morphism $\psi: \wt X \to X$ is an isomorphism of graphs.\label{bst_2}
\item The \homo $\phi : {\pi_1}(\cG, T) \to G$ is an isomorphism of groups. \label{bst_3}
\end{enumerate}
\end{theorem}

\begin{proof}
\prref{prop:locinj} $\psi: \wt X \to X$ is a surjective and locally injective epimorphism. 
\prref{lem:locallyinj} yields the equivalence between \ref{bst_1} and \ref{bst_2}. 

It remains to show the equivalence of \ref{bst_3} and \ref{bst_2}.   
Thus, we have to show that $\psi$ is injective on vertices \IFF $\phi$ is injective. 

By definition, a vertex $gG_P\cdot P$ of the \BST is mapped to the 
vertex $\phi(g)P$. If $\phi$ is not injective, then there is $1 \neq g \in {\pi_1}(\cG, T)$ with $\phi(g) = 1$. We have $g \not \in G_P$ because by definition $\phi$ is injective on the vertex groups. Hence, we have $gG_P\cdot P \neq G_P \cdot P\in V(\wt X)$, and $\psi(gG_P\cdot P) = \phi(g) P = P = \psi(G_P\cdot P)$. Thus, 
 $\psi$ is not injective on vertices.

 Let $\phi$ be injective. From $\psi(gG_P\cdot P) = \psi(hG_Q\cdot Q)$, \ie  $\phi(g) P = \phi(h) Q$, it follows that $P=Q$ and $\ov g h\in G_P$. Hence, $gG_P\cdot P = hG_Q\cdot Q$.
 This proves the 
equivalence between \ref{bst_2} and \ref{bst_3}.
\end{proof}

Serre's well-known characterization of free groups becomes a direct corollary. 
Recall that a group acts freely on a graph if all the vertex stabilizers are trivial.

\begin{corollary}[\cite{serre80}]\label{cor:freegr}
A group $G$ is free \IFF $G$ acts freely on a tree without edge inversion.
In particular, subgroups of free groups are free. 
\end{corollary}

\begin{proof}
 A group acts freely on its Cayley graph. If the group is free we can choose the 
 Cayley graph to be a tree and the action to be without edge inversion.
 The other direction follows from \prref{thm:bst} because 
 the fundamental group ${\pi_1}(\cG, T)$ is free, if all vertex groups are trivial. 
\end{proof}

\prref{thm:bst} yields also Schreier's formula for subgroups of finite index of free groups. This formula can be generalized as follows.

\begin{corollary}[\cite{Karrass73}{\cite[II.3.7]{dicks80}}]\label{cor:free_tree}
Let $\cG$ be a finite graph of groups over the graph $Y$ with finite vertex groups. Let $G={\pi_1}(\cG,T) $ and let $F$ be a free subgroup of $G$ of finite index and $r(F)$ its rank. Then the following equation holds:
\[\dfrac{r(F) -1}{(G:F)} =\sum_{y\in E(Y)} \frac{1}{2\cdot\abs{G_y}}-\sum_{P\in V(Y)} \frac{1}{\abs{G_P}}\] 
\end{corollary}

\begin{proof}
Let $\wt X$ be the \BST. Since $F$ has trivial intersection with the vertex stabilizers, it acts freely and without inversion on $\wt X$. Hence, $F$ is isomorphic to the fundamental group of $F\bs \wt X$ with all vertex groups being trivial. 
Let $T'$ be a spanning tree of $F\bs \wt X$. Then as above the unoriented edges of $F\bs \wt X - E(T')$ form a basis of $F$. This yields
\[r(F)= \frac{1}{2}\abs{E(F\bs \wt X)} -\left(\abs{V(F\bs \wt X)}-1\right).\]
Note, that the factor $1/2$ appears because we count every unoriented edge twice. By the construction of $\wt X$, we know for the number of edges:
\[\abs{E(F\bs\wt X)}=\sum_{y\in E(Y)}\abs{F\bs (G/G_y)}=\sum_{y\in E(Y)}\abs{(F\bs G)/G_y}= \sum_{y\in E(Y)}\frac{(G:F)}{\abs{G_y}}\]
The analogue formula holds for vertices. Hence, the result follows.
\end{proof}

\begin{proposition}[\!\,{\cite[II.3.7]{dicks80}}]\label{prop:endl_erz_knoten}
Let $\cG$ be a graph of groups over a finite graph $Y$ with finitely generated edge groups. Then ${\pi_1}(\cG)$ is finitely generated if and only if all the vertex groups are finitely generated. 
\end{proposition}

\begin{proof}
	Clearly, if all vertex groups are finitely generated and $Y$ is finite, the fundamental group is finitely generated, too. Thus, we only have to show the other implication.
	
	For each edge $y \in E(Y)$ let $\Sigma_y$ be a finite generating set of $G_y$.
 Let $Z = \bigcup\set{\Sigma_y}{y \in E(Y)}$ be the union over all the generating sets $\Sigma_y$ of the edge groups. Note that $Z$ is finite.  Moreover, for each 
 vertex $Q \in V(Y)$ let $\Sigma_Q$ be some generating set of the vertex group $G_Q$ (possibly infinite).   
 
 Now, let $P \in V(Y)$ be a fixed vertex.
 There is a finite generating set $\Sigma$ of ${\pi_1}(\cG) = {\pi_1}(\cG,P)$ inside $\bigcup\set{\Sigma_Q}{Q \in V}\cup E(Y) \cup Z$ such that $E(Y) \cup Z \sse \Sigma$. Consider any generator $g \in \Sigma_P$. It is enough to 
 show that $g$ can be expressed as a word in $(\Sigma \cap G_P)^*$. 
 To see this, write $g$ as a word in ${\pi}(\cG,P,P)$ (according to \prref{eq:pipq}) with letters from 
 $\Sigma$ with the fewest number of letters from $ E(Y)$ possible. Assume this word contained a factor
 $\ov y a y$ for some $y \in E(Y)$ and $a\in G_y$.
Then we could perform a Britton reduction replacing $\ov y a y$ by
 $a^{\ov y}$ (note that $a$ and $a^{\ov y}$ here are written as words over $Z$). 
 This would lead to a word in $\Sigma^*$ with fewer letters from $ E(Y)$ because 
 $E(Y) \cup Z \sse \Sigma$. Hence, the word representing $g$ is Britton reduced. Now, a Britton reduced word in ${\pi}(\cG,P,P)$ for an element in $G_P$ 
uses letters from $\Sigma\cap G_P\sse \Sigma_P \cup Z $, only. 
 \end{proof}

\subsection{Finite Vertex Groups}
The aim of this section is to show that the fundamental group of a finite graph of groups with finite vertex groups is virtually free. The result is due to Karrass, Pietrowski and Solitar \cite{Karrass73}. Here, we present the proof from \cite[II.3.6]{dicks80}. The converse is due to \cite{Karrass73}, too: virtually free groups are the fundamental groups of finite graphs of groups with finite vertex groups. (Recall our convention that virtually free groups are, by definition, finitely generated, see  \prref{def:virtfree}.) We will prove this fact in Sections \ref{sec:treewidth} and \ref{sec:cuts}.

An action of a group $G$ on a set $X$ can be viewed as a homomorphism $\alpha: G\to \Sym{X}$, where $\Sym{X}$ denotes the symmetric group over $X$.

\begin{lemma}\label{lem:2_Ops}
Let $G$ be a group acting freely (= fixed-point free) on a finite set $X$ in two ways $\alpha,\beta : G \to \Sym{X}$. Then there is some $\phi \in \Sym{X}$ such that for all $ g\in G$
we have
\[\alpha(g)=\phi^{-1}\circ\beta(g)\circ \phi.\] 
\end{lemma}
\begin{proof}
We can choose systems of representatives $R,S \sse X$ such that 
$X=\cupi _{r\in R} \alp(G) (r )= \cupi _{s\in S} \bet(G) (s)$ where the unions are disjoint.
Since both actions are free and $X$ is finite, it follows that $G$ is finite and all orbits $\alp(G) (r )$, $\bet(G) (s)$ have size $\abs{G}$.
Hence, $\abs{R}=\abs{S}=|X|/|G|$, \ie there is a bijection between these two sets, which can be extended to a bijection $X \to X$ via
$\alp(g)(r)\mapsto \bet(g)(s)$ whenever $R\ni r\mapsto s\in S$. One can easily verify that this bijection is the element $\phi$ we were looking for.
\end{proof}

\begin{theorem}[\cite{Karrass73}]\label{thm:endl_gog_virt_free}
Let $\cG$ be a graph of groups over a finite connected graph $Y$ with finite vertex groups. Then the fundamental group $\pi_1(\cG)$ has a finitely generated free subgroup of finite index.
Thus, $\pi_1(\cG)$ is a finitely generated virtually free group. 
\end{theorem}

\begin{proof}
Let ${X}$ be some finite set such that $\abs{G_{P}}$ divides $\abs{{X}}$ for every vertex ${P}\in V(Y)$. Therefore, for each $P$ we can choose a free action of $G_P$ on ${X}$, and hence an injective homomorphism $G_{P} \to \Sym{{X}}$. For each edge group $G_y$ we obtain two free actions on ${X}$. By \prref{lem:2_Ops}, for each $y \in E(Y)$ we can choose some $\phi = \phi(y) \in \Sym{X}$ such that the following diagram commutes.
 \[
 \begin{xy}
 		\xymatrix @R=5pt{
 		& G_{s(y)} \ar[r] & \Sym{X} \ar[dd]^{\sig \mapsto \phi^{-1}\circ\sig\circ \phi} \\
 G_y \ar[ru] \ar[rd] & & \\
 		& G_{t(y)} \ar[r] & \Sym{X} 
 		}

 \end{xy}
 \] 
 The universal property of the group $F(\cG)$ yields 
 a homomorphism $h: F(\cG) \to \Sym{X}$ such that restriction to $G_{P}$ is injective for all $P\in V(Y)$.
We fix some $P\in V(Y)$ and we let $F$ be the kernel of $h$ inside $\pi_1(\cG,P)$. Thus, 
 $F = \set{g \in \pi_1(\cG,P)}{h(g) = 1}$. Then we have $F \cap G_{Q} =\smallset{1}$ for all ${Q}\in V(Y)$. This means that the group $F$ acts freely on the \BST.
 Since $\pi_1(\cG,P)$ acts without edge inversion, the same is true for 
 $F$. We conclude that $F$ is free by \prref{cor:freegr}. Furthermore, we have $[\pi_1(\cG,P) : F]< \infty$, because $\abs{X} < \infty$. Thus, by \prref{cor:free_tree}, $F$ is finitely generated.
\end{proof}

\subsection{Embedding into Semidirect Products}
\newcommand{\GR}{Q}

Using an approach by Dahmani and Guirardel in \cite{DahmaniGui10}, the proof of \prettyref{thm:endl_gog_virt_free} can be refined to show that 
every fundamental group of a finite graph of finite groups and thus, by the results of the following sections, every virtually free group $G$, can be embedded in a special way into a semidirect product $F \rtimes \GR$ where $F$ is free and $\GR$ is a finite quotient group of $G$.

\begin{proposition}\label{prop:sdvf} 
Let $\cG$ be a graph of groups over a finite connected graph $Y$ with finite vertex groups.
	 Then there is a finite quotient group $\GR$, a finite graph $\Gam$, and an orientation $S^+$ of its edge set such that $\pi_1(\cG)$ embeds into a semidirect product
	$F_{S^+}\rtimes \GR$ where the action of $\GR$ on the free group $F_{S^+}$ is induced by graph morphisms on $\Gam$ (in particular, it is by permutations of vertices). Moreover, the embedded group $\pi_1(\cG)$ satisfies 
$\pi_1(\cG) \cap F_{S^+}=\pi_1(\Gam,P)$ (seen as subgroup of $F_{S^+}$) for some vertex $P\in V(\Gam)$.
\end{proposition}
Note that here $\pi_1(\Gam,P)$ is the fundamental group of the graph of groups where all vertex groups are trivial (\ie the fundamental group as a topological space).

\begin{proof}
	Let $G= {\pi_1}({\cG},P_0)$ for some $P_0 \in V(Y)$ and let $F$ be the free normal subgroup constructed in the proof of \prettyref{thm:endl_gog_virt_free} and $\GR$ the quotient $G/F$, and $\phi:G \to \GR$ the projection. 
	Recall that $G$ acts
	on the \BST $\wt X$ by left multiplication. Thus, we can form the quotient graph
	$\Gam= F\backslash {\wt X}$, which is finite because for every $G$-orbit of vertices (resp.\ edges) of the \BST there are at most $\abs{Q}$ vertices (resp.\ edges) in $\Gam$. Moreover, the action of $G$ on $\wt X$ induces an action of $\GR$  on $\Gam$. Let $S = E(\Gam)$ be the set of edges of $\Gam$ and $S^+$ and orientation of it (one edge for every pair $e, \ov e$). Since the action on $\wt X$ is without edge inversion, we obtain an action of $Q$ on $S^+$, which extends to an action on $F_{S^+} = F_S/\!\set{e \ov e = 1}{e \in S}$ denoted by $(f,y) \mapsto {}^f\!y$ for $f \in Q$ and $y \in F_{S^+}$. Now, we can define the semidirect product 
	$F_{S^+}\rtimes \GR$. Its elements are denoted  by $[x,f]$ with $x\in F_{S^+}$ and $f\in \GR$. The multiplication in $F_{S^+}\rtimes \GR$ is
	$$[x,f]\cdot[y,g] = [x\, {}^f\!y,fg].$$
	We define a map $\theta: G \to F_{S^+}\rtimes \GR$ as 
	$$\theta(w) = [\psi(w), \phi(w)]$$
	with $\psi$ as follows: take the unique geodesic from $G_{P_0} \cdot P_0$ (the ``base point'') to $w G_{P_0} \cdot P_0$ in the \BST $\wt X$ and project it to $\Gam$. Now $\psi(w)$ is defined as the sequence of edges visited by the projected path (as a word over $S$).
	We have to show that 
	\begin{itemize}
		\item $\theta$ is a homomorphism,
		\item $\theta(F) = \pi_1(\Gam,P)$, and
		\item $\theta$ is injective.
	\end{itemize}
	The first point is due to the following observation: let $v,w \in G= {\pi_1}({\cG},P_0)$. In order to obtain the geodesic path from $G_{P_0} \cdot P_0$ to $vw G_{P_0} \cdot P_0$, we first follow the geodesic from $G_{P_0} \cdot P_0$ to $v G_{P_0} \cdot P_0$, then we take the geodesic from $G_{P_0} \cdot P_0$ to $w G_{P_0} \cdot P_0$ but ``shift it by $v$'' so that it becomes a path from $vG_{P_0} \cdot P_0$ to $vw G_{P_0} \cdot P_0$; finally, in order to obtain the geodesic we have to reduce this path in the middle where the two geodesics meet. In the quotient graph $\Gam$, it suffices to shift the path by $\phi(v)$; the reduction in the middle is done by freely reducing the word in $F_{S^+}$.
	
	For the second point notice that $\theta(F) \leq F_{S^+}$ because $F$ is the kernel of $\phi$. By \prettyref{thm:bst}, $\psi_{\vert F}$ is an isomorphism $ F \to \pi_1(\Gam,\star )\leq F_{S^+}$ where $\star = FG_{P_0} \cdot P_0$ is the projection of the base point of the \BST.
	The third point follows because $\phi(g) \neq 1$ for $g \not\in F$ and, by the above, $\theta$ is injective on $F$.
\end{proof}

By \prref{thm:endl_gog_virt_free} we know that every fundamental group of a finite graph of finite groups is virtually free. 
Later, without using this section, we show the result of \cite{Karrass73}
in ``its full glory'': being the fundamental group of a finite graph of finite groups is the same as being a finitely generated virtually free group. Thus, we obtain the following corollary. (The if-part is because subgroups of virtually free groups are virtually free: if $H\leq G$ and the free group $F$ has finite index $k$ in $G$, then $H\cap F$ is free by \prref{cor:freegr} and it has index at most $k$ in $H$.)

\begin{corollary}\label{cor:sdvf}Let $G$ be a finitely generated group. 
Then $G$ is virtually free \IFF $G$ is a subgroup of a semidirect product
of a finitely generated free group by a finite group. 
\end{corollary}

\begin{remark}\label{rem:sizeG}
	For an estimation on the size of $\GR= G/F$ we have some doubly exponential bound in the size of $\cG$. Indeed, $\GR$
appears as subgroup in $\Sym{X}$ where $\abs X \leq \prod\set{\abs{G_P}}{P\in V(Y)}$ is singly exponential in the size of $\cG$. 
\end{remark}

\begin{example}
One of the easiest and most prominent examples of an amalgamated product (fundamental group of a graph of groups with two vertices and one connecting edge~-- see \prref{ex:amal}) is $\SL(2,\Z)$, the special linear group of $2\times 2$ matrices over $\Z$. It is well-known that $\SL(2,\Z)=\Z/4\Z \star_{\Z/2\Z} \Z/6\Z$. 
The corresponding graph of groups consists of two vertices $P,R$ and one undirected edge $\oneset{y, \ov y}$.
The vertex groups are $G_P= \Z/4\Z= \gen \rho $ and $G_ {R}= \Z/6\Z= \gen \del$; the edge group is $G_y=\Z/2\Z=\gen\tau$. Hence, $\rho^2=\tau=\del^3$ and $\tau^2=1$. Mapping $\rho$ to $3$ and $\del$ to $2$, we obtain a surjective 
\homo $ {G}\to  {Q}$ where $ {Q}=\Z/12\Z$. The kernel is a free group $F$ of index $12$ in $ {G}$. (It is well-known that $F$ is equal to the 
commutator subgroup $[ {G}, {G}]$.) 
The free subgroup $F$ is the fundamental group of the graph $\Gam$ 
depicted in \prref{fig:sl2z}. (In our case $\Gam$ is the complete bipartite graph $K_{32}$.)
\begin{figure}[h!]
	\begin{center}		
		\begin{tikzpicture}[
		xscale=2.5,
		yscale=0.8,
		inner sep=.5,
		arc/.style={->, >=latex},
		node/.style={circle},
		label/.style={pos=0.5, above=2pt}
		]
		
		
		\node [node] (P1) at (0,4) {$P_1$};
		\node [node] (P2) at (0,2) {$P_{\delta}$};
		\node [node] (P3) at (0,0) {$P_{\delta^2}$};
		\node [node] (R1) at (2,3) {$R_1$};
		\node [node] (R2) at (2,1) {$R_{\rho}$};
		
		
		\draw [arc] (P1) to node [label] {$a$} (R1);
		\draw [arc] (P1) to node [label] {$b$} (R2);
		\draw [arc, bend left=80,dashed,semithick, distance=2.8cm] (P2) to node [label] {$c$} (R1);
		\draw [arc] (P2) to node [label] {$d$} (R2);
		\draw [arc, bend right=70,dashed,semithick, distance=2.2cm] (P3) to node [label] {$f$} (R1);
		\draw [arc] (P3) to node [label] {$e$} (R2);
		
		\end{tikzpicture}
		\caption{The graph $\Gam$ for $\SL(2,\Z)$ with oriented 
			edge set $S^+=   \os{a\lds f}$ and bridges $c$ and $f$.}
		\label{fig:sl2z}\end{center}
\end{figure}
 Choosing $c$ and $f$ as bridges we have
$F=F(b \ov d c \ov a,b \ov ef\ov a)$, a free group of rank $2$. 
The action of $Q$ on $\Gam$ is as follows: $\rho$ stabilizes the vertices $P_\alp$ and $\del$ stabilizes the vertices $ {R}_\bet$. We can identify  $ {Q}/G_P =\os{1,\del,\del^2}$ and  $ {Q}/G_ {R} =\os{1,\rho}$.
Without restriction $\del P_\alp= P_{\del \alp}$ and $\rho  {R}_\bet=  {R}_{\rho \bet}$. We see that the action of $Q$ on $\Gam$ is not faithful. 
Therefore, we really need both factors of the semidirect product to obtain an injective mapping.

We conclude that $\SL(2,\Z)$ embeds into a semidirect product of 
a free group of rank $6$ by $\Z/12\Z$ acting as permutations on the basis of the free group. 
\end{example}

\section{Pregroups and Geodesic Rewriting Systems}\label{sec:pregroups}
We now turn to the notion of pregroup in the sense of Stallings
\cite{Stallings71,Stallings87}.
\begin{definition}[Pregroup]
A \emph{pregroup} $P$ consists of a set $P$ together with a partial multiplication $D \to P$, $(x,y)\mapsto xy$ for some $D\sse P\times P$ (the domain of the multiplication~-- if $(x,y)\in D$, we say $xy$ is defined), an involution $P\to P$, $x\mapsto \ov x$, and a distinguished element $1_P \in P$ such that the following conditions are fulfilled:
\begin{enumerate}[(P1)]
\item $(1_P,x), (x,1_P) \in D$ and $1_Px=x1_P=x$ for all $x\in P$. \label{P1}
\item $(\ov x,x),(x,\ov x) \in D$ and $x\ov x=\ov xx=1_P$ for all $x\in P$.\label{P2}
\item For $x,y,z \in P$ with $(x,y),(y,z) \in D$ (\ie  $xy$, $yz$ are defined):\label{P3}
\[(xy,z)\in D \text{ if and only if } (x,yz)\in D.\]
Moreover, whenever $(xy,z) ,(x,yz)\in D$, then $(xy)z=x(yz)$ and we simply write $xyz$.
\item For $w,x,y,z \in P$ we have:\label{P4}\[(w,x),(x,y),(y,z)\in D \text{ implies } (w,xy)\in D \text{ or } (xy,z) \in D.\]
\end{enumerate}
\end{definition}
 Note that the involution is not required to be without fixed points.
Stallings' original description also contained one more requirement:
if $ xy$ is defined, then so is
$\ov y\cdot \ov x,$ and
$\ov{(xy)} = \ov y \cdot \ov x$.
This, however, is an immediate consequence of (P1), (P2), and (P3).

Every group $G$ is a pregroup. Moreover, if $P \sse G$ is closed under forming inverses, then (P\ref{P1})--(P\ref{P3}) are fulfilled. However, (P\ref{P4}) is a strong additional requirement, which finally implies that a pregroup defines a geodesic rewriting system.

\begin{definition}[Universal Group]
The \emph{universal group} $U(P)$ of a pregroup $P$ is defined as
\[U(P) = P^* / \set{ab=c,1_P =1}{(a,b) \in D, ab = c}.\]
\end{definition}
It follows from (P2) that $U(P)$ in fact is a group. It is universal with the following property:
Given a group $G$ and a map $\phi :P\to G$, with $\phi(a)\phi(b)=\phi(ab)$ for all $a,b \in P$, whenever $ab$ is defined. Then there is a unique group homomorphism $U(P)\to G$ extending $\phi$.

\begin{example}\label{ex:prefree}
 If ${\Sigma}$ is any set, then
the disjoint union $P= \smallset{1_P} \cup \Sigma \cup \ov{\Sigma}$
where $\ov{\Sigma}$ is a copy of ${\Sigma}$
yields a pregroup with involution given by $\ov{1_P}=1_P$, $\ov{\ov{a}}=a$
for all $a\in \Sigma$ and $p \ov{p} = 1_P$
for all $p\in P$.
In this case the universal
group $U(P)$ is the free group $F_{\Sigma}$.
\end{example}

\begin{example}\label{ex:preamal}
Let $A$ and $B$ be groups intersecting in a common subgroup $H$.
Consider the subset $P=A\cup B \subseteq  A \star_HB$. Let $D=A\times A\cup B\times B$.
The result of the partial multiplication $p \cdot q$ is defined by the product in $A$ resp.\ $B$.
Then $P$ is a pregroup and we have
$U(P) =  A\star_HB$.
\end{example}

Stallings \cite{Stallings71} showed that the composition of the
inclusion $P \rightarrow P^\ast$
 with the standard quotient map $P^\ast \rightarrow U(P)$ is injective. In order to do so, he introduced {\em reduced forms} of elements of $U(P)$, which are words in $a_1\dots a_n \in P^*$ such that $(a_i,a_{i+1})\not\in D$ for all $i$. Then he showed that two such words represent the same element in $U(P)$ if and only if they can be transformed into each other by a finite sequence of rewriting steps of the form $$ a_1\dots a_i\,a_{i+1}\dots a_n\RA{}
a_1\dots [a_ic] \,[\ov ca_{i+1}]\dots a_n,$$
where $(a_i,c),(\ov c,a_{i+1}) \in D.$
As before, $[ab]$ denotes the element in $P$, whereas $ab$ denotes the word in $P^*$.
Here, we give a simpler proof, which follows \cite{ddm10}, by constructing a confluent rewriting system which defines $U(P)$.

\begin{definition}\label{def:sp}
For a pregroup $P$ we can define a rewriting system $S_P\sse P^*\times P^*$ by the following rules:
\[
\begin{array}{rcll}
1_P & \longrightarrow & 1 &\mbox{(= empty word)}\\
ab & \longrightarrow &[ab] & \mbox{if } \; (a,b) \in D \\
ab & \longrightarrow &[ac][\ov cb]& \mbox{if }
 \;(a,c), \, (\ov c,b) \in D
\end{array}
\]
\end{definition}
Obviously, we have $P^*/\DAS{*}{S}\, \cong U(P)$. 
\begin{lemma}\label{lem:sp_konfluent}
The rewriting system $S_P$ is confluent.
\end{lemma}

\begin{proof}
 By checking all overlapping pairs of rules we can see that $S_P$ is strongly confluent. Hence, by \prref{thm:konvtoCR} it is confluent.
\end{proof}

Let $S$ be some rewriting system. A \emph{geodesic} word w.r.t\ $S$ is a word which is shortest among all words in its $\DAS{*}{S}$-class.
\begin{definition}\label{def:geodesic}
A rewriting system $S$ is called \emph{geodesic} if every word can be reduced to a geodesic by only applying length reducing rules (\ie  rules $(\ell,r)\in S$ with $\abs{\ell} > \abs{r}$).
\end{definition}

In particular, if $S$ is geodesic, then every $w$ with $w\DAS*{S} 1$ can be reduced to the empty word by only applying length reducing rules. 
Together with \prref{thm:endl_pregroup}, the next result shows that virtually free group can be defined by geodesic rewriting systems (see also \refPedro, Theorem 2.5).

\begin{proposition}[\cite{ddm10}]\label{prop:sp_geodesic}
The rewriting system $S_P$ is geodesic.
\end{proposition}

\begin{proof}
We start with a sequence $w= a_1 \cdots a_n \in P^*$ to which no length reducing rule can be applied, \ie  there is no $i$ such that $a_ia_{i+1}$ is defined. We have to show that $w$ is a geodesic.
By \prref{lem:sp_konfluent} $S_P$ is confluent, and hence we know that $w$ can be reduced to a geodesic by applying rules in $S_P$. Therefore, it is sufficient to show that after applying some symmetric rule of $S_P$ to $w$, there is still no possibility to apply a length reducing rule. This can be done easily by playing with the axioms (P\ref{P3}) and (P\ref{P4}) of a pregroup and we leave it to the reader.
\end{proof}

\begin{corollary}[\cite{Stallings71}]
 The canonical map $P \to U(P)$ is injective.
\end{corollary}

From \prref{prop:sp_geodesic} it follows that every word representing the identity in $U(P)$ can be reduced to the empty word by only applying length reducing rules of $S_P$. Reading these rules backward we obtain the following corollary.

\begin{corollary}\label{cor:precf}
 Let $\cG = (P,P,R,1_P)$ be the context-free grammar with variables and terminals $P$, productions $R= \set{1_P \to 1, [ab] \to a b}{(a,b)\in D}$ and axiom $1_P$. Then $\cG$ is a context-free grammar for the word problem of $U(P)$.
\end{corollary}

The following proposition has been stated first for length reducing rewriting systems in \cite{die87stacs}.

\begin{proposition}\label{prop:geocf}
Let $G$ be a finitely generated group presented (as a monoid) by some finite geodesic rewriting system.
Then there is a special deterministic push-down automaton which recognizes the word problem of $G$.
\end{proposition}

\begin{proof}
The following proof is taken from \cite{ddm10}. The argument is originally from \cite{die87stacs} and also appears in \cite{GilHHR07}.
Starting with a finite geodesic rewriting system $S$, we describe a push-down automaton for the word problem of $G$. Let $S_R\sse S$ be the length reducing rules of $S$.

Consider a word $w$ and write it as $w=uv$
such that $u$ is geodesic. The prefix $u$ is kept on a push-down stack.
Suppose that $v = a v'$, for some letter $a$.
Push $a$ onto the top of the stack: so the stack
becomes $ua$. There is no reason to suppose that $ua$ is geodesic and, if necessary, 
we perform length reducing reduction steps to produce an equivalent
geodesic word $\widehat u$. Suppose this requires $k$ steps:
$$ua \RAS{k}{S_R} \widehat u.$$
Let us show that we can bound $k$ by some
constant depending on $S$, only. Indeed for all letters $a$ we may fix a
word $w_a$ such that $aw_a \RAS*{S_R} 1$. But this means
$$\widehat u w_a \RAS{*}{S_R} \widetilde u,$$
where $\widetilde u$ is geodesic and $\widetilde u$ represents the same
group element as $u$. Because $u$ was geodesic, it follows $\abs{u}
= \abs{\widetilde u}$. Therefore, $|\widehat u|\ge |u| - |w_a|$ and this
tells us $k \leq |w_a|$. Since $k$ is bounded by some constant, we see
that the whole reduction process involves a bounded suffix
of the word $ua$, only. This means we can factorize
$ua = pq$ and $\widehat u = pr$, where the length of $q$ is bounded
by some constant depending on $S$, only. Moreover,
$q \RAS{k}{S_R} r$. Since the length of $q$ is bounded, this reduction
can be performed using the finite control of the
push-down automaton. The automaton stops once the input has been read and
then the stack gives us a geodesic corresponding to the input word $w$.
In particular, in the end the stack is empty if and only if the input word was equal to $1$ in the group.
\end{proof}

\subsection{Finite Graphs of Groups and Pregroups}

\newcommand{\Pg}{\mathcal{P}}
In the monograph ``Pregroups and Bass-Serre theory'' \cite{Rimlinger87a} Rimlinger showed that any fundamental group of a finite graph of groups with finite vertex groups can be written as universal group of a finite pregroup, and vice versa. This section is dedicated to the proof of one direction of his theorem. Using Britton reductions we are able to simplify the proof considerably.
The other direction of Rimlinger's theorem will then follow by the detour via finite treewidth.

Let $\cG$ be a graph of groups with underlying graph $Y$. Furthermore, let the alphabet $\Sigma$, the group $F(\cG)$ and the rewriting system $B_\cG$ (Britton reductions) be as in \prref{sec:bass_serre}.
We define a subset $\Pg\sse F(\cG)$ as follows (here, $T$ again is a fixed spanning tree of $Y$, $P_0 \in V(Y)$, and $T[P,Q]$ denotes the edge sequence of the geodesic path from $P$ to $Q$ in $T$):
\begin{align*}
\Pg&= \left\{ \,g_0y_1\cdots g_{s-1}y_sg_{s}\in \pi_1(\cG, P_0) \: | \: 
	      \exists\, k:\ y_k \in (E(Y)\setminus E(T)) \cup \{1\},\right.\\ 
    &\left.\qquad\qquad y_1\cdots y_{k-1} = T[P_0, s(y_k)],\ y_{k+1}\cdots y_s = T[t(y_k),P_0] \,\right\}.
\end{align*}

Here, we allow $y_k$ to be the empty word in order to avoid to have to distinguish two cases. If $y_k=1$ we set $s(y_k)= t(y_k)= s(y_{k+1})= t(y_{k-1})$ (for $k=1$ we set $s(y_k)=P_0$).

Note that for $w,u\in\Sig^*$ with $w\RAS{*}{B_\cG}u$, we have $w\in \Pg$ if and only if $u$ is of the form $\,g_0y_1\cdots g_{s-1}y_sg_{s}$ with 
$ y_k \in (E(Y)\setminus E(T)) \cup \{1\},
 y_1\cdots y_{k-1} = T[P_0, s(y_k)],\ y_{k+1}\cdots y_s = T[t(y_k),P_0]$ and $g_i\in G_{s(y_{i+1})}$.
 
 The underlying paths of elements of $\Pg$ may have length zero. If $Y$ is finite, an upper bound is given by $2 \diam(T) + 1$ where $\diam(T)$ denotes the diameter of $T$. Therefore, if the graph $Y$ is finite and all the vertex groups are finite, $\Pg$ is finite, too.
We define the partial multiplication and involution on $\Pg$ by the respective operations in $F(\cG)$, \ie  $(x,y)\in D$ if and only if $xy \in \Pg$.

\begin{theorem}[\cite{Rimlinger87a}]\label{thm:endl_pregroup}
Let $\cG$ be a graph of groups. Then $\Pg$ is a pregroup and
\[U(\Pg) \cong \pi_1(\cG).\] 
Moreover, if $Y$ and all vertex groups are finite, then $\Pg$ is finite.
\end{theorem}

\begin{proof}
The axioms (P\ref{P1})--(P\ref{P3}) hold trivially since $\Pg$ is a subset of a group closed under forming inverses.
We define $\phi(g) = T[P_0,P]gT[P,P_0]$ for $g\in G_P, P\in V(Y)$ and $\phi (y) = T[P_0,s( y)]yT[t( y),P_0]$ for $y \in E(Y)$.
Let $a\in G_y$, then we have as equality in $U(\Pg)$:
\begin{align*}
 \phi(\ov y) \phi(a) \phi(y)
 &= \underbrace{T[P_0,s( \ov y)]\ov y T[t( \ov y),P_0]}_{\in \Pg} \cdot
	 \underbrace{ T[P_0,s( y)]aT[s( y),P_0]}_{\in \Pg} \cdot \,
	 \phi(y)\\ 
 &= \underbrace{T[P_0,t( y)]\ov yaT[s( y),P_0]}_{\in \Pg} \cdot
	 \underbrace{ T[P_0,s( y)] y T[t( y),P_0]}_{\in \Pg}\\
 &=\underbrace{T[P_0,t( y)]\ov yayT[t( y),P_0]}_{\in \Pg} 
 =T[P_0,t( y)]a^{\ov y}T[t( y),P_0]
 =\phi(a^{\ov y})
\end{align*}
Furthermore, $\phi(y) = 1$ for $y \in E(T)$.
Hence, there is a homomorphism $\phi:\pi_1(\cG,T) \to U(\Pg)$ extending this definition. Since $\Pg$ is contained in the image, it is surjective. Reading the elements of $U(\Pg)$ in $\pi_1(\cG,T)$, we obtain an inverse map. Therefore, we have $U(\Pg) \cong \pi_1(\cG,T)$.

It remains to show (P\ref{P4}). In order to do so, we need a preliminary lemma:

\begin{lemma}\label{lem:in_P}
Let $v,w\in \Pg$:
\begin{align*}
v&=g_0x_1 \cdots g_{k-1}x_k g_{k}x_{k+1}\cdots g_{r-1}x_rg_{r} \in \Pg,\\		
w&=h_0y_1 \cdots h_{\ell-1}y_\ell h_{\ell} y_{\ell+1}\cdots h_{s-1}y_s h_{s}\in \Pg,
\end{align*}
with $x_k,y_\ell \in (E(Y)\setminus E(T)) \cup \{1\}$, $x_1\cdots x_{k-1} = T[P_0, s(x_k)]$, $x_{k+1}\cdots x_{r} = T[t(x_k), P_0]$, $y_1\cdots y_{\ell-1} = T[P_0, s(y_\ell)]$, and $y_{\ell+1}\cdots y_s = T[t(y_\ell),P_0]$.
If one of the following reductions are possible, then $vw \in \Pg$.
\begin{enumerate}[(I)]
\item $x_kg_{k} \cdots x_{r}g_{r}h_0y_1\cdots h_{\ell-1}y_\ell \RAS{*}{B_\cG} \tilde{g} \in G_{t(x_{k-1})}=G_{s(y_{\ell+1})}$, \label{I}
\item $x_kg_{k} \cdots x_rg_{r}h_0y_1\cdots h_{\ell-1}y_\ell \RAS{*}{B_\cG} \tilde{g}_{k-1} y_{\ell-j}\cdots h_{\ell-1}y_\ell$ with $\tilde{g}_{k-1} \in G_{t( x_{k-1})}$, $j\geq 0$, \label{II}
\item $x_kg_{k} \cdots x_rg_{r}h_0y_1\cdots h_{\ell-1}y_\ell \RAS{*}{B_\cG} x_kg_{k}\cdots x_{k+j}\tilde{h}_{\ell}$ with $\tilde{h}_{\ell} \in G_{s(y_{\ell+1})}$, $j\geq 0$.\label{III}
\end{enumerate}
Moreover, if $v$ and $w$ are Britton reduced and $vw \in \Pg$, then one of the above reductions is possible.
\end{lemma}

\begin{proof}
Let $v,w$ be Britton reduced. Then every Britton reduction applied to $vw$ has to involve letters of both $v$ and $w$. Hence, if none of the above cancellations are applicable, then
\[vw\RAS{*}{B_\cG} g_0 x_1 \cdots g_{k-1} x_k \,\tilde u \,y_\ell h_{\ell} \cdots y_sh_{s} = u\]
is Britton reduced, where $\tilde u\in \Sigma^*$ is of the form $\tilde u = g_{k} x_{k+1}\tilde u'$ if $x_k=1$, and $\tilde u = \tilde u'' y_{\ell-1}h_{\ell-1}$ if $y_\ell=1$ for properly chosen $\tilde u', \tilde u''\in \Sigma^*$.
However, this means that $u$ does not meet the conditions for $\Pg$ since at least twice the underlying path sets back or uses an edge which is not in $T$.

Now, let (\ref{I}), (\ref{II}) or (\ref{III}) apply. If (\ref{I}), then
\[vw\RAS{*}{B_\cG} g_0 x_1 \cdots g_{k-2}x_{k-1}[ g_{k-1} \tilde{g} h_{\ell}] \,y_{\ell+1}h_{\ell+1} \cdots y_sh_{s}\]
which meets the conditions for $\Pg$ since $ g_{k-1} \tilde{g} h_{\ell} \in G_{t( x_{k-1})}$. 
The case (\ref{III}) is symmetric to (\ref{II}); therefore, we only consider (\ref{II}). We have
\[vw\RAS{*}{B_\cG} g_0x_1\cdots g_{k-2}x_{k-1} [g_{k-1}\tilde{g}_{k-1}] \, y_{\ell-j}h_{\ell-j}\cdots y_\ell h_{\ell} \cdots y_sh_{s}=u\] 
with $j\geq 0$ and $g_{k-1}\tilde{g}_{k-1} \in G_{t( x_{k-1})}$ such that $u$ is Britton reduced. 
By hypothesis the paths $x_1,\dots, x_{k-1}$ as well as $y_{\ell+1},\dots, y_s$ and $y_{\ell-j},\dots, y_{\ell-1}$ are in the spanning tree $T$ and without backtracking. Hence, it remains to show that $x_{k-1}\neq \ov{y}_{\ell-j}$. 

We have $x_k=1$, for, if we had $x_k\in E(Y)\setminus E(T)$, it could only cancel with $y_\ell$ what, by assumption, is not the case. Therefore, $x_{k-1} = \ov{x}_{k+1}$. Furthermore, $\ov{x}_{k+1} = y_{\ell-j-1}$ since these two edges are cancelled by a Britton reduction. Because $y_{1},\dots, y_{\ell-1}$ is without backtracking, it follows that $y_{\ell-j-1} \neq \ov{y}_{\ell-j}$.  Hence, we have
\[x_{k-1} = \ov{x}_{k+1} = y_{\ell-j-1} \neq \ov{y}_{\ell-j}.\]
Therefore, $x_1,\dots, x_{k-1}, \,y_{\ell-j},\dots, y_{\ell-1}$ is a path without backtracking, and hence $vw \in \Pg$.
\end{proof}

We write $v \preccurlyeq w$ if a reduction of type (I) or (II) occurs and $v \succcurlyeq w$ if a reduction of type (I) or (III) occurs. Let $u,v,w,uv,vw \in \Pg$. The following facts are immediate by \prref{lem:in_P}:
\begin{enumerate}
 \item If $u \preccurlyeq v$ and $v \preccurlyeq w$, then $uvw\in \Pg$.
 \item If $u \succcurlyeq v$ and $v \succcurlyeq w$, then $uvw\in \Pg$.
 \item If $u \preccurlyeq v$ and $v \succcurlyeq w$, then $uvw\in \Pg$.
\end{enumerate}
Now, if $u,v,w,x \in \Pg$ are Britton reduced and $uv,vw,wx \in \Pg$, then, by \prref{lem:in_P}, $u,v,w$ or $v,w,x$ at least meet one of these three conditions.
\end{proof}

\begin{example}
Let $\cG$ be a graph of groups over $Y$ with $E(Y)=\smallset{y_1,\dots ,y_m,\, \ov y_1,\dots ,\ov y_m}$ , $V(Y)=\smallset{P}$ (as in \prref{ex:free_group_gog}) such that $\pi_1(\cG,T)\cong F_{\smallset{y_1,\dots ,y_m}}$. Then we obtain the same pregroup as in \prref{ex:prefree} $\Pg=\smallset{1}\cup E(Y)=\oneset{1,y_1,\dots ,y_m,\, \ov y_1,\dots ,\ov y_m}$ as the pregroup constructed according to \prref{thm:endl_pregroup} (for $y\in E(Y)$ the multiplication is defined only with $\ov y$ and with $1$). We have $U(\Pg)\cong F_{\smallset{y_1,\dots ,y_m}}$.
\end{example}

\begin{example}
Let $\cG$ be the graph of groups over $Y$ with $E(Y)=\smallset{y,\ov y}$, $V(Y)=\smallset{P}$ and $G_v=G_y=G_{\ov y} =\Z/2\Z=\genr{a}{a^2=1}$ (the incidences and inclusions are defined the obvious way). Then the fundamental group is $\pi_1(\cG,T)\cong \Z \times \Z/2\Z$ and
\[\Pg= \oneset{1,a,y,\ov y,ay,a\ov y}.\]
The partial multiplication is as the following table shows (the order of the operands does not matter since the group is abelian):
\begin{center}
\begin{tabular}{c||c|c|c|c|c|c}
$\cdot$ & $1$ & $a$& $y$ & $\ov y$& $ay$& $a\ov y$ \\\hline\hline
$1$ & $1$ & $a$& $y$ & $\ov y$& $ay$& $a\ov y$ \\\hline
$a$ & $a$ & $1$& $ay$ & $a\ov y$& $y$& $\ov y$\\\hline
$y$ & $y$ & $ay$& - & $1$& - & $a$ \\\hline
$\ov y$ & $\ov y$ & $a\ov y$& $1$ & - & $a$ & -\\\hline
$ay$ & $ay$ & $y$& - & $a$&-& $1$\\\hline
$a\ov y$ & $a\ov y$ & $\ov y$& $a$ & - & $1$ & -
\end{tabular}
\end{center}
\end{example}

\section{Graphs and Treewidth}\label{sec:treewidth}

We introduce the concept of tree decompositions of graphs.
Tree decompositions were used by Robertson and Seymour in connection with their famous result on graph minors, 
\cite{RobertsonS04}. 

Throughout this section $\Gamma= (V,E)$ denotes an undirected simple graph
which is nonempty and connected. The restriction to simple graphs is principally for simplicity. Most proofs could be rewritten allowing multi-edges and loops. In the remainder of these notes we always assume that $\GG$ is simple even if we do not state it explicitly.

\begin{definition}[Tree Decomposition]\label{def:tree_decomposition}
A \emph{tree decomposition} of $\Gamma$
is a tree $T=(V(T),E(T))$ together with a mapping $t \mapsto X_t$ where each 
$X_t$ is a finite subset of $V$ such that the following conditions are satisfied.
\begin{enumerate}[(T1)]
\item For every node $v\in V$ there is some $t\in V(T)$ such that $v\in X_t$, \ie  $V= \bigcup_{t\in V(T)} X_t$. 
\item For every edge $e= uv\in E$ there is some $t\in V(T)$ such that $u, v \in X_t$.
For simplicity, we say that $X_t$ contains the edge $e$.
\item If $v \in X_t \cap X_s$, then we have $v \in X_r$ for all vertices $r$ 
of the tree which are on the geodesic from $s$ to $t$, \ie  the set $\set{ t\in V(T)}{ v \in X_t}$ forms a subtree of $T$.
\end{enumerate}
\end{definition}

By abuse of language we denote a tree decomposition simply with its associated tree $T$. The sets $X_t$ are called \emph{bags} in the following.
For a tree decomposition $T$ of $\GG$ we define the \emph{bag-size} $\bags(T)$ by 
$$\bags(T) = \sup\set{\abs{X_t}}{t \in V(T)};$$
\ie  the bag-size is the least $k\in \N\cup \oneset{\infty}$ such that $\abs{X_t} \leq k$ for all $t \in V(T)$.

\begin{definition}\label{def:fintreewidth}
We say that $\GG$ has \emph{finite treewidth}, if there is some 
tree decomposition $T$ with bag-size $\bags(T) < \infty$.
\end{definition}

A tree has a tree decomposition with bag-size $2$; therefore, Robertson and Seymour
defined (somewhat unfortunately) the \emph{treewidth} by the minimal value $\bags(T)-1$ over all tree decompositions. 

Later we will show that context-free groups have Cayley graphs of finite treewidth. First, we list list some basic properties of tree decompositions.

\begin{lemma}\label{lem:subgraph}
Let $\GG$ be a graph and $\GG'$ a subgraph of $\GG$. If $\GG$ has treewidth $k$, then $\GG'$ has treewidth at most $k$.
\end{lemma}
\begin{proof}
 Every tree decomposition of $\GG$ yields a tree decomposition of $\GG'$ by restricting the bags to the vertices of $\GG'$.
\end{proof}

\begin{proposition}\label{prop:basictreecut}
 Let $T=(V(T),E(T))$ be a tree decomposition of $\Gamma =(V,E)$.
 \begin{enumerate}
  \item \label{bas0} Let $X$, $Y$, $Z$ be bags and $Z$ be in the tree $T$ on a geodesic from bag $X$ to bag $Y$.
Let $x \in X$ and $y\in Y$ and 
$x=x_0, \ldots, x_n=y$ be any path in $\GG$ connecting $x$ and $y$. 
Then we have $x_i \in Z$ for some $0\leq i \leq n$. 
\item \label{basiii}
If two bags $X\neq\emptyset$ and $Y\neq\emptyset$ are connected by some edge in the tree $E(T)$, then 
$X \cap Y \neq \es$.
 \end{enumerate}
\end{proposition}

\begin{proof}
\ref{bas0} The result is clear for $n=0$. For $n>0$ consider the 
bag $X'$ which contains $x_0$ and $x_1$. If $Z$ is on the geodesic from $X$ to 
$X'$, then $Z$ contains $x_0$. Otherwise, $Z$ is on the geodesic from $X'$ to $Y$, and we are done by induction. 

\ref{basiii} Let $x \in X_s$, $y\in X_t$ such that $st \in E(T)$. Recall that we assumed $\GG$ to be connected.
Let $x=x_0, \ldots, x_n=y$ be some path in $\GG$ connecting $x$ and $y$
and $i = \max\set{i}{x_i \in X_s}$. If 
$i = n$, then $y \in X_s \cap X_t$. Thus, we may assume $0 \leq i < n$. 
Consider the bag $X'$ which contains $x_i$ and $x_{i+1}$. Due to the choice of $i$, the bag $X_s$ cannot be on the geodesic from $X'$ to $X_t$. Thus, $X_t$ is on the geodesic
from $X_s$ to $X'$. It follows $x_i \in X_s \cap X_t$.
\end{proof}

\begin{proposition}\label{prop:basictree}
Let $k \in \N$ and $\Gamma =(V,E)$ 
be a locally finite, connected simple graph having a tree decomposition of bag-size $k$. Then there 
is a tree decomposition $T=(V(T),E(T))$ of bag-size $k$ satisfying the following conditions. 
\begin{enumerate}
\item\label{basi} Each vertex $u\in V$ occurs in finitely many bags, only. 
\item \label{basiia}
We have $X_t \neq \es$ for all $t \in V(T)$.
\item\label{basiv} The tree $T$ is locally finite. 
\item \label{basv}If $X_s \sse X_t$, then $s=t$. Thus, we can identify nodes in the tree with 
 finite non-empty subsets of $V$, which are pairwise incomparable w.\,r.\,t.\ inclusion.
\end{enumerate}
\end{proposition}

\begin{proof}
We start with a tree decomposition $T=(V(T),E(T))$ of bag-size $k$. We transform $T$ 
in consecutive steps into a tree decomposition meeting the desired conditions. 

\ref{basi} For every vertex $u\in V$ we fix some vertex $t_{u} \in V(T)$ with $u \in X_{t_u}$, and the same for edges. 
For every edge $uv \in E$ we fix some vertex $t_{uv} \in V(T)$ with $ \smallset{uv} \sse X_{t_{uv}}$. Now, for each vertex $u$ let $T_u$ be the finite subtree spanned by $t_{u}$ and the $t_{uv}$ for $uv \in E$. It is finite because $\Gamma$ is locally finite.
Remove $u$ from all bags which do not belong to $T_u$. This yields still a tree decomposition and $u$ appears in finitely many bags, only. 

\ref{basiia} 
Since  $\GG$ is connected, \prref{prop:basictreecut} \ref{bas0} implies that the non-empty bags form a connected subtree of $T$. Hence, removing all empty bags still yields a tree decomposition. 

\ref{basiv}
Let $X$ be some bag. By \prref{prop:basictree} \ref{basiii} each neighbor bag of $X$
shares at least one element with $X$.
But every vertex is contained in only finitely many bags. Hence, the result follows.

\ref{basv} We perform the transformation in two phases without destroying the three properties 
above. In the first phase we make bags larger in order to achieve that either $X_s$ and $X_t$ are incomparable w.\,r.\,t.\ inclusion or
$X_s= X_t$. First we choose a vertex $r \in V(T)$ as a root. 
In case that there are bags $X$ and $Y$ with $X\sse Y$ but $X\neq Y$, 
there is also an edge $st \in E(T)$ with $X=X_s\sse X_t$ but $X_s\neq X_t$.
We choose such an $s$ of minimal distance to the root $r$ and we replace the bag $X_s$ by $X_t$.
This does not change the tree structure and every bag is replaced by some other bag only finitely many times. Thus, in the limit we obtain a well-defined tree decomposition such that 
$X_s\sse X_t$ implies $X_s = X_t$ for all $s,t\in V(T)$. The bag-size does not increase because no new bags are introduced. Moreover, still every $x\in V$ appears in finitely many bags of the new tree only. This finishes the first phase. 

Now, in the second phase we contract subtrees. Since $X_s\sse X_t$ implies $X_s = X_t$ for all $s,t\in V(T)$, we see that for each bag $X$, the set of vertices 
$s \in V(T)$ with $X_s = X$ is forms indeed a finite subtree. This gives a partition 
of $T$ into finite subtrees. We contract each such subtree into a single point and 
we obtain the desired result. 
\end{proof}

For $U \sse V(\GG)$ let $N(U) = U \cup \set{v\in V}{\exists u \in U: uv \in E}$ denote the \emph{neighborhood} of $U$. 
Furthermore, we write $N^\ell$ for the $\ell$-th neighborhood, \ie  $N^0(U) = U$ and $N^\ell(U) = N(N^{\ell-1}(U))$ for $\ell \geq 1$.

\begin{lemma}\label{lem:umgeb}
Let $T=(V(T),E(T))$ be a tree decomposition of $\GG$. If we replace all bags $X$ by
$N(X)$, we still have a tree decomposition of $\GG$.
\end{lemma}

\begin{proof}
We have to show that (T3) still holds. Let $x\in X$ and $y \in Y$ for bags $X$ and $Y$ such that $x$ and $y$ have a common neighbor $z$. Thus, $z \in N(X) \cap N(Y)$. Let $Z$ be on the geodesic from $X$ to $Y$. 
 We have to show that $z \in N(Z)$. Now, $(x,z,y)$ is a path connecting
 $x$ and $y$, hence we have $\os{x,z,y} \cap Z \neq \es$ by \prref{prop:basictreecut}. Hence, the result follows. 
\end{proof}

A \emph{clique} is a complete subgraph, \ie  a subgraph such that for every pair of vertices there is some edge connecting the two vertices. It is called \emph{maximal} if it is not contained in any other clique.

\begin{proposition}\label{prop:maxclique}
Let $T$ be a tree decomposition of a locally finite graph $\GG=(V,E)$. Then every clique 
of $\GG$ is contained in some bag $X$. 
\end{proposition}

\begin{proof}
 Let $C= \os{1 \lds n}$ be a clique and $n \in \N$. We have to show that
 $C \sse X$ for some bag. This is clear for $n\leq 1$. Therefore, let $n \geq 2$.
 There are bags $X_i$ containing $x_i$ and $x_n$ for all $1 \leq i <n$. 
 By induction there is a bag $X_n$ containing the clique $\os{1 \lds n-1}$. 
 We choose $X_n$ such that the distance to $X_1$ is minimal. If $X_1 = X_n$ we have 
 $x_n \in X_n$ and we are done. Hence, we may assume that 
 $X_1 \neq X_n$. We can think of $X_n$ as the root of the tree decomposition. 
 If for some $i \neq j$ the bags $X_i$ and and $X_j$ are in different subtrees of 
 the root, then $x_n \in X_n$ and we are done again. Thus, by contradiction, there 
 is a child $X'$ of $X_n$ such that 
 $X_1 \lds X_{n-1}$ are in the same subtree below $X'$. But then $\os{1 \lds n-1} \sse X'$, but the distance of $X'$ to $X_1$ is shorter. This is impossible. Thus, 
 indeed $x_n \in X_n$.
\end{proof}

A simple graph is called \emph{chordal} if it does not have an induced cycle of length $\geq 4$.
The next result extends the well-known characterization of finite chordal graphs
to locally finite graphs. The result for finite graphs can be found e.\,g.\ in
\cite[Prop.\ 10.3.10]{diestel06}. For infinite 
 graphs we use a limit process. 
\begin{proposition}[\,{\!\cite[Thm.\ 6.3.8]{diestel90}}]\label{prop:chordal_maxclique}
Let $\GG$ be a connected, locally finite, simple graph. Then the following assertions are equivalent. 
\begin{enumerate}
\item\label{chordal_maxi} $\GG$ is chordal.
\item\label{chordal_maxii} $\GG$ has a tree decomposition, where all bags are cliques.
\item\label{chordal_maxiii} $\GG$ has a tree decomposition, where the bags correspond to maximal cliques
and cliques which are adjacent in the tree intersect non-trivially.
\end{enumerate}
\end{proposition}
Note, that in general \prref{prop:chordal_maxclique} does not hold for graphs which are not locally finite. Diestel \cite{diestel90} gave an example for a graph which is chordal, but does not have a tree decomposition into cliques.
\begin{proof}
In order to show that \ref{chordal_maxi} implies \ref{chordal_maxii}, we first show by induction that \ref{chordal_maxi} implies \ref{chordal_maxiii} for finite graphs.
A well-known characterization says that a finite graph is chordal 
\IFF it has a ``perfect elimination ordering'', \cite{golum80}. This is 
an ordering of the vertices $v_1, v_2, \ldots, v_n$ such that the neighbors of 
$v_i$ which belong to $\os{v_{i+1} \lds v_n}$ form a clique. The induced subgraph of 
$ v_2, \ldots, v_n$ is connected and, by induction, it has a tree decomposition $T$, where the bags correspond to maximal cliques
and cliques which are adjacent in the tree intersect non-trivially.
The neighbors of $v_1$ together with $v_1$ form a maximal clique 
in the original graph. If the neighbors of $v_1$ form a maximal clique in the graph induced by $ v_2, \ldots, v_n$, we add $v_1$ to the corresponding bag. Otherwise, we can attach a new bag to $T$ with an edge to one of the existing bags which contains the neighbors of $v_1$. Thus, \ref{chordal_maxi} implies \ref{chordal_maxiii} for finite graphs. 

For infinite graphs we 
write $\GG$ as a union $\GG = \bigcup\set{\GG_n}{n \in \N}$ such that 
each $\GG_n $ is a finite, connected induced subgraph of $\GG$, the graph 
$\GG_0$ is a maximal clique, and $\GG_n \sse \GG_{n+1}$ for all $n$. 

We can define a map from the finite tree decompositions of $\GG_{n+1}$ to the finite tree decompositions of $\GG_n$ by restricting the bags to $\GG_n$ and deleting empty bags (since the $\GG_n$ are connected, the non-empty bags form one connected component in the tree).
This gives rise to an infinite directed forest with these maps as edges. The vertices are the tree decompositions which are obtained when starting with a tree decomposition into maximal cliques of some $\GG_n$.
The roots of the connected components are tree decompositions into cliques of the maximal clique $\GG_0$~-- note that here the edges are directed towards the roots. At distance $n$ from a root are the tree decomposition into cliques of $\GG_n$.
Since every vertex is contained in finitely many maximal cliques, only, the forest has finite degree and finitely many roots. By K\"onigs' lemma there is an infinite path. 
This path defines a tree decomposition
into cliques of $\GG$. Thus, we have \ref{chordal_maxii}.
(Note that we do not claim that the path yields a tree decomposition of $\GG$ into maximal cliques since the infinite path may avoid such tree decompositions.)

The result that \ref{chordal_maxii} implies \ref{chordal_maxiii}
follows directly from \prref{prop:basictree}, \prref{prop:maxclique} and \prref{prop:basictreecut} (ii). (Recall that the proof of \prref{prop:basictree} 
involved another limit process.) 

As the implication from \ref{chordal_maxiii} to \ref{chordal_maxii} is trivial, it remains to show that \ref{chordal_maxii} implies \ref{chordal_maxi}.
This can be done exactly as in the case of finite graphs. 
Assume that there is some tree decomposition
into cliques and consider a cycle of length at least $4$. Then there are two different edges $uv$ and $wx$ with $\os {u,v} \cap \os {w,x} = \es$ and simple disjoint 
paths from $u$ to $w$ and from $v$ to $x$. A simple inspection using \prref{prop:basictreecut} shows that there must be some bag which contains vertices $y,z$ with $y\neq u$, $z\neq x$ such that the vertex $y$ is on the path from $u$ to $w$, 
and $z$ is on the path from $v$ to $x$. Since every bag is a clique, we see that $yz$ is a chord. 
\end{proof}

\begin{corollary}\label{cor:chordtw}
Let $k \in \N$ and $\GG$ be a connected, locally finite, chordal graph such that 
the maximal size of a clique in $\GG$ is $k$. Then $\GG$ has finite treewidth 
$k-1$. 
\end{corollary}

\begin{proof}
 \prref{prop:chordal_maxclique} shows that the treewidth of $\GG$ is at most
$k-1$. By \prref{prop:maxclique} it is at least $k-1$.
\end{proof}

\subsection{Cayley Graphs}

Let $G$ be a group with $1$ as neutral element. Let $\Sigma \sse G$ be a generating set of $G$. For convenience we assume $1 \not\in \Sigma$. The \emph{Cayley graph} $\Gamma= \Gamma_\Sig(G)$ of $G$ (with respect to $\Sigma$) is defined by $V(\Gamma) = G$ and 
$E(\Gamma) = G \times (\Sigma\cup \Sigma^{-1})$, with the incidence functions 
$s(g,a) = g$, $t(g,a) = ga$, and involution $\ov{(g,a)} = (ga,a^{-1})$. 
For an edge $(g,a)$ we call $a$ the label of $(g,a)$ and extend this definition to paths. Thus, the label of a path is a sequence 
(or \emph{word}) in the free monoid $\Sigma^*$. The Cayley graph is a simple graph (without loops and without multi-edges). 
 It is connected because $\Sigma$ generates $G$. The Cayley graph 
$\Gamma$ is locally finite \IFF $\Sigma$ is finite. In the following we always assume that $\Sigma$ is finite.
Sometimes we suppress $\Sigma$ if there is a standard choice for the generating set. 
For example, if
$G= F_{\Sigma}$ is the free group over $\Sigma$, then the Cayley graph of $G$ refers 
to $\Sigma$ and it is a tree. Similarly, by the \emph{infinite grid} we mean the Cayley graph of $\Z \times \Z$ with generators $(1,0)$ and $(0,1)$.

\begin{example}\label{ex:nixZZ}
Let $\GG$ be the infinite grid $\Z \times \Z$, \ie  the vertices are pairs $(i,j)$, $i,j \in \Z$ and there are edges from $(i,j)$ to $(i,j \pm 1)$, $(i\pm 1,j)$.
Then $\GG$ does not have finite treewidth.

\begin{figure}[ht]

\begin{center}
\begin{tikzpicture}[scale = 0.7]
\foreach \x in {-3,-1.5,...,4}{
\draw[thin](\x-2,-2) -- (\x+2,2);
	\foreach \y in {-2,...,2}{
		\draw[thin](\x+\y-0.33,\y-0.33) -- (\x+\y,\y);
};
};

\foreach \y in {-1.5,...,1.5}{
	\draw[thin](-4+\y,\y) -- (4+\y,\y);
	\foreach \x in {-3.75,-2.25,...,4.75}{
		\draw[thin](\x-0.5+\y,\y) -- (\x+\y,\y);
};
};
\node at (-4.5,0) {$\cdots$};
\node at (4.5,0) {$\cdots$};
\end{tikzpicture}
\end{center}
 \caption{The infinite grid does not have a tree decomposition of finite bag-size.}\label{fig:grid}
\end{figure}
\end{example}
\begin{example}\label{ex:nixZ2Z3}
Let $\GG$ be the Cayley graph of the modular group $\mathrm{PSL}(2,\Z)\cong \Z/2\Z \star \Z/3\Z$. 
Then $\GG$ has a tree decomposition of finite width where bags are the triangles and the bridges between triangles. Thus, there is a tree decomposition with bag-size $3$. Therefore, the treewidth is $2$. It cannot be $1$  because $\GG$ is not a tree. 
\begin{figure}[ht]
\begin{center}
\tikzstyle{every node}=[inner sep=0pt]
\begin{small}
\begin{tikzpicture}[scale=.5, transform shape]
\node (o)at (0,0) {};
\newcommand{\dist}{1.3};
\newcommand{\factor}{0.61}
\tikzstyle{every node}=[fill,circle,inner sep=0pt]
\foreach \A/\lab in {0/1,120/b,240/bb}{
  
  	\node (1) at (\A+0:1) {};
  	\node (a) at (\A+120:1) {};
   	\path (1) -- ++(\A+291:0.24) node[fill=none] {};
	\draw (1)--(a);
	\path[draw] (1)-- ++(\A+0:1.3) node (t){};	
	\ifthenelse{\equal{\lab}{1}}
            {\path (t) -- ++(\A+255:0.24) node[fill=none] {};}
           {\path (t) -- ++(\A+265:0.3) node[fill=none] { };}

	\path (t)-- +(\A+330:\dist) node (ta){}-- +(\A+30:\dist)node (tb){}--(t){};
	\draw (ta)-- (tb);

	\foreach \B/\k in {\A+60/tb,\A+300/ta}{
	
	  	\draw (t)-- (\k)  node (ia){}-- ++ (\B+0:\dist*\factor)node(it){};
	  	\path (it)-- +(\B+330:\dist*\factor) node (ita){}-- +(\B+30:\dist*\factor)node (itb){}--(t){};
		\draw (ita)-- (itb);
			
		\foreach \C/\l in {\B+60/itb,\B+300/ita}{

			\draw (it)-- (\l)  node (iia){}-- ++ (\C+0:\dist*\factor*\factor)node(iit){};
		  	\path (iit)-- +(\C+330:\dist*\factor*\factor) node (iita){}-- +(\C+30:\dist*\factor*\factor)node (iitb){}--(t){};
			\draw (iita)-- (iitb);
			
			\foreach \D/\l in {\C+60/iitb,\C+300/iita}{
			
				\draw (iit)-- (\l)  node (iiia){}-- ++ (\D+0:\dist*\factor*\factor*\factor)node(iiit){};
				\path (iiit)-- +(\D+330:\dist*\factor*\factor*\factor) node (iiita){}-- +(\D+30:\dist*\factor*\factor*\factor)node (iiitb){}--(t){};
				\draw (iiita)-- (iiitb);
				
				\foreach \E/\l in {\D+60/iiitb,\D+300/iiita}{
				
					\draw (iiit)-- (\l)  node (iiiia){}-- ++ (\E+0:\dist*\factor*\factor*\factor*\factor)node(iiiit){};
					\path (iiiit)-- +(\E+330:\dist*\factor*\factor*\factor*\factor) node (iiiita){}-- +(\E+30:\dist*\factor*\factor*\factor*\factor)node (iiiitb){}--(t){};
					\draw (iiiita)-- (iiiitb);

					\foreach \F/\l in {\E+60/iiiitb,\E+300/iiiita}{
						
						\draw (iiiit)-- (\l)  node (iiiiia){}-- ++ (\F+0:\dist*\factor*\factor*\factor*\factor*\factor)node(iiiiit){};
						\path (iiiiit)-- +(\F+330:\dist*\factor*\factor*\factor*\factor*\factor) node (iiiiita){}-- +(\F+30:\dist*\factor*\factor*\factor*\factor*\factor)node (iiiiitb){}--(t){};
						\draw(iiiiit)-- (iiiiita)-- (iiiiitb)--(iiiiit);
						\foreach \G/\l in {\F+60/iiiiitb,\F+300/iiiiita}{
							\draw  (\l)  -- ++ (\G+0:\dist*\factor*\factor*\factor*\factor*\factor*\factor)node {}-- ++ (\G+330:\dist*\factor*\factor*\factor*\factor*\factor*\factor)node {}-- ++ (\G+90:\dist*\factor*\factor*\factor*\factor*\factor*\factor)node {}-- ++ (\G+210:\dist*\factor*\factor*\factor*\factor*\factor*\factor)node {};
						};				
					};	
				};	
			};		
		};
	};	
};
\end{tikzpicture}
%
%
%
%
%
%
 \end{small}
\caption{The Cayley graph of $\mathrm{PSL}(2,\Z)$ 
has treewidth $2$.}\end{center}
 \end{figure}
\end{example}

The following result is due to Muller and Schupp \cite{ms83, ms85}. We rephrase it in the terminology 
of treewidth. In this terminology it was first stated by Kuske and Lohrey in their work about monadic second order logic of Cayley graphs \cite{KuskeL05}. The proof is taken from \cite{diekertW13}, but it follows the original proof in \cite{ms83, ms85}.
For the proof we need the following definition: For a subset $C\sse V(\GG)$ of vertices of some graph $\GG$ we define the \emph{vertex-boundary} of $C$ as
$\beta C = \set{ u,v\in V(\Gamma)}{\exists\, uv\in E(\Gamma) \text{ with } u\in C, v\in \Comp{C} \text{ or } u\in \Comp{C}, v\in C}$.
\begin{theorem}[\cite{KuskeL05, ms83, ms85}]\label{thm:cfftw}
Let $\Gamma$ be a Cayley graph of a context-free group $G$
with respect to a finite generating set $\Sigma$. Then $\Gamma$ has finite treewidth. 
\end{theorem}

\begin{proof}If $G$ is finite, then the assertion is trivial. Therefore, let $G$ be infinite. 
We may assume that $1 \notin \Sigma\sse G$.

The vertex set of $\GG = \GG_\Sig(G)$ is the group $G = V(\Gamma)$, by $B_n$ we denote the ball of radius 
$n$ around the origin $1 \in G$, \ie  $B_n= \set{g \in V(\Gamma)}{d(1,g)\leq n}$. 
We are heading for a tree decomposition where certain finite subsets 
of $G$ become nodes in the tree. 
For $n \in \N$ we define sets $V_n$ of level $n$ such that $V_0 = V(\GG-1)$ and 
$V_n = \set{C\sse V(\Gamma) }{ C \text{ is a connected component of } \Gamma - B_{n}}$ for $n \geq 1$. This defines a tree $T$ with root $B_1$ as follows: 
\begin{align*}
V(T) &= \set{ \beta C }{ C \in V_n\,, n \in \N}, \\
 E(T) &= \set{\smallset{ \beta C, \beta D}}{D \sse C \in V_n\,, D \in V_{n+1}, \, n \in \N}.
 \end{align*}
The nodes are subsets of $G$, hence we can identify nodes $t \in T$ with their bags $X_t \sse G$. 
If $(g,a)$ is an edge in the Cayley graph $\GG$, then 
there are essentially two cases; either $d(1,g) = n $ and $d(1,ga) = n + 1$ 
or $d(1,g) = d(1,ga) = n + 1$ for some $n$. In both cases the elements $g$, $ga$ are in some bag $\bet C$ for some $C \in V_n$ and $n \in \N$. 

It remains to show that $\abs{\bet C}$ is bounded by some constant for all $C \in V_n$, $n\in\N$. It is here where
the context-freeness comes into the play. 
We denote 
$\widetilde\Sig= \Sigma \cup \Sigma^{-1}$. This is a set of monoid generators of $G$. We let 
$L_G = \set{w \in \widetilde\Sig^*}{w = 1 \in G}$ its associated group language. 
By hypothesis, $L_G$ is generated by some context-free grammar $(V,\widetilde\Sig, P,S)$, and we 
may assume that it is in Chomsky normal form. Recall that this means all rules are either of the form $A \to BC$ with $A,B,C \in V$ or of the form $A\to a$ with $A\in V$ and 
$a \in \widetilde\Sig$. 
We define a constant $k \in \N$ such that 
$$k \geq \max_{A \in V}\min\set{ \abs{w}}{ A \RAS*{P} w \in \widetilde\Sig^*}.$$ 

Consider $C\in V_n $ and $n \in \N$.
Let $g,h\in \bet C$. We are going to show that $d(g,h)\leq 3k$.
For $n = 0$ we have $\bet C = B_1$. Hence, we may assume $n \geq 1$. 

Let $\alpha$ be a geodesic path from $1$ to $g$ with label $u \in \widetilde\Sig^*$, $\gamma$ a geodesic path from $h$ to $1$ with label $w\in \widetilde\Sig^*$, and $\beta$ some path from $g$ to $h$ with label $v \in \widetilde\Sig^*$ which is entirely contained in $C$. Such a path exists since $C$ is connected.
The composition of these paths forms a closed path $\alpha \beta\gamma$ with label 
$uvw$. We have $uvw \in L_G$ and there is a derivation $S\RAS{*}{}uvw$. 
We may assume that $\abs v \geq 2$ because otherwise there is nothing to do. 

Since the grammar is in Chomsky normal form, we can find a rule $A \to BC$ and derivations as follows: 
$$S \RAS*{P} u'Aw' \RAS{}{P} u'BCw'\RAS*{P} u'v'v''w' = uvw $$
such that $B \RAS*{P} v'$, $C \RAS*{P} v''$, and $\abs{u'} \leq \abs{u} < \abs{u'v'} < \abs{uv } \leq\abs{u'v'v''}$. 

This yields three nodes $x\in \alpha$, $y \in \beta$, and $z\in \gamma$ such that $d(x,y)$, $d(y,z)$, $d(x,z)\leq k$, see \prref{fig:diamC}. (These three nodes correspond exactly to a triangle with endpoints $x,y,z$ in the $k$-triangulation of the closed path $\alpha \beta\gamma$ in \cite{ms83}.)

\begin{figure}[ht]
\begin{center}
\def\nodedist{5}
\begin{tikzpicture}
\node[] (OO) at (-5, 0) {};
\node[left of=OO,node distance=\nodedist] () {$1$};

\node[] (g) at (2, 2) {};
\node[] (h) at (2, -2) {};
\node[above of=g,node distance=\nodedist] () {$g$};
\node[below of=h,node distance=\nodedist] () {$h$};

\draw (2,-2) arc(-60:60:2.309401076758503058);
\node (vv) at (2+1.154700538379251529,0){};
\node[right of=vv,node distance=\nodedist] () {$\beta$};

\draw (-5, 0) ..controls(-3.75,0.12).. (-3,0.2)
	..controls(-1.5,0.36) and (0,1)..
node[above] (){$\alpha$} (1, 1.46)
..controls(1.5,1.69)..
(2, 2){};

\draw (-5, 0) ..controls(-3.75,0.12).. (-3,0.2)
	..controls(-2.25,0.28) and (-1.,0.6) ..
node[below] (){$\gamma$} (1, -1)
..controls(1.5,-1.4)..
(2, -2){};

\node[] (x) at (1, 1.46) {};
\node[] (z) at (1, -1) {};
\path[] (2-1.154700538379251529,0) -- ++(20:2.309401076758503058) node (y){};
\node[above of=x,node distance=\nodedist] () {$x$};
\node[below of=z,node distance=\nodedist] () {$z$};
\node[right of=y,node distance=\nodedist] () {$y$};

\draw(1, 1.46)--node[left](){$A$} (1, -1);
\draw (2-1.154700538379251529,0) + (20:2.309401076758503058)--node[above](){$C$} (1, -1);
\draw(2-1.154700538379251529,0) + (20:2.309401076758503058)--node[below](){$B$} (1, 1.46);

\end{tikzpicture}
\end{center}
\caption[]{The distance between $g$ and $h$ is bounded by $3k$.}\label{fig:diamC}
\end{figure}

Now we have:
\[d(x,g)=d(1,g)-d(1,x) \leq d(1,y)-d(1,x)\leq d(x,y).\]
The first equality holds because $\alpha$ is geodesic and $x$ lies on $\alpha$; the second one because $ d(1,g) \leq n+1 \leq d(1,y)$. Likewise we obtain $d(z,h)\leq d(z,y)$.
Thus, it follows
\begin{align*}
d(g,h)&\leq d(g,x)+d(x,z)+ d(z,h)\\
&\leq d(y,x)+d(x,z)+d(z,y) \leq 3k.
\end{align*}
This implies that the size of the bags is uniformly bounded by some constant since $\Gamma$ has uniformly bounded degree.
\end{proof}

In \cite {AutebertBS87} the following theorem has been stated in terms of \emph{locally primary} groups. A group is locally primary if every generating set can be extended such that the Cayley graph is chordal.

\begin{theorem}[\,{\!\cite[Thm.\ III.3]{AutebertBS87}}]\label{thm:chordal_cayley} 
Let $\cG$ be a finite graph of groups with underlying graph $Y$ and $T$ a spanning tree of $Y$. Every finite generating system of $\pi_1(\cG,T)$ can be extended such that the Cayley graph is chordal.
\end{theorem}

 \begin{remark}\label{rem:chordal_cayley}
 If a finitely generated group has a chordal Cayley graph $\GG$, then $\GG$ has finite treewidth. This follows by \prref{prop:chordal_maxclique} and the fact that the maximal size of a clique is bounded by the cardinality of the generating system plus one.
 \end{remark}

\begin{proof}
First, we show that for the standard generating system $\Sigma = E(Y)\cup \bigcup G_P$ as in \prref{sec:bass_serre} there is a tree decomposition of the Cayley graph of $G=\pi_1(\cG,T)$ with the barycentric subdivision of the \BST $\wt X$ as underlying tree $\widetilde T$.
That means we have $V(\widetilde T) = V(\wt X) \cup E(\wt X)/\{\,e = \ov e\,|\,e\in E(\wt X)\,\}$ (see \prref{sec:prelims_graphs}).

 To each $gG_P\cdot P \in V(\wt X)$ we associate the bag $gG_P \sse \pi_1(\cG,T)$ and 
 to each $gG_y ^y \cdot y \in E(\wt X)$ we associate the bag $gG_y^y \cup gG_y^y y$.
 This is well-defined because $G_y^y \cup G_y^y y = y G_y^{\ov y} \cup y G_y^{\ov y} {\ov y}.$
 
 Note that each $g \in \pi_1(\cG,T)$ appears in $gG_P$. Moreover, if $(g, a)$ is an edge (connecting $g$ and $ga$) in the Cayley graph with 
 generator $a \in G_P$, then the edge is contained in the bag $gG_P$. If $(g, y)$ is an edge in the Cayley graph with generator
 $y \in E(Y) $, then the edge is contained in the bag $gG_y^y \cup gG_y^y y$.
 
 Now, let $g \in g_1G_P \cap g_2G_Q$. We have to show that there is a path from $g_1G_P$ to $g_2G_Q$ in $\wt T$ such that $g$ is contained in all bags on that path. We do not need to consider bags of the type $gG_y^y \cup gG_y^y y$ since in that case we have $g \in gG_{s(y)}$. Without loss of generality we may assume that $g = g_1 = g_2$.
 We take the geodesic path $T[P,Q] = P_0, \dots, P_k$ in the spanning tree $T$ from $P$ to $Q$. This path lifts to a path $gG_{P_0}\cdot P_0, \dots,gG_{P_k}\cdot P_k$ from $gG_P\cdot P$ to $gG_Q\cdot Q$ in the \BST. Since all edges of $T$ are equal to $1$ in $\pi_1(\cG,T)$, we know that $g$ is contained in every bag on this path.

 Up to now, we have constructed a tree decomposition of the Cayley graph of $G$ such that $G$ acts on it (not only on the tree, but also on the bags). By \prref{lem:umgeb} we can extend the bags to their neighborhoods, and hence may assume that all generators of our non-standard generating system are contained in a bag. By doing this the action of $G$ is not changed.
 
 As a last step, we choose a system of representatives of the bags and for each of these bags we add generators (\ie  edges in the Cayley graph) so that these bags become cliques. Note that these are only finitely many. Because of the action of $G$, this means that all bags of the tree decomposition become cliques, and hence by \prref{prop:chordal_maxclique} we obtain the desired result.
 \end{proof}

 \subsection{Quasi-Isometries and Treewidth}

\begin{definition}[Quasi-Isometry]
 Let $\GG = (V,E)$ and $\GG' = (V',E')$ be two graphs. The distance in both graphs is denoted with $d$. A \emph{quasi-isometry} between graphs $\GG$ and $\GG'$ is a function $f: V\to V'$ satisfying the following properties:
 \begin{enumerate}
 \item There is some constant $k$ such that for every $v'\in V'$ there is some $v\in V$ with $d(v', f(v))\leq k$.
 \item There is some constant $k$ such that for all $u,v \in V$ the following inequalities hold:
 \[\frac{1}{k}\cdot d(u,v)-k \leq d(f(u),f(v)) \leq k\cdot d(u,v)+k.\]
 \end{enumerate}
\end{definition}
If there is some quasi-isometry between $\GG$ and $\GG'$, we say the two graphs are \emph{quasi-isometric}. Note that the above definition is a special case of quasi-isometries on metric spaces.
It follows from the axiom of choice that being quasi-isometric is a symmetric relation. It is easy to see that it also is reflexive and transitive.
The following well-known fact is straightforward to see. 

\begin{lemma}\label{lem:quasifi}
 Let $G$ be a group and $H$ a subgroup of finite index. Then with respect to all generating sets of $H$ and of $G$, the Cayley graph of $H$ is quasi-isometric to the Cayley graph of $G$.
\end{lemma}

\begin{proposition}\label{prop:quasitree}
Let $\GG$ have finite treewidth and assume that the degree of $\GG$ and $\GG'$ is uniformly 
bounded by some constant $d$. If $\GG'$ is quasi-isometric to $\GG$, then $\GG'$ has finite treewidth, too.
\end{proposition}

\begin{proof}
Let $\phi: \GG' \to \GG$ be a quasi-isometry and assume that we have a tree decomposition of $\GG$ with finite bag-size. 
Let $\ell$ be a constant such that for all edges $uv\in E(\GG')$ we have $d(\phi(u),\phi(v))\leq \ell$.
We construct a tree decomposition of $\GG'$ with the same underlying tree by replacing every bag $X$ with $\phi^{-1}(N^\ell(X))$. 
In fact, replacing every bag $X$ of the original tree decomposition by $N^\ell(X)$ yields a tree decomposition by \prref{lem:umgeb}. 
The step to the preimage does not destroy the conditions (T1) and (T3). 
By the choice of $\ell$ also (T2) is assured.

It remains to show that the bag-size is finite. Since $\phi(u) = \phi(v)$ implies that $u$ and $v$ are in bounded distance and the degree in $\GG'$ is uniformly bounded, there is some constant $k$ such that $\abs{\phi^{-1}(v)}\leq k$ for all $v\in V(\GG)$.
We have $\abs{N^\ell(X)} \leq d^\ell \cdot \abs X$, and hence $\abs {\phi^{-1}(N^\ell(X))} \leq k \cdot d^\ell\cdot\abs X$.
\end{proof}

The next statements follow directly from \prref{lem:quasifi} and \prref{prop:quasitree}.

\begin{corollary}\label{cor:quasiso}
Let $G$ be a virtually free group. Then the Cayley graph of $G$ is quasi-isometric to a tree. 
\end{corollary}

\begin{corollary}\label{cor:treewidthquasitotree}
Let $\GG$ be quasi-isometric to a tree and of uniformly bounded degree. Then $\GG$ has finite treewidth.
\end{corollary}

\begin{corollary}\label{cor:treewidthsubgroup}
Let $G$ be a group with a Cayley graph having finite treewidth. 
Then the Cayley graph of every finitely generated subgroup w.\,r.\,t.\ to any finite set of generators has finite treewidth. 
\end{corollary}

\begin{corollary}\label{cor:quasivirt}
Let $G$ be a finitely generated virtually free group. 
Then its Cayley graph w.\,r.\,t.\ to any finite set of generators has finite treewidth. 
\end{corollary}

\section{Cuts and Structure Trees}\label{sec:cuts}
In order to prove that the groups having a Cayley graph of finite treewidth~-- and hence the context-free groups~-- are fundamental groups of finite graphs of groups with finite vertex groups, we have to do some work. The proof we present here is via cuts and structure trees. It goes back to the ideas of Dunwoody \cite{Dunwoody82}, which was rewritten in \cite{kroen10}.
The following text is taken from \cite{diekertW13}. It is reproduced in order to keep these notes self-contained.

\subsection{Cuts in Graphs of Finite Treewidth}

Let $\GG= (V(\GG),E(\GG))$ be a connected and locally finite graph. As in the previous section we assume that $\GG$ is simple without mentioning it further. Similarly to the vertex-boundary we define the edge-boundary of some subset $C\sse V(\GG)$: 
\[
\begin{array}{lrcl}
\text{Edge-boundary: }&\delta C & =&\set{uv\in E(\Gamma)}{u\in C, v\in \Comp{C}}.\\
\text{Vertex-boundary: } &\beta C&=& \set{\; u\in V(\Gamma)}{\exists\, v\in V(\Gamma) \text{ with } uv\in \delta C}.
\end{array}
\]

\begin{definition}\label{def:cut}
A \emph{cut} is a subset $C \subseteq V(\GG)$ such that
 \begin{enumerate}
\item $C$ and $\Comp{C}$ are non-empty and connected, 
\item $\delta C$ is finite.
\end{enumerate}
The \emph{weight} of a cut is defined by $\abs{ \del C}$. If $\abs{\delta C}\leq k$, then we call $C$ a \emph{$k$-cut}.
\end{definition}
We are interested in cuts where both parts $C$ and $\Comp{C}$ are infinite. 
However, there might be no such cuts. For instance, consider the infinite grid
$\Z\times \Z$. 
It is connected and locally finite, but there is no cut splitting it into two infinite connected components. We will see that there is always such a cut if $\GG$ is a graph of finite treewidth with $\abs{\Aut(\GG)\bs \GG}<\infty$.

The following crucial observation can be found in \cite{ThomassenW93} in a slightly different formulation:

\begin{lemma}
\label{lem:endl_k_cuts}
Let $\GG= (V(\GG),E(\GG))$ be a connected and locally finite graph, let $S \sse V(\Gamma)$ be finite
and $k \geq 1$. There are only finitely many $k$-cuts $C$ with $\beta{C}\cap S \neq \es$.
\end{lemma}

\begin{proof} 
Let $e=uv \in E(\GG)$ be some fixed edge.
Since $\Gamma$ is locally finite, it is enough to show that the set of $k$-cuts $C$ with $e \in \delta C$ is finite. 
For $k=1$ this is trivial because there is at most one cut with $\smallset{e} = \delta C$. If the graph $\Gamma - e$ is not connected anymore,
\ie  $e$ is a so-called \emph{bridge}, then all cuts with $e \in \delta C$ have weight $k=1$. 
Thus, we may assume that the graph $\Gamma - e$ is still connected; and we may fix a path $\gamma$ from $u$ to $v$ in $\Gamma - e$.
Every $k$-cut $C$ with $e \in \delta C$ becomes a $k-1$-cut $C$ in the graph $\Gamma - e$. Such a cut must use one edge of $\gamma$ because otherwise we had either both $u,v \in C$ or both 
$u,v \in \Comp C$. By induction, there are only finitely
many $k-1$-cuts using vertices of $\gamma$. Thus, we are done. 
\end{proof}

In the following bi-infinite simple paths will play an important role for us (a bi-infinite path is a subgraph $(\{\dots, v_{-1},v_0,v_1,\ldots\}, \,\{\dots e_{-1},e_0,e_1,\ldots\})$ such that $s(e_i) =v_{i-1}$ and $t(e_i)=v_{i}$ for all $i\in \Z$). 
Note that if there is a cut with $C$ and $\Comp{C}$ infinite, then we can take two one-sided infinite paths one lying entirely in each component $C$ and $\Comp{C}$. Connecting the two paths results in a bi-infinite path $\alpha$ such that $\abs{\alp \cap C} = \infty = \abs{\alp \cap \Comp C}$. However, in general, not every bi-infinite path is split by a cut into two infinite pieces. For a bi-infinite simple path $\alp$ we define:
\begin{align*}
\cC(\alp) &= \set{C \sse V(\GG)}{\text{$C$ is a cut and } \abs{\alp \cap C} = \infty = \abs{\alp \cap \Comp C}}.
\end{align*}
That means $\cC(\alp)\neq \es$ \IFF there is a cut such that
the graph $\alp - \del C$ has exactly two infinite components each of these two being a one-sided infinite subpath of $\alp$. 

We say $\GG$ has \emph{more than one end} if there is some finite set $S\sse V(\GG)$ such that $\GG -S$ has more than one infinite connected component. Otherwise, $\GG$ has at most \emph{one end}.

In our setting ($\GG$ connected and locally finite) this means that $\GG$ has more than one end if and only if there is some bi-infinite simple path $\alp$ such that 
$\cC(\alp) \neq \es$. Note that in literature there are various different definitions for the ends of a graph. However, for connected, locally finite graphs they all coincide.

\begin{lemma}\label{lem:sep_von_1}
Let $\Gamma$ be a graph of finite treewidth and uniformly bounded degree. Then there exists some $k\in \N$ satisfying the following property: For every one-sided infinite simple path $\gamma$, every $v_0\in V(\Gamma)$, and every $n \in \N$ there is some $k$-cut $D$ with $d(v_0, \Comp{D})\geq n$, $v_0\in D$, and $\abs{\Comp{D}\cap \gamma} = \infty$.
\end{lemma}

\begin{proof}
Let $d$ be the maximal degree of $\Gamma$ and let $T$ be a tree decomposition with bag-size $m = \bags(T)$.
We set $k = dm$. 

Let $t_0\in V(T)$ such that $v_0\in X_{t_0}$. Consider vertices $u, v \in V(\Gamma) - X_{t_0}$ which are in bags of two different connected components of $T- t_0$. Then by \prref{prop:basictreecut}, every path from $u$ to $v$ has a vertex in $X_{t_0}$, so $u$ and $v$ are not in the same connected component of $\Gamma - X_{t_0}$. Since $ X_{t_0}$ is finite, there is exactly one connected component of $\Gamma - X_{t_0}$ which contains infinitely many vertices of $\gamma$.
Let $C_{t_0,\gamma}$ be this component.
Then the set $C_{t_0,\gam}$ is contained in the union of the bags of one connected component of $T- t_0$. Let $t_1$ be the neighbor of $t_0$ in this connected component, which is uniquely defined because $T$ is a tree.

Repeating this procedure yields a simple path $t_0, t_1, t_2, \ldots$ in $T$ and a sequence of connected sets
$C_{t_0,\gam}, C_{t_1,\gam}$, $C_{t_2,\gam}, \ldots$ such that $\abs{\gam \cap C_{t_i,\gam}} = \infty$ for all $i\in \N$. By \prref{prop:basictree}, we may assume that every node $v \in V(\Gamma)$ is contained in only finitely many bags. Hence, we can choose $\ell$ large enough such that $X_{t_\ell}$ does not contain any $v \in V(\Gamma)$ with $d(v_0,v) \leq n$. 

Now, let $D$ be the connected component of $\Comp{C_{t_\ell,\gam}}$ which contains $v_0$. Then $\Comp{D}$ is connected because every vertex in another connected component of $\Comp{C_{t_\ell,\gam}}$ is connected with $C_{t_\ell,\gam}$ inside of $\Comp{D}$ and $C_{t_\ell,\gam}$ itself is connected.

Since every edge of $\delta D$ has one of its incident nodes in $X_{t_\ell}$, we have $\abs{\delta D} \leq dm = k$. Thus, $D$ is a $k$-cut with $v_0\in D$ and $\abs{\Comp{D}\cap \gamma} = \infty$. Furthermore, since every path from $v_0$ to a vertex $v\in \Comp{D}$ uses a vertex of $X_\ell$, we have $d(v_0, \Comp{D})\geq n$.
\end{proof}

From \prref{lem:sep_von_1} we can derive that, if there is a cut splitting some bi-infinite simple path, then there is already such a cut with weight less than some constant which only depends on $\GG$. This leads to the following definition 
 due to \cite{ThomassenW93}:
\begin{definition}\label{def:accessible}
A graph is called \emph{accessible} if there exists a constant $k \in \N$ such that for every bi-infinite simple path $\alp$ either $ \cC(\alp) $ is empty or $ \cC(\alp) $ contains some $k$-cut
 \end{definition}
 
 The origin of this definition is the accessibility of groups: using results from \cite{DicksD89}, Thomassen and Woess \cite{ThomassenW93} showed that a group is accessible if and only if its Cayley graph is accessible.

\begin{proposition}\label{prop:fritz}
Let $\Gamma$ be a graph of finite treewidth and uniformly bounded degree. 
Then $\Gamma$ is accessible.
\end{proposition}

\begin{proof}
Let $\alpha$ be a bi-infinite simple path such that $ \cC(\alpha)\neq\es$
and let $C\in\cC(\alpha)$. We fix a vertex $v_0 \in \beta C$ and we let 
$n = \max\set{d(v_0,w)}{ w \in \beta C}$. Let $k \in \N$ be according to 
\prref{lem:sep_von_1}. It follows that there is a $k$-cut $D$ with $\abs{\alpha\cap \Comp{D}}= \infty $, $v_0\in D$, and $d(v_0, \Comp{D})\geq n$. Because of the choice of $n$, we also have $\beta C \sse D$ what means that either $C\sse D $ or $\Comp C \sse D$. In either case $D$ splits $\alpha$ into two infinite pieces.
\end{proof}

\begin{lemma}\label{lem:VDzweis_geod}
Let $\Gamma$ be a connected, locally finite, and infinite graph such that
$\Aut(\Gamma)\bs \Gamma$ is finite. Then there is a bi-infinite geodesic.
 \end{lemma}
\begin{proof}
Consider the infinite collection of all geodesics of odd length. 
Since $\Aut(\GG)\bs \GG$ is finite, there exists some fixed vertex $v$ and an infinite collection of geodesics of odd length having $v$ as their middle vertex. These geodesics form a tree. 
The result follows by K\"onigs Lemma.
\end{proof}
Note that we cannot remove any of the requirements in \prref{lem:VDzweis_geod}. 
In particular, we cannot remove that $\Aut(\Gamma)\bs \Gamma$ is finite. For example consider the graph $\Gamma$ with $V(\Gamma) = \Z$ and $E(\Gamma)=\set{\smallset{n,n\pm 1}, \smallset{n,-n}}{n\in \Z}$.
This graph is connected, locally finite, and infinite. It has a bi-infinite simple path, but there is no bi-infinite geodesic.

\begin{proposition}\label{prop:two_ends}
Let $\Gamma$ be connected, locally finite, and infinite such that 
$\Aut(\Gamma)\bs \Gamma$ is finite and let $\GG$ have finite treewidth. Then $\GG$ has more than one end.
 \end{proposition}

\begin{proof}
The graph $\Gamma$ has uniformly bounded degree because it is locally finite and $\Aut(\Gamma)\bs \Gamma$ is finite. By \prref{lem:sep_von_1}, there is some $k$ such that for every $n\in \N$, $v_0 \in V(\Gamma)$ and every one-sided infinite simple path $\alpha$ there is a $k$-cut $C$ with $v_0\in C$, $d(v_0,\Comp{C})\geq n$, and $\abs{\Comp{C}\cap \alpha} = \infty$.

Since $\Aut(\Gamma)\bs \Gamma$ is finite, it follows from \prref{lem:endl_k_cuts} that
there are only finitely many orbits of $k$-cuts under the action of $\Aut(\Gamma)$. Therefore, there is some $m\in \N$ such that $\max\set{d(u,v)}{ u,v \in \beta C}\leq m$ for all $k$-cuts $C$.

Assume that $\Gamma$ has only one end. By \prref{lem:VDzweis_geod}, there is some bi-infinite geodesic $\alpha = \dots, v_{-2},v_{-1},v_0,v_1,v_2\dots$. Let $C$ be a $k$-cut with $d(v_0, \Comp{C})> m$ such that $v_0 \in C$ and $\abs{\alpha\cap \Comp{C}}=\infty$. Then $\abs{\alpha\cap C}<\infty$, for otherwise $\cC(\alpha)\neq \es$.

Hence, there are $i,j>m$ with $v_{-i}, v_j \in \beta C\cap \Comp{C}$. But this implies $d(v_{-i},v_j)=d(v_{-i},v_0)+d(v_0,v_j)> 2m$ in contradiction to $d(u,v)\leq m$ for all $u,v \in \beta C$.
Hence, $\GG$ has more than one end.
 \end{proof}

\subsection{Optimally Nested Cuts}
As we have seen, the graphs we are interested in are accessible. Therefore, for the rest of this section
let $\GG= (V(\GG),E(\GG))$ be a connected, locally finite, and accessible graph.
In the following, we only want to deal with cuts having minimal weight among those cuts splitting some bi-infinite simple path. Therefore, for a bi-infinite simple path $\alp$ we define:
\begin{align*}
\cC_{\min}(\alp) &= \set{C \in \cC(\alp)}{\text{$\abs{\del C}$ is minimal in }\cC(\alp)},\\
\cC_{\min} &= \bigcup\set{\cC_{\min}(\alp)}{\alp \text{ is a bi-infinite simple path} }.
\end{align*}
That means we have $\cC_{\min} = \es$ if and only if $\GG$ has at most one end.
The set of minimal cuts may contain cuts of very different weight. 
Indeed, we might have $C,D \in \cC(\alp) \cap \cC_{\min}$
with $C \in \cC_{\min}(\alp)$, but $D \notin \cC_{\min}(\alp)$. 
In this case, there must be another bi-infinite simple path $\bet$ 
with $D \in \cC(\alp) \cap \cC_{\min}(\bet)$ and $\abs{\del C} < \abs{\del D}$. 

\begin{example}\label{ex:different_deltas}
Let $\GG$ be the subgraph of the infinite grid
$\Z \times \Z$ which is induced by the pairs $(i,j)$ satisfying $j\in \{0,1\}$ or $i=0$ and $j\geq 0$. 
Let $\alp$ be the bi-infinite simple path with $i=0$ or $j=1$ and $i\geq 0 $ and let $\bet$ be the bi-infinite simple path defined by $j=0$. 
Then there are cuts $C,D \in \cC(\alp) \cap \cC_{\min}$ with $\abs{\del C} = 1$ and $\abs{\del D} = 2$, see \prref{fig:different_deltas}.

\begin{figure}[ht]
\begin{center}
\begin{tikzpicture}[scale = 0.7]
\def\width{5};
\def\height{4};

\draw (-\width -0.5,1) -- (\width +0.5,1) ;
\draw (-\width -0.5,0) -- (\width +0.5,0) ;
\draw (0,0) -- (0,\height + 0.5) ;
\node () at (-\width -1.1, 0.5) {\footnotesize{$\cdots$}};
\node () at (+\width +1.1, 0.5) {\footnotesize{$\cdots$}};
\node () at (0, \height + 1.3) {\footnotesize{$\vdots$}};
\node (oo) at (0,0){};
\node [below=2pt] {$(0,0)$};
\foreach \x in {1,...,\width}
{
	\draw (\x,1) -- (\x,0);
	\draw (-\x,1) -- (-\x,0);
}
\foreach \y in {2,...,\height}
{
	\draw (-0.075,\y) -- (0.075,\y);
}

\node () at (1.5, 2.0) {$\delta D$};
\draw[dashed] (1.5,1.5) -- (1.5,-0.5) ;
\node () at (-1.2, 2.5) {$\delta C$};
\draw[dashed] (-0.7,2.5) -- (0.7,2.5) ;

\draw[dotted,line width=1.4pt] (-\width -0.8,0) -- (\width +0.8,0) ;

\draw[dotted,line width=1.4pt] (0,1) -- (0,\height + 0.8) ;
\draw[dotted,line width=1.4pt] ( 0,1) -- (\width +0.8,1) ;
\node () at (\width +0.8, 1.4) {$\alpha$};
\node () at (\width +0.8, -0.4) {$\beta$};

\end{tikzpicture}
\end{center}
\caption[]{
The subgraph of the grid $\Z \times \Z$ induced by the pairs $(i,j)$ satisfying $j\in \{0,1\}$ or $i=0$ and $j\geq 0$. Here we have $D \in \cC(\alp) \cap \cC_{\min}$ but $D \notin \cC_{\min}(\alp)$.}\label{fig:different_deltas}
\end{figure}
\end{example}

Two cuts $C$ and $D$ are called \emph{nested} if one of the four inclusions
$C\sse D$, $C\sse \Comp{D}$, $\Comp{C}\sse D$, or $\Comp{C}\sse\Comp{D}$ holds.

\begin{figure}[ht]
\begin{center}
\begin{tikzpicture}
\draw (0,2.3) -- (0,-2.3){};
\draw (2.5,0) -- (-2.5,0){};

\node (a) at (-1.2, 1.2) {$C\cap D$};
\node (a) at (1.2, 1.2) {$\Comp{C}\cap D$};
\node (a) at (-1.2, -1.2) {$C\cap \Comp{D}$};
\node (a) at (1.2, -1.2) {$\Comp{C}\cap \Comp{D}$};

\node[anchor=base] (a) at (-0.25, 1.95) {\footnotesize{$C$}};
\node[anchor=base] (a) at (0.25, 1.95) {\footnotesize{$\Comp{C}$}};
\node (a) at (-2.2, 0.25) {\footnotesize{$D$}};
\node (a) at (-2.2, -0.25) {\footnotesize{$\Comp{D}$}};
\end{tikzpicture}
\end{center}
\caption[]{The corners of $C$ and $D$. Nested cuts have one empty corner.}\label{fig:corners1}
\end{figure}

The set $\oneset{C \cap D, C\cap \Comp{D}, \Comp{C}\cap D, \Comp{C}\cap \Comp{D}}$ is called the set of \emph{corners} of $C$ and $D$, see \prref{fig:corners1}. Two corners $E, E'$ of $C$ and $D$ are called \emph{opposite} if either $\smallset{E, E'} =\smallset{C \cap D, \, \Comp{C}\cap \Comp{D}} $ or $\smallset{E, E'} = \smallset{\Comp{C}\cap D, \, C\cap \Comp{D} }$. 
Two different corners are called \emph{adjacent} if they are not opposite. 
Note that two cuts $C,D$ are nested if and only if one of the four corners of $C$ and $D$ is empty.

\begin{lemma}\label{lem:fred}
Let $k\in \N$ and $C$ be a cut. 
There are only finitely many $k$-cuts which are not nested with $C$.
\end{lemma}

\begin{proof}
Let $S$ be a finite 
connected subgraph of $\GG$ containing all vertices of $\bet C$. 
The number of $k$-cuts $D$ with $\bet D \cap S \neq \es$
is finite by \prref{lem:endl_k_cuts}. For all other cuts we 
may assume (by symmetry) that $\bet C \sse D$. However, this implies $C\sse D$ or $\Comp C\sse D$.
\end{proof}

Since we assume that $\GG$ is accessible, there is some constant $k$ such that for all 
bi-infinite simple paths $\alp$ with $ \cC(\alp)\neq \es $ there exists some cut $ C \in \cC(\alp) $ 
with $\abs{\delta C} \leq k$. We fix this $k$ for the rest of this section. By \prref{lem:fred}, this allows us to define a natural number $m(C)$ for every cut $C$:
$$m(C) = \abs{\set{D }{\text{$C$ and $D$ are not nested and $D$ is a $k$-cut}}}.$$
Furthermore, we use the following notation, where $\alp$ denotes a bi-infinite simple path:
\begin{align*}
m_{\alp}&= \min \set{m(C)} {C\in \cC_{\min}(\alp)},\\
\Copt(\alp)&= \set{C\in \cC_{\min}(\alp)} {m(C)=m_{\alp}},\\
\Copt&= \bigcup\set{\Copt(\alp)}{\alp \text{ is a bi-infinite simple path} }.
\end{align*}

\begin{definition}\label{def:c_opt}
A cut $C\in \Copt$ is called an \emph{optimally nested cut}.
For simplicity, an optimally nested cut is also called \emph{optimal cut}.
\end{definition}

Since every ``cuttable'' bi-infinite simple path
can be ``cut'' into two infinite parts at least by one optimal cut, we can forget all other cuts and just focus on 
optimal cuts. 
The next result shows that the optimal cuts in fact behave very well. 
\begin{proposition}\label{prop:opt_nested}
Let $C,D \in \Copt$. Then $C$ and $D$ are nested.
\end{proposition}

\begin{proof} 
Let $C\in \Copt(\alp)$ and $D \in \Copt(\bet)$ for some bi-infinite simple paths $\alp$ and $\bet$.  By contradiction, let us assume that $C$ and $D$ are not nested. By symmetry, we assume that $m_\beta \leq m_\alpha$ henceforth\footnote{This argument is missing in the corresponding proof in \cite{DiekertW17crm} and in the previous arXiv version of the present paper. However, the setting $m_\beta \leq m_\alpha$ is important for the proof.}. 
Thus, if $D \in \cC_{\min}(\alpha)$, then  $D \in \Copt(\alpha)$, too. 
In that case we replace $\beta$ by $\alpha$. Hence, we assume 
$$\alpha = \beta \iff D \in \cC_{\min}(\alpha).$$
The aim is to construct cuts $E,E'$ with $E\in \cC_{\min}(\alp)$ and $E'\in \cC_{\min}(\bet)$ such that $m(E) + m(E') < m(C) + m(D)$, which is a contradiction to $C\in \Copt(\alp)$ and $D \in \Copt(\bet)$.

As a first step we show that there are two opposite corners $E$ and $E'$ of $C$ and $D$ such that $\abs{\alp \cap E} = \abs{\beta \cap E'}= \infty$. We distinguish two cases: $D \in \cC_{\min}(\alp)$ and $D \notin \cC_{\min}(\alp)$.
First, let $D \in \cC_{\min}(\alp)$. 
Then, by our assumption, we have $\alp= \bet$. In particular, there are opposite corners $E$ and $E'$ such that $\abs{\alp \cap E} = \abs{\bet \cap E'}= \infty$, see \prref{fig:corners4}.

In the other case we have $D \notin \cC_{\min}(\alp)$,
and therefore $\alp \neq \bet$. We claim that there must be one corner $K$ of $C$ and $D$ such that $\abs{\alp \cap K} < \infty$ and $\abs{\bet \cap K} < \infty$ as depicted in \prref{fig:corners3}. Indeed, if there is no such corner $K$, then infinite parts of $\alp$ and $\bet$ are in opposite corners respectively, see \prref{fig:corners2}. 
 In particular, both $\alpha$ and $\beta$ are split by $C$ as well as by $D$ into two infinite pieces. This implies $\abs{\delta C}= \abs{\delta D}$, and hence $D \in \cC_{\min}(\alp)$. Thus, the corner $K$ exists and we define
 $E$ and $E'$ to be the adjacent corners of $K$. Without loss of generality, $E$ splits $\alpha$ into two infinite pieces and $E'$ splits $\beta$ into two infinite pieces.

 \begin{figure}[ht]
\begin{center}
\begin{tikzpicture}
\draw (0,2.3) -- (0,-2.3){};
\draw (2.5,0) -- (-2.5,0){};
\draw[snake=snake, segment length = 9mm] (1.6, -2.) --(-1.6, 1.6) ;
\node() at (-0.95,1.6){$\alpha = \beta$} ;

\node[anchor=base] (a) at (-0.25, 1.95) {\footnotesize{$C$}};
\node[anchor=base] (a) at (0.25, 1.95) {\footnotesize{$\Comp{C}$}};
\node (a) at (-2.2, 0.25) {\footnotesize{$D$}};
\node (a) at (-2.2, -0.25) {\footnotesize{$\Comp{D}$}};
\end{tikzpicture}
\end{center}
\caption[]{We have $\alp = \bet$ and $\abs {C\cap D \cap \alp}= \abs {\Comp C\cap \Comp D \cap \bet} = \infty$.}\label{fig:corners4}
\end{figure}
 
\begin{figure}[ht]
\begin{center}
\begin{tikzpicture}
\draw (0,2.3) -- (0,-2.3){};
\draw (2.5,0) -- (-2.5,0){};
\draw[snake=snake, segment length = 9mm] (1.6, -2.) --(-1.6, 1.6) ;
\node() at (-1.35,1.6){$\alpha$} ;
\draw[snake=snake, segment length = 9mm] (1.5, 1.6) --(-1.,- 2);
\node() at (1.75,1.6) {$\beta$};

\node[anchor=base] (a) at (-0.25, 1.95) {\footnotesize{$C$}};
\node[anchor=base] (a) at (0.25, 1.95) {\footnotesize{$\Comp{C}$}};
\node (a) at (-2.2, 0.25) {\footnotesize{$D$}};
\node (a) at (-2.2, -0.25) {\footnotesize{$\Comp{D}$}};
\end{tikzpicture}
\end{center}
\caption[]{For all four corners $K$ we have $\max\{\abs {K \cap \alp}, \abs {K \cap \bet}\} = \infty$.}\label{fig:corners2}
\end{figure}

\begin{figure}[ht]
\begin{center}
\begin{tikzpicture}
\draw (0,2.3) -- (0,-2.3){};
\draw (2.5,0) -- (-2.5,0){};

\draw[snake=snake, segment length = 9mm] (1., -2) --(-1.5, 1.5) ;
\node() at (-1.94,1.0){$\alpha$ or $\beta$} ;
\draw[snake=snake, segment length = 9mm] (1.7, 1.55) --(1.7,- 2);
\node() at (1.85,0.9) {$\beta$};
\draw[snake=snake, segment length = 9mm] (-1.7, 1.7) --(1.55, 1.7);
\node() at (0.8,1.8) {$\alpha$};

\node (a) at (-1.2, -1.2) {$K$};

\node[anchor=base] (a) at (-0.25, 1.95) {\footnotesize{$C$}};
\node[anchor=base] (a) at (0.25, 1.95) {\footnotesize{$\Comp{C}$}};
\node (a) at (-2.2, 0.25) {\footnotesize{$D$}};
\node (a) at (-2.2, -0.25) {\footnotesize{$\Comp{D}$}};
\end{tikzpicture}
\end{center}
\caption[]{For one corner $K$ we have 
$\max\{\abs {K \cap \alp}, \abs {K \cap \bet}\} <\infty$.}
\label{fig:corners3}
\end{figure}

In both cases, $E$ and $E'$ are defined such that 
$\abs{\alp \cap E} = \abs{\bet \cap E'}= \infty$.
 By interchanging, if necessary, $C$ with $\Comp C$ and 
 $D$ with $\Comp D$, we may assume that 
 $E= C \cap D$ and $E' = \Comp C \cap \Comp D$, too. 
 
 Thus, in all cases we are in the following situation: 
 $C$ and $D$ are not nested, $C\in \cC_{\min}(\alp)$, 
 $D\in \cC_{\min}(\bet)$, 
 $E= C \cap D$, $E' = \Comp C \cap \Comp D$, and $\abs{\alp \cap E} =\abs{\alp \cap \Comp E} = \abs{\bet \cap E'} = \abs{\bet \cap \Comp E'}= \infty$. Possibly $\alp =\bet$, 
but it is not yet clear that $E$ and $E'$ are cuts. 

The graph $\GG(E)$ contains an infinite connected component $F \sse E$ 
such that $\abs{\alp \cap F} = \infty$. 
Since $\Comp C \cup \Comp D \sse \Comp F$, it is easy to see that 
$\Comp F$ is connected and $\abs{\alp \cap \Comp F} = \infty$.
 Hence, $F$ is a cut splitting $\alp$ into two infinite pieces. 
 In a symmetric way we find a cut $F' \sse E'$ such that $\abs{\bet \cap F'} = \abs{\bet \cap \Comp F'} = \infty$.
The next step of the proof is to show that $F = E \in \cC_{\min}(\alp)$ and $F' = E'\in \cC_{\min}(\bet)$.

We have $\delta E\cup \delta E' \sse \delta C\cup \delta D$ and $\delta E\cap \delta E' \sse \delta C\cap \delta D$ by the definition of $E$ and $E'$. This yields
\begin{align}\abs{\del E} + \abs{\del E'} \leq \abs{\del C} + \abs{\del D}\label{eq:submodular}
 \end{align}
since every edge which is counted once resp.\ twice on the left-hand side is counted at least once resp.\ twice on the right-hand side.
 Because of the minimality of $\abs{\del C}$ and $\abs{\del D}$ we have $\abs{\del C}\leq \abs{\del F}$ and $\abs{\del D} \leq \abs{\del F'}$. Since $F$ is a connected component of $E$, we have $\delta F\sse \delta E$ and likewise $\delta F' \sse \delta E'$. With \prref{eq:submodular} we obtain $\abs{\del C}= \abs{\del F}= \abs{\del E}$ and $\abs{\del D} = \abs{\del F'} = \abs{\del E'}$. This implies $\del F = \del E$ and $\del F' = \del E'$, and hence $F= E \in \cC_{\min}(\alp)$ and $F'= E'\in \cC_{\min}(\bet)$.

The final step in the proof is the following assertion:
\begin{equation}\label{eq:frodo}
m(E)+m(E')<m(C)+m(D).
\end{equation}
 
 Once we have established \prref{eq:frodo} we are done since \prref{eq:frodo} implies $m(E) < m(C)$ or $m(E') < m(D)$.

In order to see \prref{eq:frodo}, we show two claims: 
\begin{enumerate}
\item If a cut $F$ is nested with $C$ \emph{or} nested with $D$, then $F$ is nested with $E$ \emph{or} nested with $E'$:\label{eq_cap_nested}

By symmetry let $F$ be nested with $C$. 
If 
$F\sse C$ (resp.{} $\Comp F\sse C$), then $F\sse \Comp {E'}$ (resp.{} $\Comp F \sse \Comp {E'}$). 
If 
$C\sse F$ (resp.{} $C \sse \Comp F $), then $E \sse F$ (resp.{} $E \sse \Comp {F}$).

\item If a cut $F$ is nested with both $C$ \emph{and} $D$, then $F$ is nested with both $E$ \emph{and} $E'$:\label{eq_cup_nested}

By symmetry in $F,\Comp{F}$ we may assume $C \sse F$ or $\Comp C \sse F$.
Using the symmetry in $E,E'$
 we may assume that $C\sse F$. Hence, we have $E\sse F$; and it remains to show
 that $E'$ and $F$ are nested. 
 If $D \sse \Comp F$ or $\Comp D \sse \Comp F$, then it follows that $C \cap D = \es$ resp.\ $C \cap \Comp D = \es$.
 Both is impossible because $C$ and $D$ are not nested. 
 For $D \sse F$ we obtain $\Comp{E'} = C \cup D \sse F$ what implies that $E'$ and $F$ are nested. 
 Finally, let $\Comp D \sse F$, then $E' \sse F$. Again $E'$ and $F$ are nested.
\end{enumerate}
\renewcommand{\labelenumi}{\arabic{enumi}.}

As in \prref{eq:submodular}, claims \ref{eq_cap_nested} and \ref{eq_cup_nested} together yield
$ m(E)+m(E')\leq m(C)+m(D).$
Now, $C$ is nested with both corners $E$ and $E'$. Hence, $C$ is not counted on the left-hand side 
of the inequality. However, $C$ is counted on the right-hand side because $C$ is not nested with $D$. That means the inequality in \prref{eq:frodo} is strict. Hence, we have shown the result of the proposition. 
\end{proof}

Analogous results to \prref{prop:opt_nested} are Theorem 1.1 in \cite{Dunwoody82} or Theorem 3.3 in \cite{kroen10}. 
In contrast to these results, \prref{prop:opt_nested} allows that $\Copt$ may contain cuts of different weights. 
We have to deal with cuts of different weights because we wish to get a ``complete'' decomposition of virtually free groups like 
$(\Z \times \Z/2\Z)*\Z/2\Z$ without applying the procedure several times. As in the graph in \prref{ex:different_deltas}, in the Cayley graph of this group cuts with weight $1$ and $2$ are necessary to split all bi-infinite paths into two infinite pieces. 

\subsection{The Structure Tree}\label{vdsec:tree_set} 
The notion of \emph{structure tree} is due to Dunwoody \cite{Dunwoody79}.
Since $\GG$ is assumed to be accessible, $\Copt$ is defined and there is some $k \in \N$ such that every cut in $\Copt$ is a $k$-cut. 

\begin{lemma}\label{lem:tree_set}
Let $C,D\in\Copt$. Then the set $\set{E \in \Copt}{C\sse E\sse D}$ is finite. 
\end{lemma}

\begin{proof}
Choose two vertices $u\in C$ and $v \in \Comp{D}$, and a path 
$\gam$ in $\GG$ connecting them. Every cut $E$ with $C\sse E\sse D$ must separate $u$ and $v$ and thus contain a vertex of of $\gam$. With \prref{lem:endl_k_cuts} and the accessibility of $\Gamma$ it follows that there are only finitely many such cuts.
\end{proof}

The set $\Copt$
is partially ordered by $\sse$. \prref{lem:tree_set} states that the partial order is discrete; hence, it is induced by 
its so-called \emph{Hasse diagram}. 
For a general partial order $(X,<)$ there is an arc in the Hasse diagram from $x$ to $y$ if and only if $x< y$ and there is no $z$ with $x<z<y$.
%

If there is an arc from $\Comp C$ to $D$, then there is also an arc from $\Comp D$ to $ C$. In such a situation we put $C$ and $D$ 
in one class: 
\begin{definition}\label{def:cuts_equiv}
For $C$, $D \in \Copt$ we define the relation $C\sim D$ by the following condition: 
\begin{equation*}
\text{Either } C=D \text{ or both } \Comp{C} \subsetneqq D \text{ and } \forall\; E \in \Copt: \,\Comp{C} \subsetneqq E \subseteq D \implies D =E.
\end{equation*}
\end{definition}
The intuition behind this definition is as follows: Consider $\Copt$ as the
edge set of some graph. Call edges $C$ and $D$
to be adjacent if $C \sim D$. This makes sense due the following property. 

\begin{lemma}\label{lem:frida}
The relation $\sim$ is an equivalence relation.
\end{lemma}
\begin{proof}
Reflexivity and symmetry are immediate.
Let $C\sim D\neq C$ and $D\sim E \neq D$.
This implies $\es \neq \Comp D \sse C\cap E$. 
We have to show that $C\sim E$. The cuts $C$ and $E$ are nested due to 
\prref{prop:opt_nested}. Hence, we have one of the following four
inclusions $C\sse E$, $E\sse C$, $E\sse \Comp{C}$, and $\Comp{C} \sse E$. In order to prove transitivity one has to check all these possibilities. This is straightforward and we leave it to the reader.
\end{proof}

\begin{definition}[Structure Tree]
Let $T(\Copt)$ denote the following graph:
\begin{align*}
V(T(\Copt))&=\set{[C]}{C \in \Copt},\\
E(T(\Copt))&= \Copt.
\end{align*}
The incidence maps are defined by $s(C)=[C]$ and $t(C)=[\Comp{C}]$. The involution $\Comp{C}$ is defined by the complementation $\Comp{C} = V(\GG)\sm C$; hence, we do not need to change notation.
\end{definition}
The directed edges are in canonical bijection with the pairs $([C],[\Comp{C}])$. Indeed, let $C\sim D$ and $\Comp{C}\sim \Comp{D}$. It follows $C=D$ because otherwise $C\ssnq \Comp{D}\ssnq C$.
Thus, $T(\Copt)$ is an undirected simple graph.

\begin{proposition}
[\cite{Dunwoody79}]\label{prop:tree}
The graph $T(\Copt)$ is a tree.
\end{proposition}

\begin{proof}
Let $\gamma$ be a simple path in $T(\Copt)$ of length at least two.
Then $\gamma$ corresponds to a sequence of cuts
\[C_0,\, \Comp{C}_0 \sim C_1, \dots, \Comp{C}_{n-2} \sim C_{n-1},\,\Comp{C}_{n-1}= C_n.\]
with $[C_{i-1}] \neq [C_{i+1}]$. It follows that $\Comp{C}_{i-1} \neq C_i$ for $1\leq i < n$ since otherwise we would have $C_{i-1}= \Comp{C}_i \sim C_{i+1}$.
Hence, we obtain a sequence
\[C_0\subsetneqq C_1\subsetneqq C_2\ssnq\, \cdots \, \subsetneqq C_{n-1}.\]
Therefore, we have $C_0 \neq \Comp{C}_{n-1}$ and $\Comp{C}_0 \not\sse\Comp{C}_{n-1}$. In particular, $C_0 \not\sim \Comp{C}_{n-1}=C_n$ and the original path is not a cycle. Hence, $T(\Copt)$ has no cycles.

It remains to show that $T(\Copt)$ is connected.
Let $[C],[D] \in V(T(\Copt))$. Since $C$ and $D$ are nested and there are edges connecting $[C]$ and $[ \Comp{C}]$ resp.~$[D]$ and $[ \Comp{D}]$, we may assume $C\sse D$. 
By \prref{lem:tree_set}, there are only finitely many cuts $E \in \Copt$, with $C\sse E\sse D$. Now, let $C_0, C_1,\dots, C_n$ be a not refinable sequence of cuts in $\Copt$ such that
\[C=C_0\subsetneqq C_1\subsetneqq C_2 \ssnq\, \cdots \, \subsetneqq C_{n-1} \subsetneqq C_n=D.\]
Then we obtain a path from $C$ to $D$:
\[C=C_0,\, \Comp{C}_0 \sim C_1,\, \Comp{C}_1 \sim C_2,\dots, \Comp{C}_{n-1} \sim C_n=D.\]
Hence, $T(\Copt)$ is connected and therefore a tree.
\end{proof}

\begin{remark}\label{rem:dunwoody}
 According to Dunwoody \cite{Dunwoody79} a \emph{tree set} is 
a set of pairwise nested cuts which is closed under complementation and 
with the property that for all $C,D\in\cC$ the set $\set{E \in \cC}{C\sse E\sse D}$ is finite. Using this terminology, \prref{prop:opt_nested} and \prref{lem:tree_set} show that
$\Copt$ is a tree set. Once this is established \prref{prop:tree} becomes a general fact due to Dunwoody \cite[Thm.~2.1]{Dunwoody79}.
\end{remark}

\subsection{Actions on the Structure Tree}\label{sec:blocks}
Now, in addition to the setting in the previous sections we want to introduce a group action on the graph.
In the following, $\GG$ denotes a connected, locally finite, and accessible graph such that the group of automorphisms $\Aut(\GG)$ acts with finitely many orbits on $\GG$. 
For example, if $\GG$ is the Cayley graph of a group $G$ with respect to some finite generating set, then $\GG$ is connected, locally finite, and the action is with finitely many orbits since it is transitive on the set of vertices.  
The action on $\GG$ induces an action of $\Aut(\GG)$ on $\Copt$ and on the structure tree $T(\Copt)$. 

\begin{lemma}\label{lem:VDendl_orbits_k_cuts}
Let $\Aut(\GG)\bs \GG$ be finite and 
$k\in \N$. Then the canonical action of $\Aut(\GG)$ on the set of $k$-cuts has 
finitely many orbits, only. In particular $\Aut(\GG)$ acts on 
$\Copt$ and on the tree $T(\Copt)$ with finitely many orbits.
\end{lemma}

\begin{proof} Let $\Aut(\GG)\bs V(\Gamma)$ be represented by some finite vertex set $U\sse V(\GG)$.
\prref{lem:endl_k_cuts} states that there are only finitely many $k$-cuts $C$ with $U\cap \beta C \neq \es$. Since every cut is in the same orbit as some cut $C$ with $U\cap \beta C \neq \es$, the group $\Aut(\GG)$ acts on the set of $k$-cuts with 
finitely many orbits.

Since $\Gamma$ is accessible, there is a $k$ such that for all cuts $C \in \Copt$ holds $\abs{\delta C} \leq k$. 
For the last statement recall that $\Copt$ is the edge set of $T(\Copt)$. Thus, 
 the action of $\Aut(\GG)$ on $T(\Copt)$ has only finitely many orbits, too. 
\end{proof}

As in \prref{sec:treewidth}, for $S\sse V(\Gamma)$ and $\ell\in\N$ we denote the $\ell$-th neighborhood of $S$ with $N^\ell (S)$. 
For a cut $C$ we can choose $\ell$ large enough such that $N^\ell (C)\cap \Comp{C}$ is connected. Indeed, all points in $\beta C \cap \Comp{C}$ can be connected 
 inside $\Comp{C}$, hence for some $\ell$ large enough these points can be connected within $N^\ell (C)\cap \Comp{C}$. This $\ell$ suffices to make $N^\ell (C)\cap \Comp{C}$
 connected. 
By \prref{lem:VDendl_orbits_k_cuts}, there are only finitely many orbits of optimal cuts. 
 Thus, we can choose some 
 $\lam\in\N\setminus \{0\}$ which makes $N^\lam (C)\cap \Comp{C}$ connected for all $C\in \Copt$.
 We fix this $\lam $ for the rest of this section. 
 
In order to derive some more information about the vertex stabilizers $G_{[C]}= \set{g \in G}{gC\sim C}$ of vertices of the tree $T(\Copt)$, we assign to each vertex of $T(\Copt)$ a so-called \emph{block}. The definition is according to \cite{ThomassenW93}. In \prref{lem:endl_orbits_Bc} we show that the blocks are somehow ``small''. 
\begin{definition}[Block]\label{def:B{[C]}}
Let $\Aut(\GG)\bs \GG$ be finite and $\Copt$ be the set of optimal cuts.
 Let $\lam\geq 1$ be defined as above such that $N^\lam (C)\cap \Comp{C}$ is connected for all $C\in \Copt$. The \emph{block} assigned to $[C] \in V (T(\Copt))$ is defined by
\[B{[C]}= \bigcap_{D\sim C} N^\lam (D).\]
\end{definition}

\begin{lemma}\label{lem:VconBc}
We have $$B{[C]} = \bigcap_{D \sim C } D \; \cup \; 
\bigcup_{D \sim C }N^\lam (D) \cap \Comp{D}.$$
In particular, blocks are nonempty.
\end{lemma}

\begin{proof}
The inclusion from left to right is obvious. 
Hence, it is enough to show that we have $N^\lam (C) \cap \Comp{C}\sse B{[C]}$.
Since $N^\lam (C) \cap \Comp{C}\sse N^\lam (C)$, we only need to consider 
 $D\sim C$, $D\neq C$ and have to show that $N^\lam (C) \cap \Comp{C}\sse N^\lam (D)$.
 This follows from
 \[N^\lam (C) \cap \Comp{C} \sse \Comp{C} \ssnq D \sse N^\lam (D).\]
\end{proof}

\begin{example}
\prref{fig:schnitte_in_cayleygraph} shows a part of the Cayley graph of the free product $\Z/2\Z\,*\,\Z/3\Z = \genr{a,b}{a^2=1=b^3}$. The minimal cuts in this graph are the cuts with weight one. The optimal cuts are exactly the minimal cuts. The three cuts depicted with dashed lines belong to the same equivalence class and the bold vertices form the respective block. Here, we can choose $\lam=1$ for the definition of the blocks.
\begin{figure}[ht]
\begin{center}
\begin{tikzpicture}[scale=.65, rotate=20]
\tikzstyle{every node}=[inner sep=0pt]

\node[circle, fill,inner sep=1.4pt] (1) at (0, 0) {};
\node[circle, fill,inner sep=1.4pt] (a) at (2, 0) {};
\node[circle, fill,inner sep=1.4pt] (aa) at (1,1.732050807568877294 ) {};

\node (aaba) at (2, 5.464101615137754588) {};
\node (aabaa) at (0,5.464101615137754588) {};
\draw[dotted] (aaba)-- ++(30:1);
\draw[dotted] (aabaa)-- ++(150:1);
\path (aaba) -- ++(30:2) node (aabac){};
\path (aabaa) -- ++(150:2) node (aabaac){};

\node (abaa) at (5.732050807568877294, -1) {};
\node (aba) at (4.732050807568877294, -2.732050807568877294) {};
\draw[dotted] (aba)-- ++(270:1);
\draw[dotted] (abaa)-- ++(30:1);
\path (aba) -- ++(270:2) node (abac){};
\path (abaa) -- ++(30:2) node (abaac){};

\node (ba) at (-3.732050807568877294, -1) {};
\node (baa) at (-2.732050807568877294, -2.732050807568877294) {};
\draw[dotted] (ba)-- ++(150:1);
\draw[dotted] (baa)-- ++(270:1);
\path (ba) -- ++(150:2) node (bac){};
\path (baa) -- ++(270:2) node (baac){};

\draw (aa)-- ++(90:2) node[circle, fill,inner sep=1.4pt] {} -- ++(60:2) -- ++(180:2) -- ++(300:2){};
\draw (a)-- ++(330:2) node[circle, fill,inner sep=1.4pt] {} -- ++(300:2)-- ++(60:2) -- ++(180:2) {};
\draw (1)-- ++(210:2) node[circle, fill,inner sep=1.4pt] {} -- ++(180:2)-- ++(300:2)-- ++(60:2) {};

\path (baac) -- ++(45:4) node (baacd){};
\draw [style=dashed] (baacd) -- ++(190:0.4) node (baacde){};
\path (bac) -- ++(15:4) node (bacd){};
\draw [style=dashed] (bacd) -- ++(230:0.4) node (bacde){};
\draw [style=dashed] (bacde) -- (baacde){};

\path (abaac) -- ++(165:4) node (abaacd){};
\draw [style=dashed] (abaacd) -- ++(310:0.4) node (abaacde){};
\path (abac) -- ++(135:4) node (abacd){};
\draw [style=dashed] (abacd) -- ++(350:0.4) node (abacde){};
\draw [style=dashed] (abacde) -- (abaacde){};

\path (aabaac) -- ++(285:4) node (aabaacd){};
\draw [style=dashed] (aabaacd) -- ++(70:0.4) node (aabaacde){};
\path (aabac) -- ++(255:4) node (aabacd){};
\draw [style=dashed] (aabacd) -- ++(110:0.4) node (aabacde){};
\draw [style=dashed] (aabacde) -- (aabaacde){};

\draw (1) -- (a){};
\draw (1) -- (aa){};
\draw (aa) -- (a){};

\end{tikzpicture}
\end{center}
\caption[]{Block of six vertices in the Cayley graph of $\Z/2\Z*\Z/3\Z$}\label{fig:schnitte_in_cayleygraph}
\end{figure}
\end{example}

\begin{lemma}\label{lem:VDconBc}The following assertions hold.
\begin{enumerate}
\item\label{conBci} For every $C \in \Copt$ the block $B{[C]}$ is connected.
\item\label{conBcii} There is a number $\ell \in \N$ such that for all $C \in \Copt$ 
and all $S \sse B{[C]}$ we have: 
Whenever two vertices $u,v \in B{[C]}- N^\ell (S)$ can be connected by some path in $\GG- N^\ell (S)$, then they can be connected by some path in $B{[C]}- S$.
\end{enumerate}
\end{lemma}

\begin{proof}
Note that \ref{conBci}.~is a special case of \ref{conBcii}.~by choosing 
$S = \es$. Let $\ell = \max\set{d(u,v)}{D\in \Copt, \, u,v \in \Comp D\cap N^\lam (D)}$. Thus, $\ell$ is a uniform bound on the diameters of the sets $N^\lam( D) \cap \Comp{D}$ for $D\in \Copt$. It exists because there are only finitely many orbits of optimal cuts.

Now, let $u, v \in B{[C]}-N^\ell (S)$ be two vertices
which are connected by some path $\gam$ in $\GG- N^\ell (S)$. 
We are going to transform the path $\gam$ into some path $\gam'$
with all vertices in $B{[C]}- S$. If $\gam$ is entirely in $B{[C]}$ we are done. Hence, we may assume that there exist
a first vertex $v_m$ of $\gamma$ which does not lie in $B{[C]}$.
Thus, for some $D \sim C$ we have $v_m \notin N^\lam (D)$. Since
$\lam \geq 1$, we have $v_{m-1} \in N^\lam (D) \cap \Comp{D}$. 
For some $n>m$ we find a vertex 
$v_{n}$ which is the first vertex after $v_m$ lying in $N^\lam (D)$ again.
As $v_{n}$ is the first one, we have $v_{n} \in N^\lam( D) \cap \Comp{D}$, too. 
 Since $N^\lam (D) \cap \Comp{D}$ is connected, we can choose a path from $v_{m-1}$ to $v_{n}$ inside $N^\lam (D) \cap \Comp{D}$. This is a path inside $B{[C]}$ by \prref{lem:VconBc}. Note that this path does not use $v_m$ anymore. Moreover, the new segment cannot meet any point in $S$ because otherwise $v_n \in N^\ell (S)$. The path from $v_n \in B{[C]}- N^\ell (S)$ to $v$ is shorter than $\gam$. Hence, by induction, $v_n$ is connected to $v$ in $B{[C]}- S$; and we can transform $\gam$ as desired. 
 \end{proof}

Let $C \in \Copt$ and $ g \in \Aut(\GG)$ be such that $gC \sim C$. Since $gB{[C]} 
= \bigcap \set{N^\lam (gD)}{gD\sim gC}= \bigcap \set{N^\lam (gD)}{gD\sim C}$, we see that the stabilizer $G_{[C]}$ of some vertex $[C]$ of $T(\Copt)$ acts on $B[C]$. Moreover, we can prove the following lemma.

\begin{lemma}\label{lem:endl_orbits_Bc}
Let $\GG$ be a connected, locally finite, and accessible graph such that a group $G$ acts on $\Gamma $ with finitely many orbits. Let $C\in \Copt$. 
Then the stabilizer $G_{[C]}= \set{g \in G}{gC\sim C}$ of the vertex 
$[C]= \set{D}{C \sim D} \in V(T(\Copt))$
acts with finitely many orbits on the block $B{[C]}$.
\end{lemma}

\begin{proof}
By \prref{lem:VDendl_orbits_k_cuts}, $G$ acts with finitely many orbits on the set $\Copt$. For $D \sim gD \sim C$ we have $g \in G_{[C]}$, and hence 
$G_{[C]}$ acts with finitely many orbits on $[C]$. This implies that
$G_{[C]}$ acts with finitely many orbits on the union $\bigcup\set{\beta D}{D \sim C}$.

We are going to show that there is some $m\in \N$ such that for every $v \in B{[C]}$ there is a cut $D\in [C]$ with $d(v,\beta D)\leq m$, \ie  that every point in $B[C]$ is close to some $\beta D$, $D\sim C$. This implies the result since $\Gamma$ is locally finite.

Let $v \in B{[C]}$. If $v \in N^\lam(D)\cap \Comp{D}$ for some $D \sim C$, then we have $d(v,\beta D) \leq \lam$ (recall that $\lam$ is a fixed constant). Thus, it remains to consider the case $v \in D$ for all $D \sim C$.

Let $U$ be a finite subset of $B{[C]}$ such that $B{[C]} \sse G\cdot U$.
There is a constant $m \geq \lam$ such that $d(u, \beta C) \leq m$ for $u\in U$. We conclude that for the node
$v \in B{[C]}$ there is some $g\in G$ and $E = gC$ such that $d(v,\beta E ) \leq m$.
Thus, we actually may assume $v \in \beta E$ and we have to show that this implies $v \in \bigcup\limits_{D\sim C} \beta D$.

Because $C$ and $E$ are nested, we can assume (after replacing $E$ with $\Comp{E}$ if necessary) that $C\sse \Comp{E}$ or $\Comp{E}\ssnq C$. If $C\sse \Comp{E}$ (thus $E\sse \Comp{C}$), then we have $\beta E\sse \beta \Comp{C} \cup \Comp{C}$. But $v \in C$, hence
$v \in \beta \Comp{C} = \beta C$. 
On the other hand, if $\Comp{E} \subsetneqq C$, then there is an optimal cut $D\sim C$ such that $\Comp{E} \sse \Comp{D} \subsetneqq C$. It follows that $v \in D \cap \beta \Comp{E} \sse D \cap ( \beta \Comp{D} \cup \Comp{D}) \sse \beta D$.
\end{proof}

The next result states a key property of blocks which finally implies that the stabilizers of the blocks are finite.
\begin{proposition}\label{prop:enden_in_Bc}
For $C\in \Copt$ the block $B{[C]}$ has at most one end. 
\end{proposition}

\begin{proof}
Assume by contradiction that $B{[C]}$ has more than one end. By \prref{lem:VDconBc} $B{[C]}$ is connected, hence there is a bi-infinite 
simple path $\alp$ and a finite subset $S\sse B{[C]}$ such that 
two different connected components of $B{[C]} -S$ contain 
infinitely many vertices of $\alp$. However, for all $D \sim C$ we have $\alp\sse B{[C]} \sse N^\lam (D)$. Since $N^\lam (D) \cap \Comp{D}$ is finite, this implies that for
$D \sim C$ almost all nodes of $\alp$ are in $D$,
and hence $\abs{\alp \cap \Comp{D}} < \infty$. 

By \prref{lem:VDconBc},
there are two different connected components of $\GG- N^\ell (S)$ each
containing infinitely many vertices of $\alp$. Thus, the set $\cC(\alp)$ is not empty and there is an optimal cut $E \in \Copt(\alp)$.
This means $\abs{\alp \cap E} = \infty = \abs{\alp \cap \Comp{E}}$.
The cuts $C$ and $E$ are nested. We cannot have 
$E \sse \Comp{C}$ or $\Comp{E} \sse \Comp{C}$ because $\abs{\alpha \cap \Comp{C}}<\infty$.
Hence, by symmetry $E \subsetneqq C$. 
By \prref{lem:tree_set}, there is some $D\in[C]$ such that $E \sse \Comp{D}\subsetneqq C$. But we just have seen that almost all nodes of $\alp$ belong to $D$. Thus, $\abs{\alp \cap E} < \infty$. This is a contradiction. 
\end{proof}

Now we have all the tools to state and prove the main theorem of this section.
\begin{theorem}\label{thm:new_alpha}
Let $\Gamma$ be a connected, locally finite graph of finite treewidth. 
Let a group $G$ act on $\Gamma$ such that $G\bs \Gamma$ is finite and each 
node stabilizer $G_v$ is finite. 
Then $G$ acts on the tree $T(\Copt)$ such that all vertex and edge 
stabilizers are finite and $G\bs T(\Copt)$ is finite.
\end{theorem}

\begin{proof}
The blocks $B{[C]}$ have finite treewidth by \prref{lem:subgraph}.
By \prref{lem:endl_orbits_Bc}, $G_{[C]}$ acts with finitely 
many orbits on $B[C]$. Hence, we can apply \prref{prop:two_ends} what implies that the blocks are finite or have more than one end. The latter case is excluded by \prref{prop:enden_in_Bc}, which states that the blocks have at most one end. Hence, the blocks are finite. 

Now, let $g \in G_{[C]}$, then we have $g(B{[C]}) = B{[C]}$. Since $B{[C]}$ is finite and all vertex stabilizers $G_v$ of vertices of $\GG$ are finite, also $G_{[C]}$ is finite.
Therefore, the action has finite node and edge stabilizers.

By \prref{lem:VDendl_orbits_k_cuts}, $G$ acts with finitely many orbits on $T(\Copt)$.
\end{proof}

\begin{corollary}\label{cor:fin_tree_decomp}
 Let $\Gamma$ be a connected, locally finite graph of finite treewidth. 
Let a group $G$ act on $\Gamma$ such that $G\bs \Gamma$ is finite and the
node stabilizers $G_v$ are finite. 
 Then the tree $T(\Copt)$ with bags $B{[C]}$ forms a tree decomposition of $\GG$ with finite bag-size and $G$ acts on $T(\Copt)$.
\end{corollary}

\begin{proof}
 We have to show that the tree $T(\Copt)$, together with bags $B{[C]}$ for $[C]\in V(T(\Copt))$, forms a tree decomposition. If $uv\in E(\Gamma)$ is in $\delta C$ for some optimal cut $C$, then $u,v\in B{[C]}$. If $uv\in E(\Gamma)$ is not in $\delta C$ for any optimal cut $C$, then $u,v\in\bigcap\set{D\in \Copt}{u\in D}\sse \bigcap_{D \sim C} D\sse B[C]$ for some optimal cut $C$. Hence, (T1) and (T2) hold. (T3) follows from the fact that optimal cuts are nested.
\end{proof}

\begin{corollary}\label{cor:nixstall}
Let a group $G$ act on a connected, locally finite graph $\Gamma$ of finite treewidth such that $G\bs \Gamma$ is finite and each 
node stabilizer $G_v$ is finite. 
Then $G$ is the fundamental group of a finite graph of finite groups and hence is virtually free.
\end{corollary}

\begin{proof}
By \prref{thm:new_alpha}, $G$ acts on a tree $T$ with finite vertex stabilizers such that $G\bs T$ is finite. 
\prref{thm:bst} and \prref{thm:endl_gog_virt_free} yield the result. 
\end{proof}

\begin{remark}
 The characterization of \prref{cor:nixstall} was shown by Kuske and Lohrey in their proof of the result that, if the Cayley graph of some group has a decidable monadic second order theory, then the group is virtually free. They used a theorem from \cite{Courcelle94,Seese91} stating that a graph with decidable monadic second order theory has finite treewidth.
 \end{remark}

\begin{remark}
 In \prref{thm:new_alpha} and \prref{cor:nixstall} the requirement ``$\abs{G\bs \Gamma}<\infty$'' can be replaced by the condition ``$G$ finitely generated''. 
 
In fact, a subgraph of $\Gamma$ on which $G$ acts with finitely many orbits can be constructed via the following procedure. Let $\Sigma$ be a finite generating set of $G$ and let $v_0\in V(\Gamma)$ be some arbitrary vertex. For all $a\in \Sigma$ we fix paths $\gamma_a$ from $v_0$ to $av_0$. Let $\Delta$ be the subgraph of $\Gamma$ induced by the vertex set $G\cdot \bigcup_{a\in \Sigma} \gamma_a$. This graph is connected, locally finite and it has finite treewidth by \prref{lem:subgraph}.
\end{remark}

\section{Summary}
In these notes we have seen many of the characterizations of virtually free groups and proven their equivalence.
\prref{fig:implikationen} depicts the implications we have shown. 
\begin{figure}[ht]
\begin{center}
\begin{footnotesize}
\begin{tikzpicture}[yscale=0.9, xscale=1.1]
\node [draw,rounded corners,text width=38mm ]	(a) at (0,6 ) 		{Fundamental group of a finite graph of finite groups};
\node [draw,rounded corners,text width=38mm ]	(aa) at (0,8 ) 		{Acting with finitely many orbits and finite vertex stabilizers on some tree};
\node [draw,rounded corners]			(b) at (1,2.4 ) 	{Virtually free};
\node [draw,rounded corners,text width=25mm]	(c) at (-3.7,4) 	{Universal group of a finite pregroup};
\node [draw,rounded corners,text width=31mm]	(db) at (-1.8,2.2 ) 	{Defined by a finite geo\-desic rewriting system};
\node [draw,rounded corners]			(d) at (-3.7,-1 ) 	{Context-free};
\node [draw,rounded corners]			(da) at (-0.9,0.4)	{Deterministic context-free};
\node [draw,rounded corners,text width=23mm ]	(ea) at (2.5,0.3 ) 		{Cayley graph quasi-isometric to a tree};
\node [draw,rounded corners,text width=26mm ]	(e) at (1,-2.5 ) 	{Cayley graph with finite treewidth};
\node [draw,rounded corners]			(f) at (3.5,3.7 ) 	{Chordal Cayley graph};
\node [draw,rounded corners,text width=31mm]			(g) at (0,4 ) 	{Subgroup of semi\-direct product free by finite};

\draw [semithick,->] 	(e) .. controls(6.3,-1.2) and (6.6,7) .. 			(aa);
				\node at (3.2,6.6) {Thm.\ \ref{thm:new_alpha}};
\draw [semithick,->] 	(a) to 	node[above left] {Thm.\ \ref{thm:endl_pregroup}} (c) ;
\draw [semithick,->] 	(c) to 	node[left] {Cor.\ \ref{cor:precf}} 		(d);
\draw [semithick,->] 	(c) to 	node[right] {Prop.\ \ref{prop:sp_geodesic}} 	(db);
\draw [semithick,->] 	(db) to	node[left] {Prop.\ \ref{prop:geocf}} 		(da);
\draw [semithick,->] 	(da) to	node[below right=-3pt] {\ Thm.\ \ref{thm:cfpda}} 	(d);
\draw [semithick,->] 	(b) to 	node[right] {Prop.\ \ref{prop:vfdetcf}} 	(da);
\draw [semithick,->] 	(d) to 	node[below left] {Thm.\ \ref{thm:cfftw}} 	(e);
\draw [semithick,->]	(b) to 	node[right] {Cor.\ \ref{cor:quasiso}} 		(ea);
\draw [semithick,->] 	(ea) to	node[left] {Cor.\ \ref{cor:treewidthquasitotree}} (e);
\draw [semithick,->] 	(a) to 	node[left] {Prop.\ \ref{prop:sdvf}} (g);
\draw [semithick,->] 	(g) to 	node[right] {Cor.\ \ref{cor:sdvf}} (b);
sdvf
\draw [semithick,<->] 	(aa) to	node[left] {Thm.\ \ref{thm:bst}} 		(a);
\draw [semithick,->] 	(a) to 	node[right] {Thm.\ \ref{thm:chordal_cayley}} 	(f);
\draw [semithick,->] 	(f) .. controls (4.3,1.8)and (5.5,0.2) .. 		(e);
				\node at (3.5,2.2) {Rem.\ \ref{rem:chordal_cayley}};
\end{tikzpicture}
\end{footnotesize}
\caption[]{Roadmap of implications.}\label{fig:implikationen}
\end{center}
\end{figure}
The most technical and rather  difficult proof is the one of \prref{thm:new_alpha}. 
In terms of Bass-Serre theory the arrow from bottom to top says that a group with finite treewidth can be written as a fundamental group of a finite graph of finite groups. The other characterizations are more straightforward from the latter characterization. Thus, a characterization by finite treewidth is somehow the ``weakest'' requirement whereas being a 
fundamental group of a finite graph of finite groups is the ``strongest'' description of a context-free group.

An interesting open problem is the complexity of the isomorphism problem of context-free groups. Up to now the best known upper bound is primitive recursive \cite{sen96dimacs}. The question is whether the methods developed here can be used to give an improved complexity bound for the isomorphism problem of context-free groups.
Other questions arise when restricting to subclasses of virtually free groups. 
For example, one can consider, instead of geodesic rewriting systems, only confluent and length reducing rewriting systems. It is known that when additionally restricting the left sides to have length at most $2$ then these systems generate exactly the 
 plain groups \cite{AvenhausMO86}, but the general case is still open. 
 It is known as Gilman's conjecture \cite{Gilman1984}. 
  \emph{Plain groups} are finitely generated free products of finite and free groups. They are called \emph{basic groups} in  \cite{CoornaertFS12}.
 A more general conjecture than the one by Gilman is related to an question of Shapiro (c.\,f.~\cite{Schupp12}): Let $\Gam$ be a Cayley graph of some finitely generated group $G$ where geodesics are unique. Is it true that $G$ is a basic (resp.{} plain) group? The conjecture is again ``Yes''. As basic groups are context-free,
 it is reasonable to settle the conjecture first under the 
 assumption that the group is context-free. Thus, the Cayley graph has the 
additional property of finite treewidth. 

\addcontentsline{toc}{section}{References}
\bibliographystyle{abbrv}

\newcommand{\Ju}{Ju}\newcommand{\Ph}{Ph}\newcommand{\Th}{Th}\newcommand{\Ch}{Ch}\newcommand{\Yu}{Yu}\newcommand{\Zh}{Zh}\newcommand{\St}{St}\newcommand{\curlybraces}[1]{\{#1\}}

\end{document}